\documentclass[11pt]{amsart}

\usepackage{amsmath}
\usepackage{amsfonts}
\usepackage{amssymb}
\usepackage{amsthm}
\usepackage{graphicx}
\usepackage{hyperref}
\usepackage{pb-diagram}
\usepackage{epstopdf}
\usepackage{amscd}
\usepackage{pdfpages}
\usepackage{bbm}
\usepackage{multirow}
\usepackage[mathscr]{eucal}
\usepackage{epsfig,epsf,epic}
\usepackage{pdfsync}
\usepackage{enumerate}
\usepackage{etex}
\usepackage[all]{xy}
\usepackage{fancyref}
\usepackage{comment}

\newtheorem{theorem}{Theorem}[section]
\newtheorem{prop}[theorem]{Proposition}

\newtheorem{corollary}[theorem]{Corollary}
\newtheorem{definition}[theorem]{Definition}
\newtheorem{example}[theorem]{Example}
\newtheorem{lemma}[theorem]{Lemma}
\newtheorem{remark}[theorem]{Remark}

\newtheorem{notation}[theorem]{Notations}
\numberwithin{equation}{section}

\newcommand{\real}{\mathbb{R}}
\newcommand{\comp}{\mathbb{C}}

\newcommand{\inte}{\mathbb{Z}}

\newcommand{\pdb}{\bar{\partial}}
\newcommand{\dd}[1]{\frac{\partial}{\partial #1}}

\newcommand{\half}{\frac{1}{2}}
\newcommand{\ddd}[2]{\frac{\partial#1}{\partial{#2}}}

\newcommand{\pdpd}[3]{\frac{\partial^2 #1}{\partial #2\partial #3}}

\newcommand{\bmc}{w}

\newcommand{\facs}{\varphi}

\newcommand{\incoming}{\Pi}

\newcommand{\tree}[1]{\mathbb{T}_{#1}}

\newcommand{\tr}{T}

\newcommand{\wptr}{\Gamma}
\newcommand{\wptree}[2]{\mathtt{W}\mathtt{P}\mathtt{T}_{#1,#2}}
\newcommand{\wrtr}{\mathcal{T}}
\newcommand{\wrtree}[2]{\mathtt{W}\mathtt{R}\mathtt{T}_{#1,#2}}

\newcommand{\hp}{\hslash}  

\newcommand{\bvd}{\Delta}

\begin{document}

\title[Tropical counting from Maurer-Cartan equations]{Tropical counting from asymptotic analysis\\ on Maurer-Cartan equations}
\author[Chan]{Kwokwai Chan}
\address{Department of Mathematics\\ The Chinese University of Hong Kong\\ Shatin\\ Hong Kong}
\email{kwchan@math.cuhk.edu.hk}

\author[Ma]{Ziming Nikolas Ma}
\address{The Institute of Mathematical Sciences and Department of Mathematics\\ The Chinese University of Hong Kong\\ Shatin\\ Hong Kong}
\email{zmma@ims.cuhk.edu.hk}

\begin{abstract}
Let $X = X_\Sigma$ be a toric surface and $(\check{X}, W)$ be its Landau-Ginzburg (LG) mirror where $W$ is the Hori-Vafa potential \cite{Hori-Vafa00}. We apply asymptotic analysis to study the extended deformation theory of the LG model $(\check{X}, W)$, and prove that semi-classical limits of Fourier modes of a specific class of Maurer-Cartan solutions naturally give rise to tropical disks in $X$ with Maslov index 0 or 2, the latter of which produces a universal unfolding of $W$. For $X = \mathbb{P}^2$, our construction reproduces Gross' perturbed potential $W_n$ \cite{Gross10} which was proven to be the universal unfolding of $W$ written in canonical coordinates. We also explain how the extended deformation theory can be used to reinterpret the jumping phenomenon of $W_n$ across walls of the scattering diagram formed by Maslov index 0 tropical disks originally observed by Gross \cite{Gross10} (in the case of $X = \mathbb{P}^2$).
\end{abstract}

\maketitle

\section{Introduction}

\subsection{Background}

The study of mirror symmetry for toric varieties goes back to Batyrev \cite{Batyrev93}, Givental \cite{Givental95, Givental96, Givental98}, Lian-Liu-Yau \cite{LLY-III}, Kontsevich \cite{Kontsevich-ENS98} and Hori-Vafa \cite{Hori-Vafa00}. Unlike the Calabi-Yau case, the mirror of a compact toric manifold $X$ is given by a {\em Landau-Ginzburg (abbrev. LG) model} $(\check{X}, W)$ consisting of a noncompact K\"ahler manifold $\check{X}$ and a holomorphic function $W:\check{X} \rightarrow \comp$ called the {\em potential} \cite{Vafa91, Witten93}. As a prototypical example, the LG mirror of $X = \mathbb{P}^2$ is given by $\check{X} = \{(z^0,z^1,z^2) \in \comp^3 \mid z^0 z^1 z^2 = 1 \} $ together with the restriction of $W = z^0 + z^1 +z^2$ to $\check{X} \subset \comp^3$.

At the genus 0 level, mirror symmetry can be understood as an isomorphism between Frobenius manifolds.
For a large class of examples, the construction of the {\em B-model Frobenius manifold} from the LG model $(\check{X}, W)$ was carried out by Douai-Sabbah \cite{douai2003gauss, douai2004gauss} (see also the book \cite{sabbah2007isomonodromic}), generalizing the classic work of K. Saito \cite{Saito83}. Mirror symmetry then says that the B-model Frobenius manifold of $(\check{X}, W)$ is isomorphic (via a possibly nontrivial mirror map) to the {\em A-model Frobenius manifold} constructed from the genus 0 Gromov-Witten (abbrev. GW) theory or {\em big} quantum cohomology of $X$. In the case of projective spaces this was proved by Barannikov \cite{barannikov2000semi}.

The geometry of this mirror symmetry can be understood using the Strominger-Yau-Zaslow (abbrev. SYZ) conjecture \cite{syz96}. Namely, a natural Lagrangian torus fibration is given by the moment map $p: X \to \mathbf{P}$, and the mirror manifold $\check{X}$ can be constructed geometrically as the moduli space of A-branes $(L, \nabla)$ consisting of a Lagrangian torus fiber $L$ of $\rho$ and a flat $U(1)$-connection $\nabla$ over it, or simply, as the total space of the fiberwise dual of $p$ restricted to the interior $\text{Int}(\mathbf{P}) \subset \mathbf{P}$ \cite{Auroux07}. The SYZ conjecture also suggests that mirror symmetry is a geometric Fourier transform; in this regard, the construction of $\check{X}$ from $X$ may be viewed ``$0$-th Fourier mode'' of the mirror geometry.

The ``higher Fourier modes'' or ``quantum corrections'' come from the singular or degenerated fibers of $p$ over the boundary $\partial\mathbf{P}$ and are captured by holomorphic disks in $X$ with boundary on a Lagrangian torus fiber of $p$ -- this gives rise to the mirror LG potential $W$.
Cho-Oh \cite{Cho-Oh06} were the first to prove, in the toric Fano case, that the so-called {\em Hori-Vafa potential} $W$ \cite{Hori-Vafa00} can be expressed in terms of counts of Maslov index 2 holomorphic disks.
This was later generalized by Fukaya-Oh-Ohta-Ono \cite{FOOO-toricI} to all compact toric manifolds.
They defined the {\em Lagrangian Floer potential} $W^{\text{LF}}$ which is determined by the {\em obstruction cochain} $\mathfrak{m}_0$ in the Floer complex of a Lagrangian torus fiber of the moment map $p$ \cite{FOOO_I, FOOO_II}.

In more explicit terms, coefficients of $W^{\text{LF}}$ are virtual counts of Maslov index 2 {\em stable} disks, or more precisely, {\em genus 0 open Gromov-Witten invariants}, and $W^{\text{LF}}$ is a perturbation of the Hori-Vafa potential $W$ of the form
$$W^{\text{LF}} = W + \text{correction terms}$$
because coefficients of $W$ only encode counts of {\em embedded} disks (which is why $W^{\text{LF}} = W$ only when $X$ is toric Fano).
In general it is very hard to compute $W^{\text{LF}}$, but explicit formulas are known in a few low-dimensional examples \cite{Auroux09, FOOO12, Chan-Lau10} and when $X$ is semi-Fano \cite{CLLT11, CLLT12, Gonzalez-Iritani17}.

Using $W^{\text{LF}}$, one obtains an isomorphism of Frobenius algebras
\begin{equation}\label{eqn:isom_QH_Jac_small}
QH^*(X) \cong Jac(W^{\text{LF}})
\end{equation}
between the {\em small} quantum cohomology ring of $X$ and the Jacobian ring of $W$, {\em without} going through a mirror map (or one can say that the mirror map is trivialized).
To upgrade this to an isomorphism between Frobenius manifolds, Fukaya-Oh-Ohta-Ono \cite{FOOO-toricII} introduced the {\em bulk-deformed potential} $W^{\text{LF}}_b$ as a perturbation of $W^{\text{LF}}$ by the ambient cycles in $X$. In \cite{FOOO-toricIII} they proved that \eqref{eqn:isom_QH_Jac_small} can be enhanced to an isomorphism between the A-model Frobenius manifold of $X$ and the B-model Frobenius manifold constructed from $W^{\text{LF}}_b$.

Not long afterward, Gross \cite{Gross10, gross2011tropical} constructed a very explicit perturbation of the Hori-Vafa potential $W = z^0 + z^1 +z^2$ mirror to $X = \mathbb{P}^2$ using counts of Maslov index 2 {\em tropical disks} in $\real^2$ (or the tropical projective plane $\mathbb{TP}^2$). He computed oscillatory integrals of his perturbed potential, producing beautiful tropical formulas for descendent GW invariants and proving that it is the universal unfolding of $W$ written in canonical coordinates, thereby giving a very transparent proof of mirror symmetry for $\mathbb{P}^2$ via tropical geometry.


Gross' work is closely connected with the influential Gross-Siebert program \cite{Gross-Siebert03, Gross-Siebert-logI, Gross-Siebert-logII, gross2011real}, where a key role is played by a combinatorial gadget called {\em scattering diagram} which was first introduced by Kontsevich-Soibelman in \cite{kontsevich-soibelman04}.

On the other hand, the precise correspondence between counting of tropical and holomorphic curves has been studied in various cases, first by Mikhalkin \cite{Mikhalkin05} in dimension $2$, and later by Nishinou-Siebert \cite{Nishinou-Siebert06} in higher dimensions. The correspondence between tropical and holomorphic disks in toric varieties was first investigated by Nishinou \cite{Nishinou12}, and more recently, clarified and refined by Hong-Lin-Zhao \cite{hong2018bulk} (in the case of toric surfaces). These works indicate that tropical geometry is indeed sufficient in describing GW theory.

\subsection{Asymptotic behavior of Maurer-Cartan solutions}

The main goal of this paper is to explain how extended deformation theory of the LG model $(\check{X}, W)$ can lead us naturally to Gross' perturbed potential constructed in \cite{Gross10}. Our main tool is asymptotic analysis on Maurer-Cartan equations and we will build on the approach developed in \cite{kwchan-leung-ma}. In a broader sense, our results are about relations between tropical disk counting on (the tropical counterpart of) a toric surface $X$ and the extended deformation theory of its LG mirror $(\check{X}, W)$, where $W$ is taken to be the Hori-Vafa potential \cite{Hori-Vafa00}.

Recall that in \cite{kwchan-leung-ma}, we considered the differential-geometric deformation theory of $\check{X}$ governed by the Kodaira-Spencer dgLa
$KS^*_{\check{X}} : = \Omega^{0,*}(\check{X},T^{1,0})$
and the associated Maurer-Cartan (abbrev. MC) equation
\begin{equation}\label{eqn:original_MC_equation}
\bar{\partial} \varphi + \half [\varphi, \varphi] = 0.
\end{equation}
An $\hp \in \real_+$ parameter was introduced there to twist the complex structure of $\check{X}$ which geometrically corresponds to shrinking of the torus fibers in $X$.
Following a proposal put forward by Kontsevich-Soibelman \cite{Kontsevich-Soibelman01} and Fukaya \cite{fukaya05}, we studied Fourier expansions of a specific class of solutions of the MC equation \eqref{eqn:original_MC_equation} along fibers of $\check{p}: \check{X} \rightarrow \text{Int}(\mathbf{P})$.
The main results in \cite{kwchan-leung-ma} showed that leading order terms (or semiclassical limits) of the Fourier modes of such a solution naturally give rise to a consistent scattering diagram $\mathscr{D}$ as $\hp \rightarrow 0$, and conversely such solutions can be constructed as sums over trees of terms with support concentrated along the walls in $\mathscr{D}$.

In this paper, we consider the extended Kodaira-Spencer complex given by {\em polyvector fields}:
$$PV^{*,*}(\check{X}) := \Omega^{0,*}\left(\check{X},\wedge^* T^{1,0}\right).$$
This is equipped with the Dolbeault differential $\bar{\partial}$, a naturally extended Lie-bracket $[\cdot,\cdot]$ and, as $\check{X}$ is Calabi-Yau, a BV operator $\Delta$, constituting {\em a differential graded Batalin-Vilkovisky (abbrev. dgBV) algebra}.
This structure is the key ingredient in the construction of the B-model Frobenius manifold in the Calabi-Yau setting \cite{Barannikov-Kontsevich98, barannikov2000semi, li2013primitive}.

For a LG model $(\check{X},W)$, the Dolbeault differential in the dgBV algebra $PV^{*,*}$ should be replaced by the {\em twisted Dolbeault differential}
$\bar{\partial}_W := \bar{\partial} + [W, \cdot],$
and it is natural to consider the associated {\em extended} Maurer-Cartan equation:
\begin{equation}\label{eqn:extended_MC_equation_intro}
\bar{\partial}_W \varphi + \half [\varphi, \varphi] = 0
\end{equation}
for $\varphi \in PV^{*,*}(\check{X})$ (see Section \ref{sec:BV_algebra}). We adapt this approach with a view towards a construction of the higher genus B-model \cite{li2011bcov, costello2012quantum}.

Except in Sections \ref{sec:asymptotic_support} and \ref{sec:homotopy_operator}, our attention will be restricted to the 2-dimensional case, so $X$ will just be a toric surface (implicitly equipped with the toric anticanonical divisor $D = D_\infty$). To analyze the equation \eqref{eqn:extended_MC_equation_intro} associated to the mirror LG model $(\check{X} \cong (\comp^*)^2, W)$, where $W$ is the Hori-Vafa potential, we apply the machinery developed in \cite{kwchan-leung-ma}. More precisely, as we will only be concerned with the leading order behavior of the MC solutions as $\hp \rightarrow 0$, the complex $PV^{*,*}$ will be replaced by the quotient $(\mathcal{G}/\mathcal{I})^{*,*}$ of a subalgebra $\mathcal{G}^{*,*} \leq PV^{*,*}$ consisting of terms with growth control as $\hp \rightarrow 0$ by the ideal $\mathcal{I}^{*,*}$ generated by error terms in $\hp$ (see Definition \ref{def:finite_Lie_alg}). We shall also employ the important notion of {\em asymptotic support} first introduced in \cite{kwchan-leung-ma} (see Definition \ref{def:asypmtotic_support_pre}). 

In the 2-dimensional case, it is natural to consider deformations using $n$ points $P_1,\dots,P_n \in \real^2$ in generic position.\footnote{These are the only nontrivial bulk deformations.} In view of this, we choose an input $\incoming \in (\mathcal{G}/\mathcal{I})^{2,2}$ of the form
\begin{equation}\label{eqn:introduction_input}
\incoming = \sum_{i} u_i \delta_{P_i}  (\partial_{1} \wedge \partial_2),
\end{equation}
where $u_i$ is a formal variable in the ring
$R = R_n := \comp\left[u_1,\dots,u_n \right] / \left(u_{i}^2 \mid 1\leq i \leq n \right)$
(equipped with the maximal ideal $\mathbf{m} = \mathbf{m}_n := (u_1,\dots,u_n)$) which corresponds to the point $P_i$, $\partial_1 \wedge \partial_2$ is the canonical holomorphic bi-vector field on $(\comp^*)^2$ and $\delta_{P_i}$ is a Dolbeault $(0,2)$-form with asymptotic support at $P_i$.
The idea to work with the ring $R_n$ is entirely motivated by the tropical geometry setup in Gross' work \cite{Gross10}. Also, the form $\delta_{P_i}$ should be viewed as a smoothing of a `delta-form' supported at $P_i$ (cf. \cite{kwchan-leung-ma}).

As in \cite{kwchan-leung-ma}, a solution $\Phi$ to the extended MC equation \ref{eqn:extended_MC_equation_intro} of $(\mathcal{G}/\mathcal{I})^{*,*}$ can be constructed using Kuranishi's method \cite{Kuranishi65}, namely, by summing over directed ribbon weighted $d$-pointed $k$-trees (see Definition \ref{def:weighted_tree}) with input $\incoming$.
The MC solution $\Phi$ can then be decomposed as
\begin{equation}\label{eqn:decomposition_MC_sol}
\Phi = \incoming + \varXi^{0,0} + \varXi^{1,1},
\end{equation}
where $\varXi^{i,i} \in (\mathcal{G}/\mathcal{I})^{i,i}$.
A major discovery of this paper is that, as $\hp \rightarrow 0$, the correction terms $\varXi^{1,1}$ and $\varXi^{0,0}$ give rise to tropical disks of Maslov index 0 and 2 respectively:
\begin{theorem}[=Theorem \ref{thm:main_theorem}]\label{thm:main_theorem_intro_1}
For a toric surface $X$ equipped with the Hori-Vafa potential $W$, there is a solution $\Phi$ to the Maurer-Cartan equation decomposed as in \eqref{eqn:decomposition_MC_sol} such that each of the terms $\varXi^{0,0}, \varXi^{1,1}$ can be expressed as a sum over tropical disks $\wptr$ transversal to the toric divsor $D_\infty$ whose moduli space $\overline{\mathfrak{M}}^{\wptr}$ is non-empty of codimension $1-\frac{MI(\wptr)}{2}$ in $B_0 = \real^2$ (where $MI$ denotes the Maslov index):
$$
\varXi^{0,0}  = \sum_{MI(\wptr)=2} \alpha_{\wptr} \text{Mono}(\wptr), \quad
\varXi^{1,1}  = \sum_{MI(\wptr)=0} \alpha_{\wptr} \text{Log}(\Theta_\wptr);
$$
here $\text{Mono}(\wptr)$ is a holomorphic function and $\text{Log}(\Theta_\wptr)$ is a holomorphic vector field defined explicitly for a tropical disk $\wptr$, and $\alpha_\wptr$ is a Dolbeault $(0,1-\frac{MI(\wptr)}{2})$-form with asymptotic support along the $(1+\frac{MI(\wptr)}{2})$-dimensional tropical polyhedral subset $Q_\wptr \subset B_0$ traced out by the stop $Q$ of the tropical disks in the moduli space $\overline{\mathfrak{M}}^{\wptr}$ (see Definition \ref{def:tropical_curve}).

Furthermore, the following properties hold:
\begin{equation*}
\begin{array}{ll}
\displaystyle\lim_{\hp \rightarrow 0} \alpha_{\wptr}|_{x} = 1 & \text{for any $x$ in the interior $\text{Int}(Q_\wptr)$ when $MI(\wptr) = 2$},\\
\displaystyle\lim_{\hp \rightarrow 0} \int_{\varrho} \alpha_{\wptr} = -1 & \text{for any $\varrho \pitchfork \text{Int}_{\text{re}}(Q_\wptr)$ positively when $MI(\wptr) = 0$},\\
\end{array}
\end{equation*}
where $\varrho$ is any affine line intersecting $Q_\wptr$ positively and transversally in its relative interior $\text{Int}_{\text{re}}(Q_\wptr)$.
\end{theorem}

Following Gross \cite{Gross10}, we define the $n$-pointed perturbed LG potential $W_n(Q)$ in terms of counts of Maslov index $2$ tropical disks with interior marked points possibly passing through $P_1,\dots,P_n$ and with stop at a fixed point $Q \in \real^2$; see Section \ref{sec:counting} for the precise definitions. Then Theorem \ref{thm:main_theorem_intro_1} gives
a bijective correspondence between leading order terms of $\varXi^{0,0}$ as $\hp \to 0$, tropical disks $\wptr$ with $MI(\wptr) = 2$ and $\overline{\mathfrak{M}}^{\wptr} \neq \emptyset$, and hence terms in the perturbed potential $W_n(Q)$:
\begin{equation*}
\left\{
\begin{array}{c}
\text{terms in the}\\
\text{expression of $\varXi^{0,0}$}
\end{array}
\right\}
\overset{\text{Theorem \ref{thm:main_theorem_intro_1}}}{\longleftrightarrow}
\left\{
\begin{array}{c}
\text{Tropical disks $\wptr$}\\
\text{with $MI(\wptr) = 2$}
\end{array}
\right\}
\overset{\text{\cite{Gross10}}}{\longleftrightarrow}
\left\{
\begin{array}{c}
\text{terms in the}\\
\text{perturbation $W_n(Q)$}
\end{array}
\right\}.
\end{equation*}
In the case of $X = \mathbb{P}^2$, the tropical disks above are precisely those considered by Gross \cite{Gross10}. Indeed we have
$$\lim_{\hp\rightarrow 0}\varXi^{0,0}(Q) = \sum_{MI(\wptr)=2} \text{Mono}(\wptr),$$
which coincides with Gross' definition of the $n$-pointed potential $W_n(Q)$. This explains how solutions to the extended MC equation \eqref{eqn:extended_MC_equation_intro} lead naturally to the perturbed potential $W_n$.

The $n$-pointed potential $W_n(Q)$ depends on $Q$, and Gross \cite{Gross10} proved that the dependence is dictated by wall-crossing formulas across walls of a scattering diagram $\mathscr{D}$ constructed from the Maslov index $0$ tropical disks with interior marked points possibly passing through the $n$ marked points $P_1,\dots,P_n$. Theorem \ref{thm:main_theorem_intro_1} gives a bijective correspondence between leading order terms of $\varXi^{1,1}$ as $\hp \to 0$, tropical disks $\wptr$ with $MI(\wptr)= 0$ and $\overline{\mathfrak{M}}^{\wptr} \neq \emptyset$, and hence walls in the scattering diagram $\mathscr{D}$:
\begin{equation*}
\left\{
\begin{array}{c}
\text{terms in the}\\
\text{expression of $\varXi^{1,1}$}
\end{array}
\right\}
\overset{\text{Theorem \ref{thm:main_theorem_intro_1}}}{\longleftrightarrow}
\left\{
\begin{array}{c}
\text{Tropical disks $\wptr$}\\
\text{with $MI(\wptr)=0$}
\end{array}
\right\}
\overset{\text{\cite{Gross10}}}{\longleftrightarrow}
\left\{
\begin{array}{c}
\text{walls $\mathbf{w}_\wptr = (m_\wptr, Q_\wptr, \Theta_\wptr)$}\\
\text{in the diagram $\mathscr{D}$}
\end{array}
\right\}.
\end{equation*}

\begin{remark}
Underlying the above bijective correspondences is an interplay between the differential-geometric properties of the dgBV algebra $(\mathcal{G}/\mathcal{I})^{*,*}$ and the combinatorial properties of tropical disks, which has also played an important role in \cite{kwchan-leung-ma}. As will be seen in Section \ref{sec:tropical_dgla}, this leads to an extended version of the tropical Lie algebra which we call the {\em tropical dgLa}.
\end{remark}

\subsection{Wall-crossing from Maurer-Cartan solutions}\label{sec:wall_crossing_formula}

Besides giving enumerative meanings to correction terms in the MC solution $\Phi$, Theorem \ref{thm:main_theorem_intro_1}, when combined with the main results in \cite{kwchan-leung-ma}, has an interesting corollary which can be viewed as an alternative proof (via gauge equivalences) of the wall-crossing formulas (\cite[Theorem 4.12]{Gross10}) for Gross' perturbed potential $W_n$. Let us explain the argument in this subsection.



By definition, the scattering diagram $\mathscr{D} = \mathscr{D}(\mathcal{P},\Sigma,P_1,\dots,P_n)$ consists of walls $\mathbf{w}_\wptr = (m_\wptr, Q_\wptr, \Theta_\wptr)$, the support of each being the tropical polyhedral subset $Q_\wptr$ traced out by the stop of a tropical disk $\wptr$ (see Section \ref{sec:scattering_diagram}). The intersection of these $Q_\wptr$'s is the set $\text{Sing}(\mathscr{D}) \setminus \{P_1,\dots,P_n\}$, where $\text{Sing}(\mathscr{D})$ denotes the singular set of $\mathscr{D}$, and a point $\mathfrak{j} \in \text{Sing}(\mathscr{D})\setminus \{P_1,\dots,P_n\}$ is called a {\em joint} in the Gross-Siebert program \cite{gross2011real}.

Restricting to a contractible open neighborhood $U = U_{\mathfrak{j}} \subset \real^2 \setminus \{P_1,\dots,P_n\}$ containing a single joint $\mathfrak{j}$, we have
$\Phi|_{U} = \varXi^{0,0} + \varXi^{1,1}$
because $\delta_{P_i}|_{U} = 0$ and hence $\incoming|_{U} = 0$ in $(\mathcal{G}/\mathcal{I})^{*,*}$. By degree reasons, we see that $\varXi^{1,1}|_{U}$ is itself a solution to the non-extended Maurer-Cartan equation \eqref{eqn:original_MC_equation}, namely,
$
\bar{\partial} \varXi^{1,1} + \half [\varXi^{1,1}, \varXi^{1,1}] = 0.
$ Now \cite[Theorems 1.5 and 1.6]{kwchan-leung-ma}, the proofs of which were by asymptotic analysis and completely different from that of \cite{Gross10}, imply the following statement which originally appeared in Gross \cite{Gross10}:
\begin{corollary}[Proposition 4.7 in \cite{Gross10}]\label{prop:consistency}
For any point $\mathfrak{j} \in \text{Sing}(\mathscr{D}) \setminus \{P_1,\dots,P_n\}$, we have
$
\Theta_{\gamma_\mathfrak{j},\mathscr{D}}  = \text{Id}
$
for any loop $\gamma_\mathfrak{j}$ around $\mathfrak{j}$ in a sufficiently small contractible neighborhood $U_\mathfrak{j}$ of $\mathfrak{j}$.
\end{corollary}

Next we would like to work locally near a wall $Q_\wptr$ of the scattering diagram $\mathscr{D}$. So we consider a contractible open subset $U \subset M_\real \setminus \left( \{P_1,\dots, P_n \} \cup \text{Sing}(\mathscr{D}) \right)$, which is separated into two chambers $U_+$ and $U_-$ by the wall $Q_\wptr \cap U$ as shown in Figure \ref{fig:delta_integrate_to_gauge} (here $U_\pm$ are chosen according to the orientation of the ray $Q_\wptr$).

Results from \cite[Section 4]{kwchan-leung-ma} imply that
\begin{equation}\label{eqn:wall_crossing_factor}
\facs
:= \left\{
\begin{array}{ll}
\displaystyle \text{Log}(\Theta_\wptr) & \text{on $U_+$}\\
\displaystyle 0 & \text{on $U_-$}
\end{array}\right.
\end{equation}
is the unique gauge solving the equation
\begin{equation}\label{eqn:gauge_equivalent}
e^{\text{ad}_\facs}  \bar{\partial}e^{-\text{ad}_\facs} = \bar{\partial} - \left[ \left( \frac{e^{\text{ad}_\facs}- \text{Id}}{\text{ad}_\facs} \right) \bar{\partial}(\facs),\cdot \right] = \bar{\partial} + [\varXi^{1,1},\cdot]
\end{equation}
and satisfying the condition that $\facs|_{U_-} = 0$. The Maurer-Cartan solution $\Xi^{1,1}$ behaves like a `delta-form' supported along the wall $Q_\wptr$, while the local gauge $\varphi$ in $U$ behaves like a step-function jumping across the wall $Q_\wptr$ as shown in Figure \ref{fig:delta_integrate_to_gauge}.

\begin{figure}[h]
	\centering
	\includegraphics[scale=0.4]{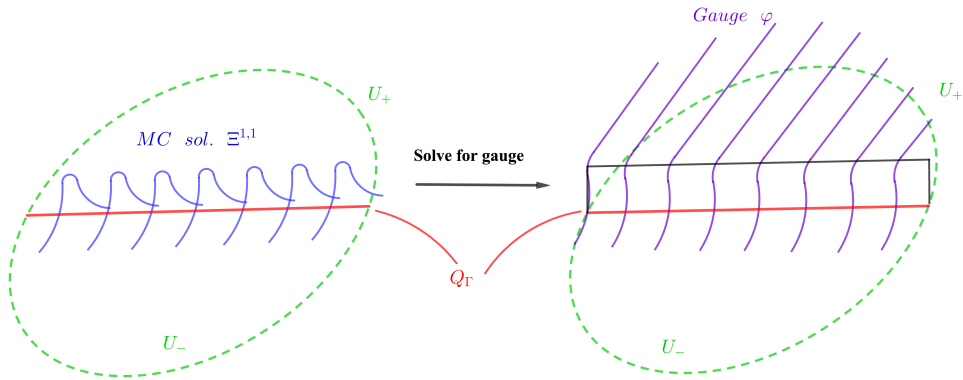}
	\caption{The gauge $\varphi$ as step-function}
	\label{fig:delta_integrate_to_gauge}
\end{figure}

When the extended MC equation \eqref{eqn:extended_MC_equation_intro} is restricted to $U$, we have $\Phi|_U = \varXi^{0,0} + \varXi^{1,1}$ and it decomposes into two equations:
\begin{align*}
\bar{\partial}\varXi^{1,1} + \half [\varXi^{1,1}, \varXi^{1,1}]  & = 0, \\
\bar{\partial}\varXi^{0,0} + [W, \varXi^{1,1}] + [\varXi^{0,0}, \varXi^{1,1}]  & = 0.
\end{align*}
The second equation can be rewritten as
$\left(\bar{\partial} + [\varXi^{1,1},\cdot]\right) \left(W + \varXi^{0,0}\right) = 0,$
meaning precisely that $W + \varXi^{0,0}$ is a {\em holomorphic function} with respect to the complex structure defined by $\bar{\partial} + [\varXi^{1,1},\cdot]$.
Since $e^{\text{ad}_{\facs}}  \bar{\partial} e^{-\text{ad}_{\facs}} = \bar{\partial} + [\varXi^{1,1},\cdot]$, this is equivalent to saying that the function $e^{-\text{ad}_\facs}(W + \varXi^{0,0})$, which is globally defined on $U$, is holomorphic with respect to the {\em original} Dolbeault operator $\bar{\partial}$.

Now letting $W_{n,\pm} := \left(W + \varXi^{0,0}\right)|_{U_{\pm}}$ on $U_{\pm}$ respectively, we have
$$
e^{-\text{ad}_\facs}(W + \varXi^{0,0}) = \left\{
\begin{array}{ll}
\displaystyle \Theta_\wptr^{-1}(W_{n,+}) & \text{on $U_+$},\\
\displaystyle W_{n,-} & \text{on $U_-$}.
\end{array}
\right.
$$
So the fact that $e^{-\text{ad}_\facs}(W + \varXi^{0,0})$ is a {\em globally defined} function on $U$ implies that $\Theta_\wptr(W_{n,-}) = W_{n,+}$.
Applying this to a finite number of walls, we obtain the following wall-crossing formula which originally appeared in Gross \cite{Gross10}:
\begin{corollary}[Theorem 4.12 in Gross \cite{Gross10}]
If $Q,Q' \in \real^2$ are not lying on any walls in the scattering diagram $\mathscr{D}$, then we have
\begin{equation}\label{eqn:wall_crossing_formula}
W_n(Q') = \Theta_{\gamma,\mathscr{D}}(W_n(Q)),
\end{equation}
for any path $\gamma \subset M_\real \setminus \text{Sing}(\mathscr{D})$ joining $Q$ to $Q'$.
\end{corollary}

\subsection{Remarks}
We end this introduction by a couple remarks.
\begin{enumerate}
\item
Just as in Gross \cite{Gross10}, the tropical disks which appear here (see Section \ref{sec:counting}) correspond to holomorphic disks in the toric variety $X$ which are transversal to the toric divisor $D_\infty$, so the tropical counts would give the correct Gromov-Witten invariants only when $X$ is a toric Fano surface.
One way to get the correct Gromov-Witten invariants in general is to apply the technique of {\em tropical modification} which deforms the ambient space so that the hidden curves (i.e. curves lying inside the toric divisor $D_\infty$) can be seen; see e.g. \cite{Kalinin15} for an introduction. We expect that our results would still hold if we replace the Hori-Vafa potential with the Lagrangian Floer potential $W^{\text{LF}}$, and in that case $W_n$ would correspond to the bulk-deformed potential function defined using suitable tropical modifications.

\item
To generalize our results to higher dimensions, one again needs a proper definition of tropical counts. In that case, we shall allow interior insertions of tropical cycles of different codimensions, instead of just points in generic position. It should be straightforward to generalize our results when only point insertions are involved. For other tropical cycles, what is missing is a description of such cycles by means of elements in the tropical dgLa as in equation \eqref{eqn:introduction_input}.

\end{enumerate}

\section*{Acknowledgement}
We thank Conan Leung and Grigory Mikhalkin for their encouragement and interest in this work. The first named author would also like to thank Nikita Kalinin and Grigory Mikhalkin for very useful discussions on the technique of tropical modification. The second named author would like to thank Si Li for useful discussions on the polyvector field approach to deformation theory of LG models, and also hospitality during his visit at the Yau Mathematical Sciences Center (YMSC) of Tsinghua University. We also thank the anonymous referee for carefully reading our paper and providing useful comments.

The work of K. Chan was supported by grants of the Hong Kong Research Grants Council (Project No. CUHK14302015 $\&$ CUHK14314516).  The work of Z. N. Ma was partially supported by the Institute of Mathematical Sciences (IMS) and Department of Mathematics at The Chinese University of Hong Kong. 
\section{Tropical counting in dimension 2}\label{sec:counting}

We fix, once and for all, a rank $2$ lattice $M$ together with its dual lattice $N$, and write $M_\real = M \otimes_\inte \real$ and $N_\real = N \otimes_\inte \real$ for the corresponding real vector spaces. An element in $M$ (resp. $N$) will be denoted by $\mathsf{m}$ (resp. $n$). Let $\Sigma\subset M_\real$ be a complete rational polyhedral fan and $X_\Sigma$ be the associated toric surface. We take $K := \mathcal{K}_{X_\Sigma} \cap H^2(X_\Sigma,\inte)$ to be the monoid of integral K\"ahler forms, where $\mathcal{K}_{X_\Sigma}$ is the K\"ahler cone of $X_\Sigma$. We also use $\Sigma(1)$ to denote the set of $1$-dimensional cones in $\Sigma$ and $D_\rho$ to denote the toric divisor corresponding to $\rho \in \Sigma(1)$. The purpose of this section is to define the counting of tropical disks following \cite{Mikhalkin05, Gross10}, with slight modifications.

\subsection{The Hori-Vafa LG mirror}

\begin{notation}\label{not:initial_monoid}

If we fix a Lagrangian torus fiber $L$ of the moment map $p: X_\Sigma \to \mathbf{P}$, then $\pi_2(X, L) \cong \sum_{\rho \in  \Sigma(1)} \inte \cdot m_\rho$ is freely generated by the classes $m_\rho$'s of Maslov index 2 holomorphic disks in $(X, L)$, where $m_\rho$ is the unique disk class which intersects the toric divisor $D_\rho$ exactly once.

Let $\mathcal{P}$ be the nonnegative cone $\pi_2(X, L)_{\geq 0} \subset \pi_2(X, L)$ generated by classes $\beta \in \pi_2(X, L)$ such that $\int_{\beta} \omega \geq 0$ for all $\omega \in K$, and $\mathcal{Q}:=\pi_2(X) \cap \pi_2(X, L)_{\geq 0}$ be the effective cone (or Mori cone) of $X_\Sigma$. We have $\mathcal{P}^{gp} = \pi_2(X, L)$ and $\mathcal{Q}^{gp} = \pi_2(X)$, where $\mathcal{P}^{gp}$ and $\mathcal{Q}^{gp}$ are the abelian groups associated to $\mathcal{P}$ and $\mathcal{Q}$ respectively, and the exact sequence of monoids
\begin{equation}\label{monoid_sequence}
\xymatrix{0 \ar[r] & \mathcal{Q}  \ar[r]  & \mathcal{P}   \ar[r]^{\theta} & M \ar[r] & 0}
\end{equation}
where $M$ is identified with $\pi_1(L)$ and $\theta(\beta) := \partial \beta$ is the map taking boundary of the class $\beta \in \mathcal{P}$.

We will use $m$ to denote an element of $\mathcal{P}$ and $\bar{m}$ to denote its image $\theta(m) \in M$.
Note that $\theta: \mathcal{P} \to M$ maps the standard basis $\{m_\rho\} \subset \mathbb{Z}^{|\Sigma(1)|}$ to the generators of the 1-dimensional cones $\rho \in \Sigma(1)$.
\end{notation}

\begin{example}
We consider a fan $\Sigma$ in $M_\real \cong \inte^2$, with its $1$-dimensional cones given by $\rho_i = \real \cdot \mathsf{m}_i$ where we have $\mathsf{m}_1 = (0,1)$, $\mathsf{m}_2 = (-1,0)$, $\mathsf{m}_3 = (0,-1)$ and $\mathsf{m}_4 = (1,1)$. The corresponding toric variety $X_\Sigma$ is the surface $F_1$. We write $\pi_2(X_\Sigma,L) = \bigoplus_{i=1}^4 \inte \cdot m_i$ with $m_i$ being the unique disk class which supports a Maslov index $2$ holomorphic disk intersecting exactly once with the toric boundary $D_{\rho_i}$. In this case $\mathcal{Q}$ is the monoid of integral points in the cone generated by $-m_1 + m_2 +m_4$ and $m_1 + m_3$, while $\mathcal{P}$ is the monoid of integral points in the cone generated by the $m_i$'s together with $-m_1 + m_2 +m_4$.
\end{example}

Consider the polynomial rings $\comp[\mathcal{P}] := \comp[z^m \mid m \in \mathcal{P}]$ and $\comp[\mathcal{Q}] := \comp[q^d \mid d \in \mathcal{Q}]$ and the corresponding affine toric varieties $\check{\mathcal{X}} := \text{Spec}(\comp[\mathcal{P}])$ and $\mathcal{S} = \text{Spec}(\comp[\mathcal{Q}])$. There is a natural morphism $\pi : \check{\mathcal{X}} \rightarrow \mathcal{S}$ induced by the map $\mathcal{Q} \rightarrow \mathcal{P}$ in \eqref{monoid_sequence}. As a result we obtain a toric degeneration of the LG model over $\mathcal{S}$ as in \cite{Gross10, gross2011tropical}. 

\begin{definition}
The {\em Hori-Vafa potential} is the polynomial $W := \sum_{\rho \in \Sigma(1)} z^{m_{\rho}}$ on $\check{\mathcal{X}}$.
\end{definition}

\subsection{Tropical disks}

Before defining tropical disks, we first introduce the underlying combinatorial structures:
\begin{definition}\label{k_tree_def}
A {\em (directed) $k$-tree} $\tr$ consists of a finite set of vertices $\bar{\tr}^{[0]}$ together with a decomposition
$\bar{\tr}^{[0]} = \tr^{[0]}_{in} \sqcup \tr^{[0]} \sqcup \{v_o\},$
where $\tr^{[0]}_{in}$, called the set of incoming vertices, is a set of size $k$ and $v_o$ is called the outgoing vertex (we also write $\tr^{[0]}_\infty := \tr^{[0]}_{in} \sqcup \{v_o\}$), a finite set of edges $\bar{\tr}^{[1]}$, and two boundary maps $\partial_{in} , \partial_o : \bar{\tr}^{[1]} \rightarrow \bar{\tr}^{[0]}$ (here $\partial_{in}$ stands for incoming and $\partial_o$ stands for outgoing), satisfying the following conditions:
\begin{enumerate}
\item
Every vertex $v \in \tr^{[0]}$ is trivalent, and satisfies $\# \partial_{o}^{-1}(v) = 2$ and $\# \partial_{in}^{-1}(v) = 1$.
\item
Every vertex $v \in \tr^{[0]}_{in}$ has valency one, and satisfies $\# \partial_{o}^{-1}(v) = 0$ and $\# \partial_{in}^{-1}(v) = 1$; we let $\tr^{[1]} := \bar{\tr}^{[1]}\setminus \partial_{in}^{-1}(\tr^{[0]}_{in})$.
\item
For the outgoing vertex $v_o$, we have $\# \partial_{o}^{-1}(v_o) = 1$ and $\# \partial_{in}^{-1}(v_o) = 0$; we let $e_o := \partial_o^{-1}(v_o)$ be the outgoing edge and denote by $v_r \in \tr^{[0]}_{in} \sqcup \tr^{[0]}$ the unique vertex (which we call the root vertex) with $e_o = \partial^{-1}_{in}(v_r)$.
\item
The {\em topological realization}
$|\bar{\tr}| := \left( \coprod_{e \in \bar{\tr}^{[1]}} [0,1] \right) / \sim$
of the tree $\tr$ is connected and simply connected; here $\sim$ is the equivalence relation defined by identifying boundary points of edges if their images in $\tr^{[0]}$ are the same.
\end{enumerate}
Two $k$-trees $\tr_1$ and $\tr_2$ are {\em isomorphic} if there are bijections $\bar{\tr}^{[0]}_1 \cong \bar{\tr}^{[0]}_2$ and $\bar{\tr}^{[1]}_1 \cong \bar{\tr}^{[1]}_2$ preserving the decomposition $\bar{\tr}^{[0]}_i = \tr^{[0]}_{i,in} \sqcup \tr^{[0]}_i \sqcup \{v_{i,o}\}$ and boundary maps $\partial_{i,in}$ and $\partial_{i,o}$. The set of isomorphism classes of $k$-trees will be denoted by $\tree{k}$. For a $k$-tree $\tr$, we abuse notations and use $\tr$ (instead of $[\tr]$) to denote its isomorphism class.

A ribbon $k$-tree is a $k$-tree $\tr$ with a cyclic ordering of $\partial_{in}^{-1}(v) \sqcup \partial_o^{-1}(v)$ for each trivalent vertex $v \in \tr^{[0]}$, and isomorphism of ribbon $k$-trees are required to preserve this ordering.
\end{definition}

\begin{notation}\label{not:deformation_ring}
As in \cite{Gross10, gross2011tropical}, we take $n$ points $P_1,\dots,P_n \in M_\real$ and introduce the ring 
$$R = R_n := \frac{\comp\left[u_1,\dots,u_n \right]}{\left(u_{i}^2 \mid 1\leq i \leq n \right)},$$
where we associate to each point $P_i$ the variable $u_i$. This ring has the maximal ideal $\mathbf{m} = \mathbf{m}_n := (u_1,\dots,u_n)$.
\end{notation}

\begin{definition}\label{def:weighted_tree}
A {\em weighted $d$-pointed $k$-tree} 
is a $(k+d)$-tree $\wptr$ together with an injective map $p: \{1,\dots,d\} \hookrightarrow \partial_{in}^{-1} (\wptr^{[0]}_{in})$ (we use $p_j$ to denote the image $p(j) \in \partial_{in}^{-1} (\wptr^{[0]}_{in})$), a weight $m : \bar{\wptr}^{[1]}  \rightarrow \mathcal{P}$ (we use $m_e$ to denote the image $m(e) \in \mathcal{P}$) and a map $u : \bar{\wptr}^{[1]} \rightarrow R_n$ (we use $u_e$ to denote the image $u(e) \in R_n$) satisfying the following conditions:
\begin{enumerate}
\item
for every $e \in \bar{\wptr}^{[1]}$, $u_e \in R_n$ is a monomial defined via the rule: for each $j = 1, \dots, d$, we have $u_{p_j} = u_{i_j}$ for some $i_j \in \{1, \dots, n\}$ such that $1 \leq i_1 < i_2 < \cdots < i_d \leq n$, and for each $e \in \partial^{-1}_{in}(\wptr^{[0]}_{in}) \setminus \{p_1,\dots,p_d\}$, we have $u_e = 1$;
\item
for every trivalent vertex $v \in \wptr^{[0]}$ attached with two incoming edges $e_1, e_2$ and an outgoing edge $e_3$ we require at least one of $e_1, e_2$ do not belong to $\partial^{-1}_{in}(\wptr^{[0]}_{in}) \setminus \{p_1,\dots,p_d\}$, and we further require $u_{e_3} = u_{e_1} \cdot u_{e_2}$ and $m_{e_3} = m_{e_1} + m_{e_2}$;
\item
$m_e = 0 $ if and only if $e \in \{p_1,\dots,p_d\}$;
\item
for every $e \in \partial^{-1}_{in}(\wptr^{[0]}_{in}) \setminus \{p_1,\dots,p_d\}$, we have $m_e = m_\rho$ for some $\rho \in \Sigma(1)$.
\end{enumerate}
Two weighted $d$-pointed $k$-trees $\wptr_1$ and $\wptr_2$ are said to be {\em isomorphic} if they are isomorphic as $k$-trees and the isomorphism preserves the marked points $p_i$'s and the weight functions $m_e$'s. The set of isomorphism classes of weighted $d$-pointed $k$-trees will be denoted by $\wptree{k}{d}$.

A {\em weighted ribbon $d$-pointed $k$-tree} is a weighted $d$-pointed $k$-tree $\Gamma$ equipped with a ribbon structure on it such that if $e_1$ and $e_2$ are the incoming edges of $v$ with outgoing edge $e_3$ with $e_1,e_2,e_3$ in cyclic ordering, then only $e_1$ can possibly be an edge from $\Gamma^{[0]}_{in}=\partial^{-1}_{in}(\Gamma^{[0]}_{in}) \setminus \{p_1,\dots,p_d\}$. Isomorphisms between these trees are defined as isomorphisms between weighted $d$-pointed $k$-trees which preserve the relevant structures. The set of isomorphism classes of weighted ribbon $d$-pointed $k$-tree is denoted by $\wrtree{k}{d}$.

For a weighted $d$-pointed $k$-tree $\wptr$ (or weighted ribbon $d$-pointed $k$-tree resp.), we abuse notations and use $\wptr$ ($\wrtr$ resp.) instead of $[\wptr]$ ($[\wrtr]$ resp.) to stand for its isomorphism class.
\end{definition}

\begin{notation}\label{not:integer_weight}
Given a weighted $d$-pointed $k$-tree $\wptr$ ($d \geq 0$ and $k \geq 1$), we write $\wptr^{[1]}_{in} := \partial^{-1}_{in}(\wptr^{[0]}_{in}) \setminus \{p_1,\dots,p_d\}$ for the set of incoming edges excluding those which correspond to the marked points.
	
Given any $e \in \bar{\wptr}^{[1]} \setminus \{p_1,\dots,p_d\}$, we let $k_e = 0$ if $\bar{m}_e = 0$, and when $\bar{m}_e \neq 0$, we let $k_e \in \inte_{\geq 0}$ be the unique positive integer such that $\bar{m}_e = k_e \hat{m}_e$ where $\hat{m}_e \in M$ is the primitive element.
The integers $\{k_e\}_{e \in \bar{\wptr}{[1]} \setminus \{p_1,\dots,p_d\}}$ define the {\em weight function} $w_{\wptr}: \bar{\wptr}{[1]} \setminus \{p_1,\dots,p_d\}\to \inte_{\geq 0}$ (hence the name ``weighted $k$-tree'') and the formula $\bar{m}_{e_2} = \bar{m}_{e_1} + \bar{m}_{e_0}$ corresponds to the {\em balancing condition}, both of which appear in the original definition of tropical curves in \cite{Mikhalkin05, Gross10}.
	
We also write $m_{\wptr}$ (or $k_{\wptr}$) and $u_{\wptr}$ for the weight and monomial, respectively, associated to the unique outgoing edge $e_o$ attached to the unique outgoing vertex $v_o$ of $\wptr$.
\end{notation}

\begin{definition}\label{def:multiplicity}
Given a $d$-pointed weighted $k$-tree $\wptr$, we define the {\em multiplicity at a trivalent vertex} $v \in \wptr^{[0]}_0 := \wptr^{[0]}\setminus \partial_o(\{p_1,\dots,p_d\})$ by
$$
\text{Mult}_{v}(\wptr) := |\det(\bar{m}_{e_1},\bar{m}_{e_2})| = |\det(\bar{m}_{e_1},\bar{m}_{e_3})| = |\det(\bar{m}_{e_2},\bar{m}_{e_3})|,
$$
where $e_1, e_2$ are the incoming edges and $e_3$ the outgoing edge attached to $v$.
Then we define the {\em multiplicity} $\text{Mult}(\wptr)$ of $\wptr$ by
$
\text{Mult}(\wptr) := \prod_{v\in \wptr_0^{[0]}} \text{Mult}_v(\wptr).
$
\end{definition}
Note that at a trivalent vertex $v \in \wptr^{[0]}_0$ with incoming edges $e_1, e_2$, the multiplicity $\text{Mult}_{v}(\wptr) \neq 0$ if and only if $\bar{m}_{e_1}, \bar{m}_{e_2}$ are linearly independent in $M_{\real}$.

Given a weighted $d$-pointed $k$-tree $\wptr$, a {\em realization} of $\wptr$ is defined as
$
|\wptr_{\vec{s}}| := \Big( \big(\bigsqcup_{e \in \partial_{in}^{-1}(\wptr^{[0]}_{in})} (\real_{\leq 0})_e \big) \sqcup \left(\bigsqcup_{e \in \wptr^{[1]}} [s_e, 0] \right) \Big) / \sim,
$
for a set of parameters $\vec{s} := (s_e)_{e\in \wptr^{[1]}} \in (\real_{<0})^{|\wptr^{[1]}|}$; here $(\real_{\leq 0})_e$ is just a copy of $\real_{\leq 0}$ and $\sim$ is the equivalence relation defined by identifying boundary points of edges if their images in $\wptr^{[0]}$ are the same. The set of realizations of $\wptr$ is parametrized by $\vec{s} \in (\real_{<0})^{|\wptr^{[1]}|}$.\\

\begin{definition}\label{def:tropical_curve}
A {\em $d$-pointed tropical disk $\varsigma$ 
in $(\mathcal{P},\Sigma,P_1,\dots,P_n;Q)$} consists of a weighted $d$-pointed $k$-tree $\Gamma$ with $u_{\wptr} \neq 0$, a set of parameters $\vec{s} = (s_e)_{e\in \Gamma^{[1]}} \in (\real_{<0})^{|\Gamma^{[1]}|}$, and a proper map $\varsigma: |\Gamma_{\vec{s}}| \rightarrow M_\real$ from the realization
$|\Gamma_{\vec{s}}|$ of $\Gamma$ to $M_\real$, satisfying the following conditions:
\begin{enumerate}
\item
$\varsigma|_{(\real_{\leq 0})_{p_j}} \equiv P_{i_j}$ if the monomial assigned to $p_j$ is $u_{i_j}$; in particular $\varsigma|_{(\real_{\leq 0})_{p_j}}$ a constant map (playing the role of a marked point).
\item
For each incoming edge $e\in \Gamma^{[1]}_{in}$, we have $\varsigma|_{(\real_{\leq 0})_e}(s) = \varsigma|_{(\real_{\leq 0})_e}(0) + s (-\bar{m}_{e})$ for all $s \in \real_{\leq 0}$.
\item
For each $e \in \Gamma^{[1]}$, we have $\varsigma|_{[s_e,0]}(s) = \varsigma|_{[s_e,0]}(0) + s (-\bar{m}_e)$ for all $s \in [s_e, 0]$ (so that the image $\text{Im}(\varsigma|_{[s_e,0]})$ is an affine line segment with slope $-\bar{m}_e$).
\item
The point $\varsigma(v_o) := \varsigma|_{[s_{e_0},0]}(0) \in M_\real$ is called the {\em stop} of the tropical disk $\varsigma$ and we require that $\varsigma(v_o) = Q$.
\end{enumerate}
	
The {\em multiplicity} $\text{Mult}(\varsigma)$ of a tropical disk $\varsigma$ is defined as the multiplicity $\text{Mult}(\Gamma)$ of the underlying weighted $d$-pointed $k$-tree $\Gamma$. Note that $\text{Mult}(\varsigma) \neq 0$ if and only if the images of the two incoming edges at any trivalent vertex are intersecting transversally. The underlying tree $\wptr$ is said to be the {\em combinatorial type} of the tropical disk $\varsigma$. We use $\mathfrak{M}_d^{\wptr} (\mathcal{P},\Sigma,P_1,\dots,P_n;Q)$ to denote the moduli space of tropical disks in $(\mathcal{P},\Sigma,P_1,\dots,P_n;Q)$ with a fixed combinatorial type $\wptr$.

Similarly, we define a {\em tropical disk $\varsigma$ in $(\mathcal{P},\Sigma, P_1,\dots,P_n)$} by allowing the stop $Q$ to vary or by dropping condition (4) above, and we denote by $\mathfrak{M}_d^{\wptr} (\mathcal{P},\Sigma,P_1,\dots,P_n)$ the moduli space of tropical disks in $(\mathcal{P},\Sigma,P_1,\dots,P_n)$ with a fixed combinatorial type $\wptr$. In other words,
$\mathfrak{M}_d^{\wptr} (\mathcal{P},\Sigma,P_1,\dots,P_n) = \bigcup_{Q} \mathfrak{M}_d^{\wptr} (\mathcal{P},\Sigma,P_1,\dots,P_n;Q).$ Notice that there is a natural $\real_+$ action on $\mathfrak{M}_d^{\wptr} (\mathcal{P},\Sigma,P_1,\dots,P_n)$ given by translating the stop $\varsigma(v_o) = Q$ along the direction $-\bar{m}_{e_o}$, so we have a well-defined quotient
$\mathfrak{M}_d^{\wptr} (\mathcal{P},\Sigma,P_1,\dots,P_n) / \real_+,$
which can be regarded as the moduli space of tropical disks as the stop $Q$ goes to infinity along the direction $-\bar{m}_{e_o}$; see \cite{Gross10, gross2011tropical}.

We further define a {\em tropical disk $\varsigma$ in $(\mathcal{P},\Sigma)$} with a fixed combinatorial type $\wptr$ by dropping condition (4) above and replacing condition (1) by only requiring that $\varsigma|_{(\real_{\leq 0})_{p_j}}$ is a constant map for each $j = 1, \dots, d$
\end{definition}

The reader may ask why all the internal vertices $\wptr^{[0]}$ are required to be trivalent. Indeed we have only defined {\em generic} tropical disks and the above moduli spaces are all noncompact. We use this approach because this suffices for the purpose of tropical counting. To compactify these moduli spaces, we need to allow the intervals $[s_e,0]$'s corresponding to the internal edges $e \in \wptr^{[1]}$ to shrink to zero lengths (i.e. by allowing $s_e = 0$), so that some internal vertices are allowed to be of higher valencies.

We use $\overline{\mathfrak{M}}_d^{\wptr} (\mathcal{P},\Sigma)$ to denote the compactified moduli space of tropical disks $\wptr$ in $(\mathcal{P},\Sigma)$ with a fixed combinatorial type thus obtained, which gives a compactification of the union of the moduli spaces $\mathfrak{M}_d^{\wptr} (\mathcal{P},\Sigma,P_1,\dots,P_n)$ as $P_1, \dots, P_n$ vary. We use the notation $\partial \overline{\mathfrak{M}}_d^{\wptr} (\mathcal{P},\Sigma)$ to stand for the set of tropical disks with at least one degenerated internal edges (i.e. $s_e = 0$ for some $e \in \wptr^{[1]}$).


It is not hard to see that
$
\overline{\mathfrak{M}}_d^{\wptr} (\mathcal{P},\Sigma)\cong (\real_{\leq 0})^{|\wptr^{[1]}|} \times M_\real,
$
where the first component $\vec{s} \in (\real_{\leq 0})^{|\wptr^{[1]}|}$ parametrizes the realization $|\wptr_{\vec{s}}|$ of $\wptr$ and the second component $M_\real$ parametrizes the stop $\varsigma(v_o)$. Its dimension is given by
\begin{equation}
\text{dim}_{\real} (\overline{\mathfrak{M}}_d^{\wptr} (\mathcal{P},\Sigma)) = |\Delta(\wptr)| + d +1,
\end{equation}
where $|\Delta(\wptr)| := k$ for a $d$-pointed weighted $k$-tree $\wptr$. This moduli space has a natural stratification coming from the one on $(\real_{\leq 0})^{|\wptr^{[1]}|}$ given naturally by the coordinate hyperplanes $s_e = 0$.

We also need to consider the partial compactification
$$\hat{\mathfrak{M}}_d^{\wptr} (\mathcal{P},\Sigma) := \left(\overline{\mathfrak{M}}_d^{\wptr} (\mathcal{P},\Sigma) \setminus \{\varsigma \mid s_{e_o} = 0\} \right) / \real_+,$$
and we denote by $\partial \hat{\mathfrak{M}}_d^{\wptr} (\mathcal{P},\Sigma)$ the set of tropical disks with $s_e = 0$ for some $e \in \wptr^{[1]} \setminus \{e_o\}$.

\begin{definition}\label{def:evaluation_map}
We define the evaluation maps
$ev_{*} : \overline{\mathfrak{M}}_d^{\wptr} (\mathcal{P},\Sigma) \rightarrow M_\real,$
where $* \in \{1, \dots, d\} \cup \{o\}$, to be the evaluation at a marked point $ev_{*}(\varsigma) = \varsigma(p_*)$ when $* \in \{1,\dots,d\}$, and the evaluation at the outgoing vertex $ev_*(\varsigma) = \varsigma(v_o)$ if $* = o$.
We put these evaluation maps together to obtain the map
$\vec{ev} = (ev_1, \dots, ev_d, ev_o): \overline{\mathfrak{M}}_d^{\wptr} (\mathcal{P},\Sigma) \rightarrow M_\real^{d+1}.$ Similarly, we have the evaluation map
$\hat{ev} = (ev_1, \dots, ev_d): \hat{\mathfrak{M}}_d^{\wptr} (\mathcal{P},\Sigma) \rightarrow M_\real^{d}.$
\end{definition}

\begin{definition}\label{def:generic_position}
We say that $n$ distinct points $P_1, \dots, P_n$ are in {\em generic position} if for any $d \leq n$, any $d$-tuple $(P_{i_1}, \dots, P_{i_d})$ is not lying in the image $\hat{ev}(S)$ of a stratum $S \subset \hat{\mathfrak{M}}_d^{\wptr} (\mathcal{P},\Sigma)$ over which the evaluation map $\hat{ev}$ is degenerated, meaning that the differential $D(\hat{ev}|_S)$ is not surjective (notice that $\hat{ev}|_S$ is an affine map and hence $D(\hat{ev}|_S)$ is a well-defined constant linear map), and this holds for any combinatorial type $\wptr$.
	
We say that $n+1$ distinct points $P_1, \dots, P_n, Q$ are in {\em generic position} if the $n$ points $P_1, \dots, P_n$ are in generic position, and for any $d \leq n$ and any $d$-tuple $(P_{i_1}, \dots, P_{i_d})$, the $(d+1)$-tuple $(P_{i_1}, \dots, P_{i_d}, Q)$ is not lying in the image $\vec{ev}(S)$ of a stratum $S \subset (\overline{\mathfrak{M}}_d^{\wptr} (\mathcal{P},\Sigma))$ over which the evaluation map $\vec{ev}$ is degenerated and this holds for any combinatorial type $\wptr$.
\end{definition}


\begin{definition}\label{def:maslov_index}
We define the {\em Maslov index} $MI(\wptr)$ of a weighted $d$-pointed $k$-tree $\wptr$ by
$MI(\wptr) = 2 (k - d),$
and the Maslov index $MI(\varsigma)$ of a tropical disk $\varsigma$ to be that of its combinatorial type $\wptr$.
\end{definition}

\begin{lemma}[Lemma 2.6 in \cite{Gross10}]\label{lem:moduli_space_dimension}
If $P_1,\dots,P_n, Q$ are in generic position and $MI(\wptr) = 2r$, then $\mathfrak{M}_d^{\wptr} (\mathcal{P},\Sigma,P_1,\dots,P_n;Q)$ is an $(r-1)$-dimensional (over $\real$) affine linear subspace of $\overline{\mathfrak{M}}_d^{\wptr}(\mathcal{P},\Sigma) \setminus \partial\overline{\mathfrak{M}}_d^{\wptr}(\mathcal{P},\Sigma)$; in particular, $\mathfrak{M}_d^{\wptr} (\mathcal{P},\Sigma,P_1,\dots,P_n;Q) = \emptyset$ when $r \leq 0$.
	
If $P_1, \dots, P_n$ are in generic position and $MI(\wptr) = 2r$, then $\mathfrak{M}_d^{\wptr} (\mathcal{P},\Sigma,P_1,\dots,P_n)/\real_+$ is an $r$-dimensional (over $\real$) affine linear subspace of $\hat{\mathfrak{M}}_d^{\wptr} (\mathcal{P},\Sigma) \setminus \partial\hat{\mathfrak{M}}_d^{\wptr} (\mathcal{P},\Sigma)$; in particular, we have $\mathfrak{M}_d^{\wptr} (\mathcal{P},\Sigma,P_1,\dots,P_n) /\real_+ = \emptyset$ when $r < 0$.
\end{lemma}

\subsection{The perturbed Landau-Ginzburg potential}
We now define an $n$-pointed LG potential as a perturbation of the Hori-Vafa mirror family $(\check{\mathcal{X}},W)$ by tropical disk counts, following \cite{Gross10}.

\begin{definition}\label{def:mono_LG_potential}
Given a tropical disk $\varsigma \in \mathfrak{M}_d^{\wptr}(\mathcal{P}, \Sigma, P_1, \dots, P_n ; Q)$ with $MI(\varsigma) =2 $, the {\em monomial associated to $\varsigma$} is defined by
$$\text{Mono}(\varsigma) := \text{Mult}(\wptr) z^{m_{\wptr}} u_{\wptr} \in \comp[\mathcal{P}],$$
where $m_{\wptr} \in \mathcal{P}$ is the weight and $u_\wptr$ is the monomial associated to the unique outgoing edge $e_o$ of $\varsigma$ as in Definition \ref{def:weighted_tree}.
\end{definition}

\begin{definition}[cf. Definition 2.7 in \cite{Gross10}]\label{def:LG_potential}
Fixing the points $P_1, \dots, P_n, Q$ in generic position, we define the {\em $n$-pointed Landau-Ginzburg (LG) potential} as
$$W_n(Q) := \sum_{\varsigma} \text{Mono}(\varsigma),$$
where the sum is over all Maslov index $2$ tropical disks $\varsigma$ in $(P_1, \dots, P_n; Q)$.
\end{definition}

Note that the $0$-pointed LG potential $W_0(Q) = W$ is precisely the Hori-Vafa potential, so the $n$-pointed potential $W_n(Q)$ is indeed a deformation (in the formal variables $u_i$'s) of $W$.

\subsection{Scattering diagram from the Maslov index $0$ disks}
According to \cite{Gross10, gross2011tropical}, the dependence of the $n$-pointed LG potential $W_n(Q)$ on $Q$ is governed by a scattering diagram constructed from the Maslov index $0$ tropical disks. Here we recall the definition of scattering diagrams from \cite[Section 3]{kwchan-leung-ma} with slight modifications; the original definition was due to Kontsevich-Soibelman \cite{kontsevich-soibelman04} and can be found in \cite{gross2010tropical}.

\subsubsection{Tropical vertex group}
We consider $\comp[\mathcal{P}] \otimes_\inte N$, whose general elements are finite linear combinations of elements of the form $z^{m} \otimes \check{\partial}_n$ (here $\check{\partial}_n$ is a holomorphic vector field on $\check{X}$ associated to $n\in N$ to be defined in \eqref{holo_vector_field_2}).  We also define the Lie-bracket $[\cdot,\cdot]$ on $\comp[\mathcal{P}] \otimes_\inte N$ by the formula:
\begin{equation}\label{vertex_lie_algebra}
\left[ z^m \otimes \check{\partial}_n , z^{m'} \otimes \check{\partial}_{n'} \right]= z^{m+m'} \check{\partial}_{( \bar{m}',n ) n' - ( \bar{m} ,n' ) n},
\end{equation}
where $(\cdot,\cdot)$ is the natural pairing between $M$ and $N$. We consider the Lie algebra
$\mathfrak{g} := \left(\comp[\mathcal{P}] \otimes_\inte N \right) \otimes_\comp R_n,$
where $R_n$ is the formal power series ring $R_n$ in Notation \ref{not:deformation_ring} equipped with its maximal ideal $\mathbf{m}$.

\begin{definition}[\cite{gross2010tropical}]\label{trop_lie_algebra}
The {\em tropical Lie-algebra} over $R = R_n$ is defined to be the nilpotent Lie subalgebra $\mathfrak{h} \hookrightarrow \mathfrak{g}$ given explicitly by
$\big(\bigoplus_{m \in \mathcal{P} \setminus \{0\}} \comp \cdot z^m \otimes_\inte (m^\perp)\big) \otimes_\comp R \rightarrow \mathfrak{g}.$
The {\em tropical vertex group} is defined as the exponential group of $\mathfrak{h}$.
\end{definition}

\begin{definition}
Given $m \in \mathcal{P} \setminus\{0\}$ and $n \in m^\perp$, we let
$\mathfrak{h}_{m,n} := (\comp[z^m]\cdot z^m)\check{\partial}_n \otimes_\comp \mathbf{m}  \hookrightarrow \mathfrak{h},$
whose general elements are of the form
$\sum_{k \geq 1} \sum_{I} a_{k,I} z^{km}\check{\partial}_{n} u_I,$
where $I \subset \{1,\dots,n\}$. This defines an abelian Lie subalgebra of $\mathfrak{h}$ by \eqref{vertex_lie_algebra}.
\end{definition}

\begin{definition}\label{wall}
A {\em wall} $\mathbf{w}$ over $R$ is a triple $(m, Q, \Theta)$, where
\begin{itemize}
\item
$m \in \mathcal{P} \setminus \{0\}$ such that $\bar{m}$ parallel to $Q$,
\item
$Q$, called the {\em support} of $\mathbf{w}$, is a connected oriented codimension one convex tropical polyhedral subset of $M_\real$ \footnote{It means a connected convex subset locally defined by affine linear equations and inequalities defined over $\mathbb{Q}$.},
\item
$\Theta \in \exp(\mathfrak{h}_{m,n_Q})$, where $n_Q \in  N$ is the unique primitive element satisfying $n_Q \in (TQ)^\perp$ and $(\nu_Q, n) < 0$, and $\nu_Q \in M_\real$ here is a vector normal to $Q$ such that the orientation of $TQ \oplus \real \cdot \nu_Q$ agrees with that of $M_\real$.
\end{itemize}
\end{definition}

\begin{definition}\label{scattering_diagram_def}
A {\em scattering diagram} $\mathscr{D}$ over $R = R_n$ is a finite set of walls $\left\{ ( m_\alpha, Q_\alpha, \Theta_\alpha) \right\}_{\alpha }$.
\end{definition}

\begin{notation}\label{not:support_of_diagram}
For a scattering diagram $\mathscr{D}$, its {\em support} is defined as
$\text{supp}(\mathscr{D}) := \bigcup_{\mathbf{w} \in \mathscr{D}} Q_{\mathbf{w}},$
and its {\em singular set} as
$\text{Sing}(\mathscr{D}) := \bigcup_{\mathbf{w} \in \mathscr{D}} \partial Q_{\mathbf{w}} \cup \bigcup_{\mathbf{w}_1\pitchfork \mathbf{w}_2} \left(Q_{\mathbf{w}_1} \cap Q_{\mathbf{w}_2}\right),$
where $\mathbf{w}_1 \pitchfork \mathbf{w}_2$ means transversally intersecting walls.
\end{notation}

\subsubsection{Path ordered products}\label{analytic_continuation}

An embedded path
$\gamma: [0,1] \rightarrow B_0 \setminus \text{Sing}(\mathscr{D})$
is said to be {\em intersecting $\mathscr{D}$ generically}
if $\gamma(0), \gamma(1) \notin \text{supp}(\mathscr{D})$, $\text{Im}(\gamma) \cap \text{Sing}(\mathscr{D}) = \emptyset$ and it intersects all the walls in $\mathscr{D}$ transversally.
Given such an embedded path $\gamma$ with a sequence of real numbers
$0=t_0 < t_1< t_2<\cdots<t_s<t_{s+1} = 1$
such that $\left\{ \gamma(t_1), \ldots, \gamma(t_s) \right\} = \gamma \cap \text{supp}(\mathscr{D})$, we define the {\em path ordered product} along $\gamma$, denoted by
$$\Theta_{\gamma, \mathscr{D}} := \Theta_{\gamma(t_s)}^{\sigma_s}\cdots \Theta_{\gamma(t_i)}^{\sigma_i} \cdots \Theta_{\gamma(t_1)}^{\sigma_1}$$
to be the product of the wall crossing factors $\Theta_{\gamma(t_j)}$'s according to the direction of the path $\gamma$ following \cite{gross2010tropical}, where $\sigma_j = 1$ if if orientation of $P_{i,j} \oplus \real \cdot \gamma'(t_i)$ agree with that of $M_\real$ and $\sigma_j = -1$ otherwise. For more details, we refer readers to \cite[Section 3.2.1.]{kwchan-leung-ma}.

\begin{definition}\label{consistent_def}
A scattering diagram $\mathscr{D}$ is said to be {\em consistent} if we have
$\Theta_{\gamma,\mathscr{D}} = \text{Id},$
for any embedded loop $\gamma$ intersecting $\mathscr{D}$ generically. Two scattering diagrams $\mathscr{D}$ and $\tilde{\mathscr{D}}$ are said to be {\em equivalent} if
$\Theta_{\gamma,\mathscr{D}} = \Theta_{\gamma,\tilde{\mathscr{D}}}$
for any embedded path $\gamma$ intersecting both $\mathscr{D}$ and $\tilde{\mathscr{D}}$ generically.
\end{definition}

\subsubsection{Maslov index $0$ tropical disks}\label{sec:scattering_diagram}

\begin{definition}\label{def:MI_zero_disk_diagram}
We define $\mathscr{D}(\mathcal{P},\Sigma,P_1,\dots,P_n)$ to be the scattering diagram which consists of walls
$$\mathbf{w}_\wptr = (m_\wptr, Q_\wptr, \Theta_\wptr)$$
for each weighted $d$-pointed $k$-tree $\wptr$ with $MI(\wptr) = 0$ and $\mathfrak{M}_d^{\wptr} (\mathcal{P},\Sigma,P_1,\dots,P_n)/\real_+ \neq \emptyset$, where
\begin{enumerate}
\item
the ray $Q_\wptr \subset M_\real$ is given by the closure of the image of $ev_o: \mathfrak{M}_d^{\wptr} (\mathcal{P},\Sigma,P_1,\dots,P_n) \rightarrow M_{\real}$ at the outgoing vertex $v_o$ (i.e. the locus of the stop of a tropical disk $\varsigma$),
\item
the Fourier mode $m_\wptr = m_\varsigma$ is the weight associated to the outgoing edge $e_o$ attached to the unique outgoing vertex $v_o$, and
\item
the wall-crossing automorphism $\Theta_\wptr$ is given by the formula
$\text{Log}(\Theta_\wptr) = k_{\wptr} \text{Mult}(\wptr) z^{m_{\wptr}} \check{\partial}_{n_{\wptr}} u_\wptr,$
where $k_{\wptr}$ is introduced in Notation \ref{not:integer_weight}, $u_{\wptr}$ is defined as in Definition \ref{def:mono_LG_potential}, and $n_{\wptr} \in N$ is the clockwise primitive normal to $Q_\wptr$.
\end{enumerate}
\end{definition}

We end this section by stating two of the main results in \cite{Gross10} which describe how the perturbed LG potential $W_n(Q)$ jumps across the walls in the scattering diagram $\mathscr{D}(\mathcal{P},\Sigma,P_1,\dots,P_n)$:

\begin{theorem}[Proposition 4.7 and Theorem 4.12 in \cite{Gross10}]\label{thm:superpotential_wall_crossing}
For any point $\mathfrak{j} \in \text{Sing}(\mathscr{D}) \setminus \{P_1,\dots,P_n\}$ and any loop $\gamma_\mathfrak{j}$ around $\mathfrak{j}$ in a sufficiently small contractible neighborhood $U_\mathfrak{j}$ of $\mathfrak{j}$, we have
$$\Theta_{\gamma_\mathfrak{j},\mathscr{D}}  = \text{Id}.$$
Furthermore, if $Q,Q' \in M_\real$ are not lying on any walls in $\mathscr{D}(\mathcal{P},\Sigma,P_1,\dots,P_n)$, then we have
$$
W_n(Q') = \Theta_{\gamma,\mathscr{D}}(W_n(Q)),
$$
for any path $\gamma \subset M_\real \setminus \text{Sing}(\mathscr{D})$ joining $Q$ to $Q'$.
\end{theorem}
As we have seen in the introduction, the main result of this paper (i.e. Theorem \ref{thm:main_theorem_intro_1}) combined with the main results of \cite{kwchan-leung-ma} can give new interpretations and alternative proofs of these two results.

\section{Extended deformation theory of the LG model}\label{sec:MC_equation}

In this section, we investigate the dgBV algebra governing the extended deformation theory of the LG model $(\check{\mathcal{X}},W)$ and the asymptotic behavior of the Maurer-Cartan solutions when the torus fibers of the fibration $\check{p}: \check{X} \rightarrow \text{Int}(\mathbf{P})$ shrink, building on the techniques developed in \cite{kwchan-leung-ma}.

\subsection{The dgBV algebra coming from polyvector fields}\label{sec:BV_algebra}

Given a LG model $(\check{X}, W)$ equipped with a holomorphic volume form $\check{\Omega}$, one can construct a natural differential graded Batalin-Vilkovisky (dgBV) algebra on the Dolbeault resolution of the sheaf of polyvector fields on $\check{X}$ given by $PV^{i,j}(\check{X}) := \Omega^{0,j}(\check{X},\wedge^iT^{1,0}_{\check{X}}),$
where the degree on $PV^{i,j}(\check{X})$ is taken to be $j-i$. We briefly review this construction; see, e.g. \cite{li2013primitive}.

\begin{notation}
Given local holomorphic coordinates $u^1,\dots,u^n$ on $\check{X}$ and an ordered subset $I = \{i_1,\dots,i_k\} \subset \{1,\dots,n\}$, we set
$d\bar{u}^I := d\bar{u}^{i_1}\wedge\cdots\wedge d\bar{u}^{i_k},\ \partial_I := \dd{u^{i_1}} \wedge \cdots \wedge \dd{u^{i_k}},$
and similarly for $du^I$ and $\bar{\partial}_I$.
\end{notation}


First of all, the space of smooth sections of $\bigwedge^* T^{1,0}_{\check{X}}$ is equipped with a natural wedge product $\wedge$. With a holomorphic volume form $\check{\Omega} =  e^{f} du^1 \cdots du^n$ and a polyvector field of the form $\partial_I$ where $I = \{i_1,\dots,i_k\}$, we define
$\partial_I \dashv \check{\Omega} := \iota_{\dd{u^{i_1}}}\cdots \iota_{\dd{u^{i_k}}} \check{\Omega}.$

\begin{definition}\label{BVdifferential}
The {\em BV differential} $\bvd_{\check{\Omega}}$ is defined by\footnote{We will suppress the dependence of the BV differential on $\check{\Omega}$ whenever there is no danger of confusion.}
\begin{equation}
\bvd_{\check{\Omega}} \alpha \dashv\check{\Omega} : = \partial(\alpha \dashv \check{\Omega}).
\end{equation}
\end{definition}

The operation $\delta_v : \bigwedge^* T^{1,0} \rightarrow \bigwedge^{*-1} T^{1,0}$ defined by
\begin{equation}\label{BVderivation}
\delta_v(w) : = \bvd(v\wedge w ) - \bvd(v) \wedge w -(-1)^{k} v\wedge \bvd(w)
\end{equation}
is a derivation of degree $k+1$.

\begin{definition}
We define the {\em bracket} $[\cdot,\cdot] : V \otimes V \rightarrow V$ by $[v,w] = (-1)^{|v|+1}\delta_v(w),$ where $|v|$ stands for the degree of a homogeneous element $v$.\footnote{The bracket $[\cdot,\cdot]$ agrees with the well-known {\em Schouten-Nijenhuis Lie bracket} on smooth sections of $\bigwedge^* T^{1,0}$ which can be expressed as
$[v_1\wedge \dots \wedge v_{k}, \mathsf{v}_1\wedge \dots \wedge \mathsf{v}_{k'}] = \sum_{\substack{1 \leq i \leq k \\ 1 \leq j \leq k'}} (-1)^{i+j} [v_i, \mathsf{v}_j]\wedge  v_1 \wedge \dots\wedge  \widehat{v}_i\wedge  \dots \wedge v_k \wedge \dots \widehat{\mathsf{v}}_j \wedge \dots \wedge \mathsf{v}_{k'}$ and
$[v_1\wedge \dots \wedge v_k , f ]  = \sum_{i} (-1)^{k-i} v_i(f) v_1 \wedge \dots \wedge \widehat{v_i} \wedge \dots v_k$.}
\end{definition}

These structures can be extended to the Dolbeault resolution $PV^{*,*}(\check{X})$ of $\bigwedge^* T^{1,0}$ equipped with the twisted differential $\pdb_W = \pdb + [W,\cdot]$ and the graded commutative wedge product $\wedge$. In the local holomorphic coordinates $u^1, \dots, u^n$, writing $\alpha = \alpha^I_J d\bar{u}^J \wedge \partial_{I}$ (with $|I| =i$ and $|J|=j$) and $\beta = \beta^K_L d\bar{u}^L \wedge \partial_{K}$ (with $|K| = k$ and $|L| = l$), we have
\begin{align*}
\pdb(\alpha)  = \pdb(\alpha^I_J)d\bar{u}^J \wedge \partial_{I}; \quad
&\bvd (\alpha)  = (-1)^{j} d\bar{u}^J \wedge \bvd (\alpha^I_J\partial_{I});\\
\alpha \wedge \beta  = (-1)^{il} \alpha^I_J \beta^K_L d\bar{u}^J \wedge d\bar{u}^L \wedge \partial_{I} \wedge\partial_K; \quad
&[\alpha,\beta]  = (-1)^{(i+1)l} d\bar{u}^Jd\bar{u}^L [\alpha^I_J \partial_{I},\beta^K_L\partial_{K}].
\end{align*}

From these we obtain the differential graded Lie algebra (dgLa) $(PV^{*,*}[1],[\cdot,\cdot],\pdb_W)$, where $[1]$ is a degree shift. We will study the asymptotic behavior of solutions of the Maurer-Cartan equation \eqref{eqn:extended_MC_equation_intro} for degree $0$ elements $\varphi$ in $PV^{*,*}\otimes_\comp \widehat{\comp[Q]}_{\mathcal{S}} \otimes_\comp \mathbf{m}_n$. 

Going back to our situation, by extending the exact sequence of monoids \eqref{monoid_sequence} to the associated abelian groups, we get the so-called {\em fan sequence} in toric geometry \cite{CLS_toric_book, Fulton_toric_book}:
\begin{equation}\label{eqn:group_sequence}
\xymatrix{0 \ar[r] & \mathcal{Q}^{gp}  \ar[r]  & \mathcal{P}^{gp}  \ar[r]^{\theta} & M  \ar[r] & 0}.
\end{equation}
We have $\mathcal{P}^{gp} \cong \mathbb{Z}^{|\Sigma(1)|} \cong M \times \mathcal{Q}^{gp}$, giving a trivialization of the family over $\text{Spec}(\comp[\mathcal{Q}^{gp}]) \subset \text{Spec}(\comp[\mathcal{Q}])$ as
\begin{equation}
\check{\mathcal{X}} \times_{\text{Spec}(\comp[\mathcal{Q}])} \text{Spec}(\comp[\mathcal{Q}^{gp}]) \cong T_N \times \text{Spec}(\comp[\mathcal{Q}^{gp}]),
\end{equation}
where $T_N := \left(N \otimes_\inte \comp\right) / N$ is a 2-dimensional algebraic torus.

Since $\mathcal{Q}$ is a strictly convex polyhedral cone, there is a natural maximal ideal $ \mathbf{m} = \mathbf{m}_\mathcal{Q} := \langle z^{m} \mid m \in \mathcal{Q} \setminus \{0\} \rangle$ in $\comp[\mathcal{Q}]$. We consider the completion $\widehat{\comp[\mathcal{Q}]} : = \varprojlim_{k} \comp[\mathcal{Q}]/ \mathbf{m}^{k}$ and its localization $\widehat{\comp[\mathcal{Q}]}_{\mathcal{S}}$ at the multiplicative system $\mathcal{S} = \{z^{m} \mid m \in \mathcal{Q} \setminus \{0\}\}$. By taking the tensor product $\comp[\mathcal{P}^{gp}] \otimes_{\comp[\mathcal{Q}^{gp}]}  \widehat{\comp[\mathcal{Q}]}_{\mathcal{S}} = \comp[M] \otimes_{\comp} \widehat{\comp[\mathcal{Q}]}_{\mathcal{S}}$, we can treat $W \in \comp[M] \otimes_{\comp} \widehat{\comp[\mathcal{Q}]}_{\mathcal{S}}$ as a family of LG potentials parametrized by $\widehat{\comp[\mathcal{Q}]}_{\mathcal{S}}$ on the (fixed) algebraic torus $T_N$. For the LG model $(\check{X}, W) := (T_N, W)$, we choose the local holomorphic coordinates as follows:
\begin{notation}\label{not:holomorphic_coordinates}
We fix, once and for all, a $\inte$-basis $e_1, e_2$ for $M$ and identify $\mathsf{m} = \mathsf{m}_1 e_1 + \mathsf{m}_2 e_2$ with $(\mathsf{m}_1,\mathsf{m}_2) \in \inte^2$. We also use $\bmc^{\mathsf{m}} = (\bmc^1)^{\mathsf{m}_1} (\bmc^2)^{\mathsf{m}_2}$, for $\mathsf{m} = (\mathsf{m}_1,\mathsf{m}_2) \in M$, to denote a monomial on $\check{X}$. Notice that every $\mathsf{m} \in M$ naturally gives a $(1,0)$-form $d \log (\mathsf{m}):= d \log (\bmc^{\mathsf{m}})$; similarly, every $n \in N$ naturally gives a vector field $\check{\partial}_n$ satisfying
$\check{\partial}_n (\bmc^\mathsf{m}) = (n , \mathsf{m}) \bmc^{\mathsf{m}},$
where $( \cdot, \cdot )$ is the natural pairing between $M$ and $N$.
\end{notation}
Equipped with the natural holomorphic volume form $\check{\Omega}: = d \log \bmc^1 \wedge d\log \bmc^2$ on $\check{X}$, we obtain the triple $(\check{X}, W, \check{\Omega})$, and hence a dgBV algebra by the above discussion.

\subsubsection{$\hp$-family of SYZ fibrations}\label{sec:semi_flat_family}

Following a proposal by Kontsevich-Soibelman \cite{Kontsevich-Soibelman01} and Fukaya \cite{fukaya05}, we consider an $\hp$-family of SYZ fibrations which corresponds to a large complex structure limit, so that we can apply asymptotic analysis as in the previous work \cite{kwchan-leung-ma}.

We consider the log map $\text{Log} : T_N \cong (N_\comp / N) \rightarrow \sqrt{-1}N_\real$ which is naturally a torus fibration.
We fix a symplectic structure $\omega_0$ on the toric surface $X_\Sigma$ and consider the associated moment polytope $\mathbf{P} \subset \sqrt{-1}N_\real$.
From the SYZ viewpoint \cite{syz96, Chan-Leung10a}, the mirror manifold is obtained by dualizing the moment map on $X_\Sigma$, so we choose the base of the SYZ fibration to be $\check{B}_0 := \text{Int}(\mathbf{P})$ and take $\check{X} =  \check{p}^{-1}(\check{B}_0)$ instead of the whole algebraic torus $T_N$.

Let $\{e^1,e^2\}$ be the $\mathbb{Z}$-basis of $N$ dual to the chosen basis $\{e_1, e_2\}$ of $M$.
We then let $(x^1,x^2)$ be the oriented affine coordinates of $\check{B}_0$ with respect to the basis $\{e^1,e^2\}$ and $(y^1,y^2)$ be the affine coordinates on the torus fibers of $\check{p}:\check{X} \to \check{B}_0$.

Associated to the symplectic structure $\omega_0$, there is a symplectic potential $\check{\phi}$ in the action-angle coordinates written explicitly in \cite{GU94}. We take $\check{\phi}$ and apply the Legendre transform
$\check{L}_{\check{\phi}}: \text{Int}(\mathbf{P}) \rightarrow M_\real$
to obtain the dual integral affine manifold $B_0$ equipped with affine coordinates
$x_1 := \ddd{\check{\phi}}{x^1}, \ x_2 := \ddd{\check{\phi}}{x^2}.$
We prefer to work with the affine manifold $B_0$ because then we can deal with tropical trees instead of Morse trees, as explained in \cite{gross2011real} (see also \cite[Section 2]{kwchan-leung-ma}).

We then introduce a small $\hp > 0$ parameter to rescale the affine coordinates on $\check{B}_0$ as $x^j \mapsto \hp^{-1} x^j$, and obtain the ($\hp$-dependent) holomorphic coordinates $\bmc^j = \exp(- 2\pi i (y^j + i \hp^{-1} x^j ))$ (cf. \cite[Section 2]{kwchan-leung-ma}). Under these $\hp$-twisted coordinates, the holomorphic vector field $\check{\partial}_{j}$ is explicitly given by
\begin{equation}\label{holo_vector_field_2}
n = (n^j) \mapsto \check{\partial_n}  := \sum_j n^j \check{\partial}_{j} = \sum_j n^j \dd{\log \bmc^j} = \frac{i}{4\pi} \sum_j n^j \left( \dd{y^j}- i\hp \sum_k   g_{jk} \dd{x_k} \right).
\end{equation}
The corresponding $\hp$-dependent dgLa of polyvector fields will be denoted by $PV^{*,*}_\hp$. We will consider differential forms on $B_0$ depending on $\hp$ and hence we introduce the following:
\begin{notation}\label{hp_diff_forms}
We use $\Omega^*_\hp(B_0)$ (similarly for $\Omega^*_\hp(U)$ for any open subset $U \subset B_0$) to denote the space of smooth sections of $\bigwedge^*T^*B_0$ over $B_0 \times \real_{>0}$, where the extra $\real_{>0}$ direction is parametrized by $\hp$.
\end{notation}

\subsubsection{Fourier expansions of polyvector fields}
Recall that the Fourier transform $\hat{\mathcal{F}}$
\begin{equation}\label{eqn:Fourier_transform}
\hat{\mathcal{F}} :\mathbf{G}^{*,*}_n := \left( \bigoplus_{m \in \mathcal{P}}\Omega^*_{\hp}(B_0) z^{m}\right) \otimes_{\inte} \wedge^*N \otimes_\comp R_n \hookrightarrow PV^{*,*}_\hp \otimes_\comp \widehat{\comp[\mathcal{Q}]}_{\mathcal{S}} \otimes_\comp R_n,
\end{equation}
introduced in \cite[Section 2]{kwchan-leung-ma}, gives an inclusion of dg Lie subalgebras by\footnote{It is an inclusion since we restrict ourselves to Fourier modes $m$ in $\mathcal{P}$ and we take finite sums instead of infinite Fourier series.}
\begin{enumerate}
\item
identifying the Fourier modes $z^{m} \in \comp[\mathcal{P}] \hookrightarrow \comp[\mathcal{P}^{gp}]$ as $z^{m} = \bmc^{\bar{m}} \otimes q^{\hat{m}} \in \comp[M]\otimes_\comp \widehat{\comp[\mathcal{Q}]}_\mathcal{S}$ through the isomorphism $ \comp[\mathcal{P}^{gp}] \cong \comp[M]\otimes_\comp \widehat{\comp[\mathcal{Q}]}_\mathcal{S}$,

\item
pulling back smooth functions $f(x,\hp)$ on $B_0$ to $\check{X}$ via the equation $\hat{\mathcal{F}}(f(x,\hp)) = \check{p}^{-1}(f(x,\hp))$ and using the torus fibration $\check{p}:\check{X} \rightarrow B_0$,

\item
identifying the 1-form $dx^j = \sum_{k} \pdpd{\phi}{x_j}{x_k} dx_k$ (where $\phi: B_0 \rightarrow \real$ is the Legendre dual to $\check{\phi}$) on $B_0$ with the $(0,1)$-form $\hat{\mathcal{F}}(dx^j)  = \frac{\hp}{4\pi} d\log \bar{\bmc}^j$ on $\check{X}$ for $j = 1, 2$,

\item
identifying $n\in N$ with the holomorphic vector field $\check{\partial}_n$ on $\check{X}$ by \eqref{holo_vector_field_2}, and

\item
extending the map skew-symmetrically.
\end{enumerate}

The Dolbeaut differential $\bar{\partial}$ is identified with the deRham differential $d$ acting on each summand $\Omega^*_{\hp}(B_0)$ via $\hat{\mathcal{F}}$. The action of a vector field $\check{\partial}_n = (n^1,n^2)$ on $f(x,\hp)$ by differentiation is identified as
\begin{equation}\label{eqn:holo_vector_field_action}
\check{\partial}_n (f) =  \frac{\hp}{4\pi} \sum_{j, k} n^j \pdpd{\check{\phi}}{x^j}{x^k} \dd{x_k} (f)
\end{equation}
via $\hat{\mathcal{F}}$ (recall that $x^1, x^2$ are affine coordinates on $\check{B}_0 \cong \text{Int}(\mathbf{P})$ while $x_1, x_2$ are affine coordinates on $B_0 = M_\real$).

\subsection{Differential forms with asymptotic support}\label{sec:asymptotic_support}

We will work with a dgLa constructed as a suitable quotient of a subalgebra of $\mathbf{G}^{*,*}_n$ (defined above in \eqref{eqn:Fourier_transform}), which turns out to be directly related to the tropical counting defined in Section \ref{sec:counting}. In order to do so, we need to recall the notion of {\em asymptotic support on a closed codimension $k$ tropical polyhedral subset $P\subset U$} for some {\em convex} $U\subset B_0$ which describes the behavior of differential forms $\alpha \in \Omega^*_\hp(B_0)$ as $\hp \rightarrow 0$ and also some of its basic properties from \cite{kwchan-leung-ma}.

First of all, by a {\em tropical polyhedral subset} in $U$ we mean a connected convex subset which is defined by {\em finitely many} affine linear equations or inequalities over $\mathbb{Q}$. For the purpose of proving the main theorem in this paper, we only need the cases when $P$ is either a point (whence $\text{dim}_\real (P) = 0$) or a ray/line (whence $\text{dim}_\real(P) = 1$) or a polyhedral domain (whence $\text{dim}_\real(P) = 2$). However, since the new properties established in this subsection should be of independent interest and useful in a broader context, we will work with a convex open subset $U\subset B_0$ in a general (oriented) affine manifold $B_0$ in arbitrary dimensions.

\begin{definition}\label{exponential_decay}
We define $\mathcal{W}^{-\infty}_k(U) \subset \Omega^k_\hp(U)$ to be the set of differential $k$-forms $\alpha \in \Omega^k_\hp(U)$ such that for each point $q \in U$, there exists a neighborhood $V$ of $q$ where we have
$\|\nabla^j \alpha\|_{L^\infty(V)} \leq D_{j,V} e^{-c_V/\hp}$
for some constants $c_V$ and $D_{j,V}$. The association $U \mapsto \mathcal{W}^{-\infty}_k(U)$ defines a sheaf over $B_0$ which we denote by $\mathcal{W}^{-\infty}_k$.
\end{definition}

We also need differential forms which only blow up at polynomial orders in $\hp^{-1}$:
\begin{definition}\label{polynomial_growth}
We define $\mathcal{W}^{\infty}_k(U) \subset \Omega^k_\hp(U)$ to be the set of differential $k$-forms $\alpha \in \Omega^k_\hp(U)$ such that for each point $q \in U$, there exists a neighborhood $V$ of $q$ where we have
$\|\nabla^j \alpha\|_{L^\infty(V)} \leq D_{j,V} \hp^{-N_{j,V}}$
for some constants $D_{j,V}$ and $N_{j,V} \in \inte_{>0}$. The association $U \mapsto \mathcal{W}^{\infty}_k(U)$ defines a sheaf over $B_0$ which we denote by $\mathcal{W}^{\infty}_k$.
\end{definition}

Notice that the sheaves $\mathcal{W}^{\pm \infty}_k$ in Definitions \ref{exponential_decay} and \ref{polynomial_growth} are closed under the actions of $\nabla_{\dd{x}}$, the deRham differential $d$ and the wedge product of differential forms. We also observe the fact that $\mathcal{W}^{-\infty}_k$ is a differential graded ideal of $\mathcal{W}^{\infty}_k$. In particular, we can consider the sheaf of differential graded algebras $\mathcal{W}^{\infty}_*/\mathcal{W}^{-\infty}_*$, equipped with the deRham differential.

\begin{definition}\label{def:asypmtotic_support_pre}
A differential $k$-form $\alpha \in \mathcal{W}^{\infty}_k(U)$ is said to have {\em asymptotic support on a closed codimension $k$ tropical polyhedral subset $P \subset U$ with weight $s$}, denoted by $\alpha \in \mathcal{W}_{P}^s$ if the following conditions are satisfied:
\begin{enumerate}
\item
For any $p \in U \setminus P$, there is a neighborhood $V \subset U \setminus P$ of $p$ such that $\alpha|_V \in \mathcal{W}^{-\infty}_k(V)$ on V.
		
\item
There exists a neighborhood $W_P$ of $P$ in $U$ such that we can write
$\alpha =  h(x,\hp) \nu_P + \eta,$
where $\nu_P \in \bigwedge^k N_{\real}$ is the unique affine $k$-form which is normal to $P$, $h(x,\hp) \in C^\infty(W_P \times \real_{>0})$ and $\eta$ is an error term satisfying $\eta \in \mathcal{W}^{-\infty}_k(W_P)$ on $W_P$.
		
\item
For any $p \in P$, there exists a sufficiently small convex neighborhood $V$ containing $p$ equipped with an affine coordinate system $x = (x_1,\dots x_n)$ such that $x' := (x_1, \dots, x_k)$ parametrizes codimension $k$ affine linear subspaces of $V$ parallel to $P$, with $x' = 0$ corresponding to the subspace containing $P$. Within the foliation $\{(P_{V, x'})\}_{x' \in N_V}$ where $P_{V,x'} = \{ (x_1,\dots,x_n) \in V  \ | \ (x_1,\cdots,x_k) = x' \}$ of $V$, we require that, for all $j \in \inte_{\geq 0}$ and multi-index $\beta = (\beta_1,\dots,\beta_k) \in \inte_{\geq 0}^k$, the estimate
\begin{equation}\label{estimate_order_s}
\int_{x'}  (x')^\beta \left(\sup_{P_{V,x'}}|\nabla^j (\iota_{\nu_P^\vee} \alpha)| \right) \nu_P \leq D_{j,V,\beta} \hp^{-\frac{j+s-|\beta|-k}{2}},
\end{equation}
for some constant $D_{j,V,\beta}$ and some $s \in \inte$, where $|\beta| = \sum_l \beta_l$ is the vanishing order of the monomial $(x')^{\beta}= x_1^{\beta_1}\cdots x_k^{\beta_k}$ along $P_{x'=0}$ and $\nu_P^\vee = \dd{x_1}\wedge\cdots \wedge\dd{x_k}$ in this local coordinate.
\end{enumerate}
\end{definition}

\begin{remark}
Note that condition (3) in Definition \ref{def:asypmtotic_support_pre} is independent of the choices of the convex neighborhood $V$, the transversal slice $N_V$ and the local affine coordinates $x = (x_1,\dots x_n)$ (although the constant $D_{j,V,\beta}$ may depend on these choices). Therefore this condition can be checked simply by choosing a sufficiently nice neighborhood $V$ at every point $p \in P$.
\end{remark}

By definition, we have the nice property that
\begin{equation}\label{asy_support_basic_property}
(x')^{\beta} \nabla_{\dd{x_{l_1}}}\cdots \nabla_{\dd{x_{l_j}}} \mathcal{W}^s_P(U) \subset \mathcal{W}^{s+j-|\beta|}_P(U)
\end{equation}
for any affine monomial $(x')^{\beta}$ with vanishing order $|\beta|$ along $P$.

The weight $s$ in Definition \ref{def:asypmtotic_support_pre} defines the following filtration (the $U$ dependence will be dropped whenever it is clear from the context):\footnote{Note that the degree $k$ of the differential forms has to be equal to the codimension of $P$. Also note that the sets $\mathcal{W}^{\pm \infty}_k(U)$ are independent of the choice of $P$.}
\begin{equation}\label{filtration}
\mathcal{W}^{-\infty}_k \cdots \subset \mathcal{W}^{-s}_P\subset \cdots \mathcal{W}^{-1}_P\subset \mathcal{W}^0_P \subset \mathcal{W}^1_P \subset \mathcal{W}^2_P\subset \cdots \subset \mathcal{W}^s_P \subset \cdots \subset \mathcal{W}^{\infty}_k \subset \Omega^k_\hp(U).
\end{equation}
This filtration keeps track of the polynomial orders of $\hp$ for differential $k$-forms with asymptotic support on $P$ and provides a convenient tool for expressing and proving results in asymptotic analysis.

\subsubsection{Behavior under $d$ and $\wedge$}
\begin{definition}\label{def:asymptotic_support}\label{def:asy_support_algebra}
A differential $k$-form $\alpha$ is said to be in $\tilde{\mathcal{W}}^{s}_k(U)$ if it there exist finitely many polyhedral subsets $P_1, \dots, P_l$ of codimension $k$ such that $\alpha \in \sum_{j=1}^l \mathcal{W}^s_{P_j}(U)$; if we further have $d \alpha \in \tilde{\mathcal{W}}^{s+1}_{k+1}(U)$, then we say $\alpha$ is in $\mathcal{W}^s_k(U)$. We also let $\mathcal{W}^s_*(U) := \bigoplus_k \mathcal{W}^{s+k}_k(U)$ for every $s \in \inte$.
\end{definition}

We have the following lemma on the compatibility between the filtration and the wedge product:
\begin{lemma}\label{lem:support_product}
For two closed tropical polyhedral subsets $P_1, P_2 \subset U$ of codimension $k_1, k_2$ respectively, we have $\mathcal{W}^s_{P_1}(U) \wedge \mathcal{W}^r_{P_2}(U) \subset \mathcal{W}^{r+s}_P(U)$ for any codimension $k_1+k_2$ polyhedral subset $P$ containing $P_1 \cap P_2$ normal to $\nu_{P_1} \wedge \nu_{P_2}$ if they intersect transversally (in particular if $\text{codim}_\real(P_1\cap P_2) = k_1+k_2$ we can take $P = P_1 \cap P_2$), and $\mathcal{W}^s_{P_1}(U) \wedge \mathcal{W}^r_{P_2}(U) \subset \mathcal{W}^{-\infty}_{k_1 + k_2}(U)$ if their intersection is not transversal.
Furthermore, we have $\mathcal{W}^{s_1}_{k_1}(U) \wedge \mathcal{W}^{s_2}_{k_2}(U) \subset \mathcal{W}^{s_1+s_2}_{k_1+k_2}(U)$. Hence $\mathcal{W}^0_*(U) \subset \mathcal{W}^{\infty}_*(U)$ is a dg subalgebra and $\mathcal{W}^{-1}_*(U) \subset \mathcal{W}^0_*(U)$ is a dg ideal of $\mathcal{W}^0_*(U)$, under the operations $d$ and $\wedge$.
\end{lemma}

Before giving the proof, let us clarify that when we say two closed tropical polyhedral subsets $P_1, P_2 \subset U$ of codimension $k_1, k_2$ are {\em intersecting transversally}, we mean the affine subspaces containing $P_1, P_2$ and of codimension $k_1, k_2$ respectively are intersecting transversally; this applies even to the case when $\partial P_i \neq \emptyset$.

\begin{proof}[Proof of Lemma~\ref{lem:support_product}]
The first statement is nothing but \cite[Lemma 4.22]{kwchan-leung-ma}, which in turn implies the second statement as follows: Given polyhedral subsets $P_1$ and $P_2$ of codimensions $k_1$ and $k_2$ respectively, notice that we always have some polyhedral subset $P$ of codimension $k = k_1 + k_2$ such that $\mathcal{W}^{s_1}_{P_1}(U) \wedge \mathcal{W}^{s_2}_{P_2}(U) \subset \mathcal{W}^{s_1+s_2}_P(U)$. Therefore, we conclude that $\tilde{\mathcal{W}}^{s_1}_{k_1}(U) \wedge \tilde{\mathcal{W}}^{s_2}_{k_2}(U) \subset \tilde{\mathcal{W}}^{s_1+s_2}_{k_1+k_2}(U)$. Now, suppose $\alpha_i \in \mathcal{W}^{s_i}_{k_i}(U)$. Then we have $d\alpha_i \in \tilde{\mathcal{W}}^{s_i+1}_{k_i+1}(U)$ and therefore $(d\alpha_1) \wedge \alpha_2 \in \tilde{\mathcal{W}}^{s_1+s_2+1}_{k_1+k_2+1}(U)$; similar statement holds for $\alpha \wedge (d\alpha_2)$.
Finally, the statements that $\mathcal{W}^0_*(U) \subset \mathcal{W}^{\infty}_*(U)$ is a dg subalgebra and $\mathcal{W}^{-1}_*(U) \subset \mathcal{W}^0_*(U)$ is a dg ideal follow from $\mathcal{W}^{s_1}_{k_1}(U) \wedge \mathcal{W}^{s_2}_{k_2}(U) \subset \mathcal{W}^{s_1+s_2}_{k_1+k_2}(U)$.
\end{proof}

\subsubsection{Behavior under integral operators}\label{sec:integral_lemma}
In this subsection, we study the behavior of $\mathcal{W}^s_P(U)$ under the action of an integral operator $I$, generalizing some of the results in \cite[Section 4.2]{kwchan-leung-ma}. For a given closed tropical polyhedral subset $P \subset U$, we choose a reference tropical hyperplane $R \subset U$ which divide the domain $U$ into $U\setminus R = U_+ \cup U_-$, together with an affine vector field $v$ (meaning $\nabla v = 0$) not tangent to $R$ pointing into $U_+$.

By shrinking $U$ if necessary, we assume that for any point $p \in U$, the unique flow line of $v$ in $U$ passing through $p$ intersects $R$ uniquely at a point $x \in R$.
Then the time-$t$ flow along $v$ defines a diffeomorphism
$\tau : W \rightarrow U,\ (t, x) \mapsto \tau(t,x),$
where $W \subset \real \times R$ is the maximal domain of definition of $\tau$ (namely, for any $x \in R$, there is a maximal time interval $I_x \subset \real$ so that the flow line through $x$ has its image lying inside $U$).
For any point $x \in R$, we denote by $\tau_x(t) := \tau(t,x)$ the flow line of $v$ passing through $x$. Figure \ref{fig:integral_flow} illustrates the situation.

\begin{figure}[h]
\centering
\includegraphics[scale=0.27]{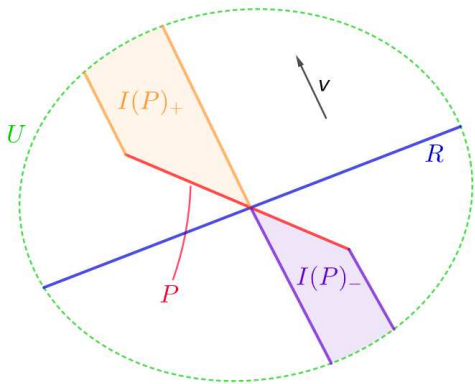}
\caption{The flow along $v$ and $I(P)$}
\label{fig:integral_flow}
\end{figure}

We let $P_\pm = P \cap \overline{U}_{\pm}$ and define
\begin{equation}\label{defining_IP}
I(P)_+ :=( P_+ + \real_{\geq 0} \cdot v ) \cap U; \quad
I(P)_-  : = (P_- + \real_{\leq 0 } \cdot v) \cap U,
\end{equation}
(see again Figure \ref{fig:integral_flow}). We also write $I(P) = I(P)_+ \cup I(P)_-$.
Now we define an integral operator $I$ by
\begin{equation}\label{general_integral_operator}
I(\alpha)(t,x) := \int_{0}^t \iota_{\dd{s}}(\tau^*(\alpha))(s,x) ds.
\end{equation}
Note that $I$ depends on the choice of the tropical hyperplane $R$. We have the following lemma, which is a modification of \cite[Lemma 4.23]{kwchan-leung-ma}:
\begin{lemma}[cf. Lemma 4.23 in \cite{kwchan-leung-ma}]\label{integral_lemma}
For $\alpha \in \mathcal{W}^s_P(U)$, we have $I(\alpha) \in \mathcal{W}^{-\infty}_{k-1}(U)$ if $v$ is tangent to $P$, and $I(\alpha) \in \mathcal{W}^{s-1}_{I(P)_+}(U) + \mathcal{W}^{s-1}_{I(P)_-}(U)$ if $v$ is not tangent to $P$, where $I(P)_\pm$ is defined in \eqref{defining_IP}. Moveover for $\alpha \in \tilde{\mathcal{W}}^s_k(U)$, we have $I(\alpha) \in \tilde{\mathcal{W}}^{s-1}_{k-1}(U)$.
\end{lemma}

\begin{proof}
We only describe the modifications needed in order to adapt the proof of \cite[Lemma 4.23]{kwchan-leung-ma}. We introduce a decomposition $\alpha = \alpha_+ + \alpha_-$ of $\alpha$, where the components $\alpha_+$ and $\alpha_-$ have asymptotic support of the same weight on $P_+$ and $P_-$ respectively, using cut-offs as follows. First we consider the functions depending only on the $t$-coordinate given by
$$
\chi_{+}(t)  := \left( \frac{1}{\hp \pi} \right)^{\half} \int_{-\infty}^{t}e^{-\frac{s^2}{\hp}} ds, \quad
\chi_{-}(t)  := 1 - \chi_{+}(t) = \left( \frac{1}{\hp \pi} \right)^{\half} \int_{t}^{\infty}e^{-\frac{s^2}{\hp}} ds;
$$
they have asymptotic support with weight $0$ on $U_+ = \{t \geq 0 \} \cap U$ and $U_- = \{t \leq 0\} \cap U$ respectively. Lemma \ref{lem:support_product} implies that the cut-offs
$\alpha_\pm:= \chi_{\pm} \alpha$
have asymptotic support with the same weight $s$ on $U_\pm \cap P$ respectively.
Therefore we may start by assuming $\alpha \in \mathcal{W}^{s}_P(U)$ with $P \subset \overline{U}_+$ and we simply write $I(P)$ to stand for $I(P)_+$. The rest of the proof is essentially the same as that of \cite[Lemma 4.23]{kwchan-leung-ma}.
\end{proof}

In order to understand the effect of $I$ on $\mathcal{W}^s_k(U)$, we need the following Lemmas \ref{lem:inclusion_lemma} and \ref{lem:projection_lemma} which describe the behavior of $\mathcal{W}^s_P(U)$ under pullbacks. Following the notations in Lemma \ref{integral_lemma}, we consider the tropical hypersurface $\mathbf{i}: R \subset U$ with an affine projection $\mathbf{p} : U \rightarrow R$ (which are explicitly given by the $\mathbf{i}(x) = (0,x)$ and $\mathbf{p}(t,x) = x$ using the affine coordinates given by $\tau$).

\begin{lemma}\label{lem:inclusion_lemma}
For $\alpha \in \mathcal{W}^s_P(U)$, we have $\mathbf{i}^*(\alpha) \in \mathcal{W}^s_{Q}(R)$ if $P$ intersects $R$ transversally
and $Q$ is any polyhedral subset of $R$ of codimension $k$ ($ = \text{codim}_\real(P \subset U)$) which contains $P\cap R$ and is normal to $\mathbf{i}^*(\nu_P)$, and $\mathbf{i}^*(\alpha) \in \mathcal{W}^{-\infty}_k(U)$ if $P$ does not intersect $R$ transversally. Moveover, the pull back gives a map $\mathbf{i}^* : \mathcal{W}^s_k(U) \rightarrow \mathcal{W}^s_k(R)$.
\end{lemma}
\begin{proof}
We begin by showing the corresponding statement for $\mathcal{W}^s_P(U)$. First, we verify condition (1) of Definition \ref{def:asypmtotic_support_pre}. Suppose that $p \in R \setminus P$, then we can find a neighborhood $V$ of $p$ in $U \setminus P$ such that $\alpha|_V \in \mathcal{W}^{-\infty}_k(V)$ from the assumption that $\alpha \in \mathcal{W}^{s}_P(U)$. Therefore $\alpha|_{V \cap R} \in \mathcal{W}^{-\infty}_k(V\cap R)$.
	
For condition (2) of Definition \ref{def:asypmtotic_support_pre}, we first assume that $P$ and $R$ are not intersecting transversally. We notice that there is a neighborhood $W_P$ of $P$ such that $\alpha$ can be written as $h(x,\hp) \nu_P + \eta$ in $W_P$ from the assumption that $\alpha \in \mathcal{W}^{s}_P(U)$. Therefore we have $\mathbf{i}^*(\nu_P) = 0$ if the intersection is not transversal, and so $\mathbf{i}^*(\alpha) \in \mathcal{W}^{-\infty}_k(R)$. Suppose $P$ and $R$ intersect transversally, then we can take $W_{P\cap R} := \mathbf{i}^{-1}(W_P)$, and we will have
$\mathbf{i}^{*}(\alpha)|_{W_{P\cap R}} = \mathbf{i}^*(h) \mathbf{i}^{*}(\nu_P) + \mathbf{i}^*(\eta)$
in $W_{P\cap R}$ with $\mathbf{i}^*(\nu_P)$ being the volume form of normal bundle of $\mathbf{i}^{-1}(P)$ as desired for condition (2). Notice that if $Q \neq R \cap P$ and for any point $p \notin R \cap P$, there is a neighborhood $V$ of $p$ such that $\alpha|_{V \cap R} \in \mathcal{W}^{-\infty}_k(V\cap R)$ from our earlier discussion, and therefore condition (2) still holds for arbitrary such $Q$.
	
For condition (3), we consider a point $p \in R \cap P$ with affine coordinates
$(x_1,\dots,x_k,x_{k+1},\dots,x_n) \in (-\delta,\delta)^n$
in $V \subset U$ such that $R \cap V = \{x_n = 0\}$, and $x' = (x_1,\dots,x_k)$ are parametrizing the parallel foliation $\{P_{V,x'}\}_{x'\in (-\delta,\delta)^k}$ to $P$ in $V$. Then $\{P_{V,x'} \cap R\}_{x' \in (-\delta,\delta)^k}$ is the foliation parallel to $P \cap R$ in $V \cap R$. Using the fact that $\sup_{P_{V,x'} \cap R}|\nabla^j(\iota_{\nu_{P}^\vee}\alpha)| \leq \sup_{P_{V,x'}}|\nabla^j(\iota_{\nu_{P}^\vee}\alpha)|$, we have
\begin{align*}
\int_{x'\in N_V} (x')^\beta \left(\sup_{P_{V,x'} \cap R}|\nabla^j (\iota_{\nu_P^\vee} \alpha)| \right) \nu_P
\leq \int_{x'\in N_V} (x')^\beta \left(\sup_{P_{V,x'}}|\nabla^j (\iota_{\nu_P^\vee} \alpha)| \right) \nu_P \leq D_{j,V,\beta} \hp^{-\frac{j+s-|\beta|-k}{2}},
\end{align*}
which is the desired estimate for condition (3) of Definition \ref{def:asypmtotic_support_pre}.
	
The statement that $\mathbf{i}^*$ is a map from $\mathcal{W}^s_k(U)$ to $\mathcal{W}^s_k(R)$ is a direct consequence of the first statement.
\end{proof}

\begin{lemma}\label{lem:projection_lemma}
For $\alpha \in \mathcal{W}^s_{P}(R)$, we have $\mathbf{p}^* (\alpha) \in \mathcal{W}^s_{\mathbf{p}^{-1}(P)}(U)$. Moreover, the pull back gives a map $\mathbf{p}^* : \mathcal{W}^s_k(R) \rightarrow \mathcal{W}^s_k(U)$.
\end{lemma}
\begin{proof}
For condition (1) of Definition \ref{def:asypmtotic_support_pre}, suppose we take $x \in U \setminus \mathbf{p}^{-1}(P)$, then we have an open subset $V \subset R \setminus P$ containing $\mathbf{p}(x)$. Therefore from the fact that $\alpha|_V \in \mathcal{W}^{-\infty}_k(V)$ (here $k = \text{codim}_\real(\mathbf{p}^{-1}(P))$) we get $\mathbf{p}^*(\alpha)|_{\mathbf{p}^{-1}(V)} \in \mathcal{W}^{-\infty}_k(\mathbf{p}^{-1}(V))$.
	
For condition (2) of Definition \ref{def:asypmtotic_support_pre}, we take a neighborhood $W_P$ of $P$ in $R$ such that we can write $\alpha$ as $h \nu_P + \eta$ with $\eta \in \mathcal{W}^{-\infty}_k(R)$ and $\nu_P$ is the normal of $P$ in $R$. We let $W_{\mathbf{p}^{-1}(P)} = \mathbf{p}^{-1}(W_P)$, and observe that $\mathbf{p}^*(\alpha) = \mathbf{p}^*(h) \mathbf{p}^*(\nu_P) + \mathbf{p}^*(\eta)$ with $\mathbf{p}^*(\nu_P)$ being normal of $\mathbf{p}^{-1}(P)$ in $U$ which is the desired decomposition.
	
For condition (3), we consider a point $p \in \mathbf{p}^{-1}(P)$ with affine coordinates
$(x_1,\dots,x_k,x_{k+1},\dots,x_{n-1}) \in (-\delta,\delta)^{n-1}$
around $q := \mathbf{p}(p)$ in $V \subset R$ such that $x' = (x_1,\dots,x_k)$ are parametrizing the foliation $\{P_{V,x'}\}_{x'\in (-\delta,\delta)^k}$ parallel to $P$ in $V$.
Therefore, we can extend the affine coordinates as $(x_1,\dots,x_k,x_{k+1},\dots,x_{n})$ of $\mathbf{p}^{-1}(V)$ such that
$\mathbf{p}(x_1,\dots,x_k,x_{k+1},\dots,x_{n}) = (x_1,\dots,x_k,x_{k+1},\dots,x_{n-1})$
in these coordinates. We notice that $\{\mathbf{p}^{-1}(P_{V,x'})\}_{x'\in (-\delta,\delta)^k}$ is the foliation parallel to $\mathbf{p}^{-1}(P)$ in $\mathbf{p}^{-1}(V)$ and we also have $\sup_{P_{V,x'} \cap R}|\nabla^j(\iota_{\nu_{P}^\vee}\alpha)| =  \sup_{\mathbf{p}^{-1}(P_{V,x'})}|\nabla^j(\iota_{\mathbf{p}^*(\nu_{P})^\vee} \mathbf{p}^*(\alpha))|.$
Therefore we conclude that
\begin{align*}
  &\int_{x'\in N_{\mathbf{p}^{-1}(V)}}  (x')^\beta \left(\sup_{\mathbf{p}^{-1}(P_{V,x'})}|\nabla^j (\iota_{\mathbf{p}^*(\nu_{P})^\vee} \mathbf{p}^*(\alpha))| \right) \mathbf{p}^*(\nu_P) \\
= & \int_{x'\in N_V}  (x')^\beta \left(\sup_{P_{V,x'}}|\nabla^j (\iota_{\nu_P^\vee} \alpha)| \right) \nu_P \leq D_{j,V,\beta} \hp^{-\frac{j+s-|\beta|-k}{2}},
\end{align*}
which is the desired estimate.
	
The statement that $\mathbf{p}^*$ is a map from $\mathcal{W}^s_k(R)$ to $\mathcal{W}^s_k(U)$ is a direct consequence of the first statement.
\end{proof}

\begin{lemma}\label{lem:integral_lemma_modified}
For $\alpha \in \mathcal{W}^s_k(U)$, we have $I(\alpha) \in \mathcal{W}^{s-1}_{k-1}(U)$.
\end{lemma}
\begin{proof}
Using the same notations as in Lemma \ref{integral_lemma}, note that the integral operator $I$ satisfies the equation $dI + Id = \text{Id} - \mathbf{p}^*\circ \mathbf{i}^*$. For a given $\alpha \in \mathcal{W}^s_k(U)$, we have $I(\alpha) \in \tilde{\mathcal{W}}^{s-1}_{k-1}(U)$ by Lemma \ref{integral_lemma}. Making use of Lemmas \ref{lem:inclusion_lemma} and \ref{lem:projection_lemma}, we have
$d(I(\alpha)) = -I(d(\alpha)) + \alpha - \mathbf{p}^*\circ \mathbf{i}^*(\alpha) \in \tilde{\mathcal{W}}^s_k(U),$
which implies $I(\alpha) \in \mathcal{W}^{s-1}_{k-1}(U)$.	
\end{proof}
	
Now we consider a chain of affine subspaces $\{q_0\}=U_0 \leq U_1 \cdots \leq U_n = U$ with $\text{dim}_\real(U_j) = j$, equipped with the natural inclusions $\mathbf{i}_j : U_j \rightarrow U_{j+1}$ and affine projections $\mathbf{p}_j : U_{j+1} \rightarrow U_j$ such that the fiber of $\mathbf{p}_j$ is tangent to a constant affine vector field $v_j$ on $U_{j+1}$. Composition of the inclusion operators gives $\mathbf{i}_{i,j} : U_i \rightarrow U_j$, and similarly for the projection operator $\mathbf{p}_{i,j}:U_j \rightarrow U_i$ for $i<j$. We let $I_j : \mathcal{W}^s_{k}(U_{j+1}) \rightarrow \mathcal{W}^{s-1}_{k-1}(U_{j+1})$ be the integral operator defined on $U_{j+1}$ using the vector field $v_j$ as in the beginning of this subsection (Section \ref{sec:integral_lemma}).
	
We choose $q_0$ to be an {\em irrational} point in $U_1$ (strictly speaking it is not a tropical polyhedral subset of $U_1$) for later applications in Section \ref{sec:solve_MC_equation}. The definitions of $\mathbf{p}_{0,j}^*$'s are still valid if they are treated as inclusions of constant functions. Despite the fact that $q_0$ is irrational, the operator $I_0$ defines a map $\mathcal{W}^s_k(U_1) \rightarrow \mathcal{W}^{s-1}_{k-1}(U_1)$ because every $\alpha \in \mathcal{W}^s_{1}(U_1)$ is a finite sum of $\sum_{l} \alpha_l$ with $\alpha_l \in \tilde{\mathcal{W}}^s_{P_l}(U_l)$ for some {\em rational} points $P_l$'s on $U_1$ which in particular miss $q_0$ and therefore $I_0(P_l)$ is still a tropical subspace of $U_1$.
	
We then define a {\em new} integral operator by
\begin{equation}\label{eqn:generalized_integral_operator}
I = \mathbf{p}_{1,n}^* I_{0}\mathbf{i}_{1,n}^* + \dots + \mathbf{p}_{n-1,n}^* I_{n-2}\mathbf{i}_{n-1,n}^* + I_{n-1},
\end{equation}
which is defined as $\mathcal{W}^s_*(U) \rightarrow \mathcal{W}^{s-1}_{*-1}(U)$, with the corresponding operator $\mathbf{i}^*:=\mathbf{i}_{0,n}^*$ being the evaluation at $q_0$ and the operator $\mathbf{p}^* : =  \mathbf{p}_{0,n}^*$.
	
\begin{prop}\label{prop:homotopy_operator_identity}
We have the identity
$d I + I d= \text{Id} - \mathbf{p}^* \circ \mathbf{i}^*,$
meaning that $I$ is contracting the cohomology of $U$ to that of the point $q_0$.
\end{prop}
\begin{proof}
We first notice that
$\mathbf{p}^*_{j+1,n}(dI_{j} + I_{j} d) \mathbf{i}^*_{j+1,n} = \mathbf{p}^*_{j+1,n}(id_{U_{j+1}} - \mathbf{p}^*_{j,j+1} \mathbf{i}^*_{j,j+1}) \mathbf{i}^*_{j+1,n},$
which gives $d (\mathbf{p}^*_{j+1,n}I_{j}\mathbf{i}^*_{j+1,n}) + (\mathbf{p}^*_{j+1,n}I_{j} \mathbf{i}^*_{j+1,n}) d = \mathbf{p}^*_{j+1,n}\mathbf{i}^*_{j+1,n} - \mathbf{p}^*_{j,n} \mathbf{i}^*_{j,n}.$
Taking summation over $j = 0,\dots,n-1$ gives the desired equation.
\end{proof}

\subsection{The tropical dgLa and its homotopy operator}\label{sec:homotopy_operator}
From now on, we restrict ourselves to the case that $B_0 = M_\real$ with $U \subset M_\real$.
\subsubsection{The tropical dgLa and the extended tropical vertex group}\label{sec:tropical_dgla}
As an analogue of the dgLa $\mathbf{G}^{*,*}_n$ introduced in \eqref{eqn:Fourier_transform}, we impose the requirement of the asymptotic behavior as $\hp \rightarrow 0$ and replace $\Omega^*_\hp(U)$ by the dg subalgebra $\mathcal{W}^0_*(U)$.
\begin{definition}\label{def:finite_Lie_alg}
For every convex open subset $U\subset B_0$, we define a dg Lie subalgebra of $PV^{*,*}_\hp|_U \otimes_\comp R_n$ by
$$
\mathcal{G}^{*,*}_n(U) :=\hat{\mathcal{F}}\Big\lbrack \left(\bigoplus_{m \in \mathcal{P}} \mathcal{W}^0_*(U) z^m\right) \otimes_\inte \wedge^*N \otimes_\comp R_n \Big\rbrack,
$$
making use of the Fourier transform \eqref{eqn:Fourier_transform}. Abusing notations, we will drop the identification via Fourier transform in \eqref{eqn:Fourier_transform} and simply write
$\mathcal{G}^{*,*}_n(U) = \left(\bigoplus_{m \in \mathcal{P}} \mathcal{W}^0_*(U) z^m\right) \otimes_\inte \bigwedge^*N \otimes_\comp R_n.$
Then we take the quotient by the dg Lie ideal
$\mathcal{I}^{*,*}_n(U) := \left(\bigoplus_{m \in \mathcal{P}} \mathcal{W}^{-1}_*(U) z^m\right) \otimes_\inte \bigwedge^*N \otimes_\comp R_n$
to obtain
\begin{equation*}
(\mathcal{G}/\mathcal{I})^{*,*}_n(U) := \left(\bigoplus_{m \in \mathcal{P}} \left(\mathcal{W}^{0}_*(U)/ \mathcal{W}^{-1}_*(U)\right) z^m\right) \otimes_\inte \wedge^* N \otimes_\comp R_n
\end{equation*}
which defines a dgLa (since $\mathcal{W}^{-1}_*(U)$ is a dg ideal of $\mathcal{W}^{0}_*(U)$).
\end{definition}

A general element of $\mathcal{G}^{i,j}_n$, $\mathcal{E}^{i,j}_n$ and $(\mathcal{G}/\mathcal{I})^{i,j}_n$ is a finite sum of the form
$$\sum_{I } \sum_{m, n_1,\dots,n_j} \alpha_{m,I}^{n_1,\dots,n_j} z^{m} \check{\partial}_{n_1} \wedge \cdots \wedge \check{\partial}_{n_j} u_I,$$
where $I \subset \{1,\dots,n\}$ and $u_I = \prod_{i \in I } u_i$, with $\alpha_{m,I}^{n_1,\dots,n_j} \in \mathcal{W}^{i}_i(U)$, $\alpha_{m,I}^{n_1,\dots,n_j}  \in \mathcal{W}^{i-1}_i(U)$ and $\alpha_{m,I}^{n_1,\dots,n_j}  \in \mathcal{W}^{i}_i(U) /\mathcal{W}^{i-1}_i(U)$ respectively. We will be concerned with the Maurer-Cartan equation \eqref{eqn:extended_MC_equation_intro} of the dgLa $(\mathcal{G}/\mathcal{I})^{*,*}_n(U)$ instead of $PV^{*,*}_\hp |_U\otimes_\comp \widehat{\comp[Q]}_\mathcal{S} \otimes_\comp R_n$, because we only care about the leading order behavior of the MC solutions as $\hp \rightarrow 0$.

Making use of the holomorphic volume form $\check{\Omega}$ on $\check{X}$, we obtain a BV operator $\bvd$ acting on $PV^{*,*}_\hp$ as in Section \ref{sec:BV_algebra}. The BV operator can be carried to $\mathcal{G}^{*,*}_n$ and naturally to $(\mathcal{G}/\mathcal{I})^{*,*}_n$, equipping them with dgBV structures. Explicitly, the BV operator is given by
$\bvd( \alpha z^{m} \check{\partial}_{n_1} \wedge \dots \wedge \check{\partial}_{n_j}) = \sum_{j} (-1)^{|\alpha|+r-1}\check{\partial}_{n_r}(\alpha z^m ) \check{\partial}_{n_1} \wedge \cdots \widehat{\check{\partial}_{n_{r}}} \cdots \wedge \check{\partial}_{n_j}$
in $\mathcal{G}^{*,*}_n$, which is further reduced to
$$\bvd( \alpha z^{m} \check{\partial}_{n_1} \wedge \dots \wedge \check{\partial}_{n_j})  = \sum_{j} (-1)^{|\alpha|+r-1} (n_r,\bar{m}) \alpha  z^m  \check{\partial}_{n_1} \wedge \cdots \widehat{\check{\partial}_{n_r}} \cdots \wedge \check{\partial}_{n_j}$$
in $(\mathcal{G}/\mathcal{I})^{*,*}_n$. This is because of the extra $\hp$ in the formula \eqref{eqn:holo_vector_field_action} giving $\check{\partial}_n (\alpha) \in \mathcal{I}^{*,*}_n$. As a consequence, the Lie bracket $[\cdot,\cdot]$ in $(\mathcal{G}/\mathcal{I})^{*,*}_n$ is given by
$[\alpha z^m \check{\partial}_{n_{I}},\beta z^{m'} \check{\partial}_{n_{J}}] = (-1)^{(|I|+1)|\beta|} \alpha \beta [z^m \check{\partial}_{n_{I}},z^{m'} \check{\partial}_{n_{J}}],$
where $\check{\partial}_{n_I} = \check{\partial}_{n_{i_1}} \wedge \cdots \wedge\check{\partial}_{n_{i_l}}$.

\begin{definition}\label{def:extended_tropical_lie_alg}
We call the dg Lie subalgebra $\mathcal{H}^{*,*}_n \leq (\mathcal{G}/\mathcal{I})^{*,*}_n$ defined by $\mathcal{H}^{*,*}_n := ker(\bvd)$, which is equipped with the differential $\bar{\partial}_W$ and Lie-bracket $[\cdot,\cdot]$, the {\em tropical dgLa}.
	We also call $\mathfrak{h}^{*}_n := \mathcal{H}^{*,0}_n \cap ker(\bar{\partial})$ the {\em extended tropical Lie-algebra}. The corresponding exponential group $\exp(\mathfrak{h}^{*}_n)$ is called the {\em extended tropical vertex group}.
\end{definition}

Explicitly, we have
$$
\mathfrak{h}^{0}_n  = \left( \bigoplus_{m \in \mathcal{P}} \comp \cdot z^m\right)  \otimes_\comp R_n, \quad
\mathfrak{h}^1_n  = \left(\bigoplus_{m \in \mathcal{P}} \comp\cdot z^m\right) \otimes_\inte m^{\perp} \otimes_\comp R_n,\quad
\mathfrak{h}^2_n  = \comp \cdot \check{\partial_1} \wedge \check{\partial}_2 \otimes_\comp R_n,
$$
and $\mathcal{H}^{*,*}_n$ can be viewed as the Dolbeault resolution of $\mathfrak{h}^*_n$. We will see that solving the Maurer-Cartan equation \eqref{eqn:extended_MC_equation_intro} in $\mathcal{H}^{*,*}_n$ is intimately and directly related to tropical counting.

\subsubsection{The homotopy operator}
In order to solve the Maurer-Cartan equation \eqref{eqn:extended_MC_equation_intro} using Kuranishi's method \cite{Kuranishi65}, we need a homotopy operator $H$ (also called a {\em propagator}) to fix the gauge. Here we explain the construction of such a homotopy operator $H$ using the operator $I$ defined in \eqref{eqn:generalized_integral_operator}. We will take $U = B_0 = M_\real$ and drop the dependence on $U$ in notations in the rest of this subsection.

\begin{notation}\label{not:perpendicular_hyperplane}
For each $m \in \mathcal{P}$ with the associated $\bar{m} \in M$, $\bar{m}$ naturally gives an affine vector field on $B_0 = M_\real$ which, by abuse of notations, will also be denoted as $\bar{m}$. We fix an affine linear metric $g_0$ on $M_\real$. 
Then, given any real number $R$, we choose a chain of affine subspaces $\{pt\} = U^m_{0} \leq U_{1}^m \leq U^m_{2} = M_\real$ as follows. First, we take $v^m_{1} = -\bar{m}$ if $\bar{m} \neq 0$ and take $v^m_{1}$ to be an arbitrary nonzero element in $M$ if $\bar{m} = 0$. Then we set $U^m_1 = \{x \mid g_0(v^m_{1},x) = -R \}$ and choose $U^m_{0}$ to be an irrational point on $U^{m}_1$.
Such a choice defines a homotopy operator $H_m : \mathcal{W}^{0}_*(U) \rightarrow \mathcal{W}^0_{*-1}(U)$ using the construction in \eqref{eqn:generalized_integral_operator} (which was denoted by $I$ there).
We also denote the half space $\{x \mid g_0(-\bar{m},x) \geq -R \}$ by $U^m_{1,+}$.

\end{notation}

\begin{definition}\label{MC_homotopy}
For each $m \in \mathcal{P}$, we define the homotopy operator
$H_m : \mathcal{W}^0_*(U) z^m \rightarrow \mathcal{W}^0_{*-1}(U) z^m$
on the direct summand for each Fourier mode $z^m$ by simply taking $H_m(\alpha z^m) := H_m(\alpha) z^m$. We also define the projection
$\mathsf{P}_m : \mathcal{W}^0_*(U) z^m \rightarrow \mathcal{W}^0_0(U^m_0) z^m$
by $\mathsf{P}_m (\alpha z^m) := (\alpha|_{x^0_m})z^{m}$ at degree $0$ and $0$ otherwise, where $\alpha|_{x^0_m}$ is evaluation of $\alpha$ at the point $\{x^0_m\} = U^m_0$, and the operator
$\iota_m :\mathcal{W}^0_0(U^m_0) z^m \rightarrow \mathcal{W}^0_*(U) z^m $
by $\iota_m (\alpha z^m) := \iota_m(\alpha) z^m$ at degree $0$ and $0$ otherwise, by setting $\iota_m: \mathcal{W}^0_0(U^m_0) \hookrightarrow \mathcal{W}^0_*(U)$ to be the embedding of constant functions over $M_\real$.
	
We abuse notations by treating $H_m$, $\mathsf{P}_m$ and $\iota_m$ as acting on the spaces $\mathcal{W}^0_*(U)$ and $\mathcal{W}^0_0(U^m_0)$ respectively.
\end{definition}

As in \cite{kwchan-leung-ma}, these operators satisfy the following identity of homotopy retracting $\mathcal{W}^0_*(U) z^m$ onto its cohomology $\mathcal{W}^0_0(U^m_0)z^m =  H^*(\mathcal{W}^0_*(U),d) z^m$, i.e. we have
\begin{equation}
\text{Id} - \iota_m \mathsf{P}_m = d H_m + H_m d.
\end{equation}
Moreover, these operators can be descended to $(\mathcal{W}^0_*(U)/\mathcal{W}^{-1}_*(U)) z^m$ contracting to its cohomology $\comp \cdot z^m \cong \big( \mathcal{W}^0_0(U^m_0)/\mathcal{W}^{-1}_0(U^m_0) \big) z^m$.

\begin{definition}\label{pathspacehomotopy}
We define the operators
$H := \bigoplus H_m, \ \mathsf{P} := \bigoplus \mathsf{P}_m \text{  and  } \iota := \bigoplus \iota_m$
acting on the direct sum $\bigoplus_m \mathcal{W}^0_*(U) z^m$ and its cohomology.
These operators extend naturally to the tensor product $\mathcal{G}^{*,*}_n(U) =  \left(\bigoplus_{m \in \mathcal{P}} \mathcal{W}^0_*(U) z^m\right) \otimes_\inte \bigwedge^*N \otimes_\comp R_n $, and descend to the quotient $(\mathcal{G}/\mathcal{I})^{*,*}_n$. Moveover, these operators preserve $\mathcal{H}^{*,*}_n$ and hence can also be defined on $\mathcal{H}^{*,*}_n$. All of the above operators will be denoted by the same notations.
\end{definition}

\subsection{Solving the Maurer-Cartan equation}\label{sec:solve_MC_equation}


Recall that we have fixed $n$ points $P_1, \dots, P_n$ in generic position, with each $P_i$ corresponding to a formal variable $u_i \in R_n$. To each $P_i$, we associate an input term of the form
$\incoming^{(i)} = u_i \delta_{P_i} (\check{\partial}_{1} \wedge \check{\partial}_2),$
where $\check{\partial}_1, \check{\partial}_1$ are the holomorphic vector fields corresponding to the basis $\{e^1,e^2\}$ of $N$ (fixed at the beginning of Section \ref{sec:semi_flat_family}), and
\begin{equation}\label{eqn:delta_P_i}
\delta_{P_i} = \frac{1}{\pi \hp} e^{-(\eta_{i,1}^2 + \eta_{i,2}^2)/\hp} d\eta_{i,1} \wedge d\eta_{i,2} \in \mathcal{W}^{\infty}_2
\end{equation}
is an $\hp$-dependent smoothing of the `delta-form' at $P_i$, for some affine coordinates $(\eta_{i,1},\eta_{i,2})$ on $M_\real$ taking the values $(0,0)$ at $P_i$.
We are interested in the Maurer-Cartan solutions in $\mathcal{H}^{*,*}_n$ constructed by summing over trees with input $\sum_{i=1}^n \incoming^{(i)}$.

Notice that we have $\delta_{P_i} \in \mathcal{W}^{2}_{P_i}$ because we can apply Lemma \ref{lem:support_product} to the expression
$$\delta_{P_i} = \left(\left(\frac{1}{\pi \hp}\right)^{1/2} e^{-(\eta_{i,1}^2)/\hp}d\eta_{i,1}\right) \wedge \left(\left(\frac{1}{\pi \hp}\right)^{1/2} e^{-(\eta_{i,2}^2)/\hp}d\eta_{i,2}\right),$$
and we have the following lemma from \cite{kwchan-leung-ma}:
\begin{lemma}[Lemma 4.14 in \cite{kwchan-leung-ma}]
For any affine linear function $\eta$ on $U$, the 1-form $\left(\frac{1}{\pi \hp}\right)^{1/2} e^{-(\eta^2)/\hp} d\eta$ has asymptotic support on the line $L := \{ \eta = 0\}$ with weight $1$.
\end{lemma}

Instead of solving the Maurer-Cartan equation directly, we will solve the equation \eqref{eqn:pseudoMC}:
\begin{equation}\label{eqn:pseudoMC}
\Phi= \incoming - H ([W,\Phi] + \half [\Phi,\Phi]),
\end{equation}
where $\Phi$ is a degree $0$ element in $\mathcal{H}^{*,*}_n$,
with the input
\begin{equation}\label{eqn:input}
\incoming := \sum_{i=1}^n \incoming^{(i)}.
\end{equation}
This originates from a method of Kuranishi \cite{Kuranishi65} in solving the Maurer-Cartan equation of the classical Kodaira-Spencer dgLa. His method can be generalized to our current situation as follows (see e.g. \cite{manetti2005differential})
\begin{prop}\label{prop:pseudoMC_to_MC}
Suppose that $\Phi $ satisfies the equation \eqref{eqn:pseudoMC}. Then $\Phi$ satisfies the Maurer-Cartan equation \eqref{eqn:extended_MC_equation_intro} if and only if $\mathsf{P}([W,\Phi] + \half [\Phi,\Phi]) = 0$.
\end{prop}
\begin{proof}
Applying $\bar{\partial}$ to both sides of \eqref{eqn:pseudoMC} (recall that $\bar{\partial}$ is identified with the de Rham differential $d$ using the Fourier transform $\mathcal{F}$ \eqref{eqn:Fourier_transform}) and using $\bar{\partial} \incoming = 0$, we obtain
$$\bar{\partial} \Phi +[W,\Phi]+ \half [\Phi,\Phi]=  H([\bar{\partial} \Phi,W+ \Phi]) + \iota\circ \mathsf{P}( [W,\Phi]+\half [\Phi,\Phi] ).$$
Suppose that $\Phi$ satisfies the MC equation \eqref{eqn:extended_MC_equation_intro}. Then we see that $[\bar{\partial}\Phi,W+\Phi] = -[[W,\Phi] + \half [\Phi, \Phi], W+ \Phi] = 0$ and hence $\mathsf{P}( [W,\Phi]+\half [\Phi,\Phi]) = 0$.
	
For the converse, we let $\delta = \bar{\partial}\Phi + [W,\Phi] + \half[\Phi,\Phi]$. It follows from the assumption $\mathsf{P}( [W,\Phi]+\half [\Phi,\Phi]) = 0$ that
$\delta = H[W+\Phi,\delta] = (H\circ ad_{W+\Phi})^m (\delta)$
for any $m\in \inte_+$. Then by the fact that $\Phi \in \mathcal{H}^{*,*}_n\otimes \mathbf{m}$, and the fact that $ad_{W}$ is an operator of degree $(1,0)$, we have $\delta = 0$ by taking $m$ large enough.
\end{proof}

We notice that $\mathsf{P} \alpha \neq 0$ only if $\alpha \in \mathcal{H}^{i,0}_n$ by its construction. When we write $\Phi = \sum_{i=0}^2\Phi^{i,i}$ with $\Phi^{i,i} \in \mathcal{H}^{i,i}_n$, and consider the term $\mathsf{P}( [W,\Phi]+\half [\Phi,\Phi]) = 0$, we notice that
$\mathsf{P}([W,\Phi^{1,1}+\Phi^{2,2}] +\half [\Phi^{1,1}+\Phi^{2,2},\Phi^{1,1}+\Phi^{2,2}] +[\Phi^{0,0},\Phi^{1,1}+\Phi^{2,2}]) = 0$
by degree reasons. Furthermore, we have $[W,\Phi^{0,0}] = 0 = [\Phi^{0,0},\Phi^{0,0}]$, and therefore $\mathsf{P}( [W,\Phi]+\half [\Phi,\Phi]) = 0$. As a result, it suffices to solve the equation \eqref{eqn:pseudoMC}.

Now we look at the equation \eqref{eqn:pseudoMC}. Letting $\varXi = \Phi - \incoming$, we can solve
$\varXi + H([W,\Phi]+\half [\Phi,\Phi]) = 0$
iteratively in $\mathcal{H}^{*,*}_n\otimes_R (R/\mathbf{m}^k)$ by increasing the power in $\mathbf{m}^k$. We write $\varXi = \sum_{i=1} \varXi_i$ and $\Phi = \sum_{i} \Phi_i$ with $\varXi_i, \Phi_i \in \mathcal{H}^{*,*}_n\otimes (\mathbf{m}^{i}/\mathbf{m}^{i+1})$. We further decompose each $\varXi_i$ and $\Phi_i$ by its degree and write $\varXi_i= \sum_{j=0}^2 \varXi^{j,j}_i$ with $\varXi^{j,j}_i \in \mathcal{H}^{j,j}_n \otimes \mathbf{m}^{i}/\mathbf{m}^{i+1}$ and similarly $\Phi_i= \sum_{j=0}^2 \Phi^{j,j}_i$.

The first order terms are simply given by $\varXi^{1,1}_1 = -H [W,\incoming]$ and $\varXi^{0,0}_1 = -H[W,\varXi^{1,1}_1]$. In general, the $k$-th order equation is given by
\begin{equation}\label{eqn:MCiterationequation}
\varXi_k +H[W,\Phi_k] +\sum_{j+l = k} \half H [\Phi_j, \Phi_l] = 0,
\end{equation}
and $\varXi_k^{j,j}$ is uniquely determined by $\varXi_i$ with $i < k$ and $\varXi_k^{r,r}$ with $r>j$. In this way, the solution $\varXi$ to \eqref{eqn:pseudoMC} is uniquely determined.

There is a beautiful way to express the unique solution $\varXi$ as a sum of terms involving the input $\incoming$ over directed trees (reminiscent of a Feynman sum). To this end, let us introduce the notion of a {\em weighted $d$-pointed $k$-tree with ribbon structure}, whose definition originated from \cite{fukaya2003deformation} (see also \cite{kwchan-leung-ma}).

\begin{definition}\label{def:r_tree_operation}
Given a weighted ribbon $d$-pointed $k$-tree $\wrtr \in \wrtree{k}{d}$, we align the marked points $p_1,\dots,p_d$ (recall that marked points is itself an edge in $\partial^{-1}_{in}(\wrtr^{[0]}_{in})$) by $p_{i_1},\dots,p_{i_d}$ according to its cyclic ordering (or the clockwise orientation on $D$ if we use the embedding $|\wrtr| \hookrightarrow D$). We define the graded operator
$\mathfrak{l}_{\wrtr}: \mathcal{H}^{*,*}[2]^{\otimes d} \rightarrow \mathcal{H}^{*,*}[2]$
for input $\zeta_1,\dots,\zeta_d \in \mathcal{H}^{*,*}[2]$ by
\begin{enumerate}
\item
writing $\zeta_j = \sum_{I \subset \{1,\dots,n\}} \alpha_{j,I} u_{I}$, and extracting the term $\alpha_{j,i_j}u_{i_j}$ in $\zeta_j$ and aligning it as the input at $p_{i_j}$, where $u_{i_j} \in R_n$ is the monomial associated to the marked point $p_{i_j}$ in Definition \ref{def:weighted_tree},

\item
aligning the term $z^{m_{e}}$ at each incoming edge in $ e \in \wrtr^{[1]}_{in} = \partial^{-1}_{in}(\wrtr^{[0]}_{in}) \setminus \{p_1,\dots,p_d\}$,

\item
applying $[\cdot,\cdot]$ at each vertex in $\wrtr^{[0]}$ according to the ordering of the ribbon structure,

\item
applying the homotopy operator $-H$ to each edge in $\wrtr^{[1]}$.
\end{enumerate}
We then define $\mathfrak{l}_{k,d}: \mathcal{H}^{*,*}[2]^{\otimes d} \rightarrow \mathcal{H}^{*,*}[2]$ by
$\mathfrak{l}_{k,d} := \sum_{\wrtr \in \wrtree{k}{d}} \frac{1}{2^{d-1}}\mathfrak{l}_{\wrtr}.$
\end{definition}

	

Setting
\begin{equation}\label{eqn:solve_sum_over_trees}
\Phi := \incoming + \varXi =  \sum_{k,d\geq 1} \mathfrak{l}_{k,d}(\incoming, \ldots, \incoming)
\end{equation}
gives the unique solution to the equation \eqref{eqn:pseudoMC} which is obtained by recursively solving \eqref{eqn:MCiterationequation}.
Note that the sum above is finite because the ideal $\mathbf{m}_n$ is nilpotent.

\section{Proof of Theorem \ref{thm:main_theorem_intro_1} by asymptotic analysis}\label{sec:theorem}

In this section, we prove our main result (i.e. Theorem \ref{thm:main_theorem_intro_1}) by using asymptotic analysis to relate the Maurer-Cartan solution $\Phi \in \mathcal{H}^{*,*}_n$, which we constructed via the sum-over-tree formula in \eqref{eqn:solve_sum_over_trees} with the specified input $\incoming$ \eqref{eqn:input}, with the tropical disk counts defined in Section \ref{sec:counting}.

\begin{notation}\label{not:compacified_disk_moduli_P_i}
Given $n$ points $P_1, \dots, P_n \in M_\real$ in generic position, we use $\overline{\mathfrak{M}}_d^{\wrtr} (\mathcal{P},\Sigma,P_1,\dots,P_n)$ to denote the space $\vec{ev}^{-1}((P_{i_1},\dots,P_{i_d}) \times M_{\real})$ which gives a compactification of $\mathfrak{M}_d^{\wrtr} (\mathcal{P},\Sigma,P_1,\dots,P_n)$ for any weighted ribbon tree $\wrtr$. Here, $\vec{ev}$ is the evaluation map defined in Definition \ref{def:evaluation_map}, $P_{i_j}$ is the point such that the monomial weight at the marked point $p_j$ is $u_{i_j}$, and note that the subset $\{i_1, \dots, i_d\} \subset \{1, \dots, n\}$ is determined by the weight of $\wrtr$.
\end{notation}

\begin{definition}\label{def:Q_edge}
Given a weighted ribbon $d$-pointed $k$-tree $\wrtr \in \wrtree{k}{d}$ with $u_\wrtr \neq 0$, we associate to each of its edges $e \in \bar{\wrtr}$ a tropical polyhedral subset $Q_{e} \subset M_\real$ as follows.
For each incoming edge $e \in \wrtr^{[1]}_{in}$, we assign $Q_{e} = M_\real$, and for each marked point $p_j$ we assign $Q_{p_j} = P_{i_j}$ where the monomial weight at $p_j$ is $u_{i_j}$. We then inductively assign a (possibly empty) tropical polyhedral subset $Q_e$ to each edge $e \in \wrtr^{[1]}$ by the following rule:
	
If $e_1$ and $e_2$ are two incoming edges meeting at a vertex $v$ with an outgoing edge $e_3$ for which $Q_{e_1}$ and $Q_{e_2}$ are defined beforehand, we set
$Q_{e_3} := (Q - \real_{\geq 0} \bar{m}_{e_3})$
if both $Q_{e_1}$ and $Q_{e_1}$ are non-empty and they intersect transversally at $Q := Q_{e_1} \cap Q_{e_2}$, and
$Q_{e_3} := \emptyset$
otherwise. 
	
We denote the tropical polyhedral subset associated to the unique outgoing edge $e_o$ by $Q_{\wrtr}$.
\end{definition}

We start with a combinatorial lemma concerning the tropical polyhedral subset $Q_\wrtr$.
\begin{lemma}\label{lem:asy_support_combinatorics}
If $MI(\wrtr) <0$, then both $Q_\wrtr$ and $\overline{\mathfrak{M}}_d^{\wrtr} (\mathcal{P},\Sigma,P_1,\dots,P_n)$ are empty. For $MI(\wrtr)=0$ or $2$ and $\text{Mult}(\wrtr) \neq 0$, $ev_o$ is a diffeomorphism onto its image and we have $Q_\wrtr = ev_o(\overline{\mathfrak{M}}_d^{\wrtr} (\mathcal{P},\Sigma,P_1,\dots,P_n))$, which is of dimension $\frac{MI(\wrtr)}{2}+1$ if $Q_{\wrtr} \neq \emptyset$.
\end{lemma}
\begin{proof}	
We prove by induction on the number of vertices in $\wrtr^{[0]}$. The initial case is when $\wrtr^{[0]} = \emptyset$, i.e. when there are no trivalent vertices. Then the only possible trees are the ones with a unique edge $e$. In this case we have $MI(\wrtr)=2$ and $Q_\wrtr = M_\real$, and the lemma holds automatically.
	
For the induction step, suppose we have a tree $\wrtr$ with $\wrtr^{[0]} \neq \emptyset$ with the unique root vertex $v_r \in \wrtr^{[0]}$ connecting to the outgoing edge $e_o$ with two incoming edges $e_1$ and $e_2$. We split $\wrtr$ at $v_r$ to obtain two trees $\wrtr_1, \wrtr_2$ with outgoing edges $e_1, e_2$ and $k_1, k_2$ incoming edges, $d_1$ and $d_2$ marked points respectively. Then we have the decomposition
\begin{equation}\label{eqn:moduli_space_splitting}
\left( \overline{\mathfrak{M}}_{d_1}^{\wrtr_1} (\mathcal{P},\Sigma,P_1,\dots,P_n) {}_{ev_o}\times_{ev_o} \overline{\mathfrak{M}}_{d_2}^{\wrtr_2} (\mathcal{P},\Sigma,P_1,\dots,P_n) \right) \times \real_{\geq 0} \cdot (-\bar{m}_{\wrtr})
 =\overline{\mathfrak{M}}_{d}^{\wrtr} (\mathcal{P},\Sigma,P_1,\dots,P_n),
\end{equation}
and there are two cases to consider.
	
The first case is when one of the incoming edges, say $e_2$, is an edge corresponding to a marked point so that $k_2 = 0$ and $d_2 =1$. In this case $\wrtr_2$ is not a weighted tree in the sense of Definition \ref{def:weighted_tree}, but we can still take $Q_{\wrtr_2}$ to be the point $P_{e_2}$ associated to $e_2$.

If $MI(\wrtr_1) \leq 0$, then by the induction hypothesis and the generic assumption (Definition \ref{def:generic_position}), $Q_{\wrtr_1}$ cannot intersect $Q_{\wrtr_2}$ transversally and hence $Q_\wrtr = \emptyset$. On the other hand we have $MI(\wrtr)<0$, so $\mathfrak{M}_d^{\wrtr} (\mathcal{P},\Sigma,P_1,\dots,P_n) = \emptyset$.
	
If $MI(\wrtr_1) = 2$, then $Q_{\wrtr_1}$ intersect $Q_{\wrtr_2}$ transversally at $Q_{\wrtr_2}$ automatically if $Q_{\wrtr_2}$ lies on $Q_{\wrtr_1}$, and otherwise both $Q_{\wrtr}=\mathfrak{M}_d^{\wrtr} (\mathcal{P},\Sigma,P_1,\dots,P_n) = \emptyset$. In this case $MI(\wrtr) = 0$, $\text{Mult}(\wrtr) = \text{Mult}(\wrtr_1)$ and $m_\wrtr = m_{\wrtr_1}$.
Assuming $\text{Mult}(\wrtr) = \text{Mult}(\wrtr_1) \neq 0$, we have $Q_{\wrtr_1} = (ev_o)_*(\overline{\mathfrak{M}}_{d_1}^{\wrtr_1} (\mathcal{P},\Sigma,P_1,\dots,P_n))$ by the induction hypothesis and the above decomposition becomes
\begin{align*}
\left(\overline{\mathfrak{M}}_{d_1}^{\wrtr_1} (\mathcal{P},\Sigma,P_1,\dots,P_n) \cap (ev_{\wrtr_1,o})^{-1}(Q_{\wrtr_2}) \right) \times \real_{\geq 0} \cdot (-\bar{m}_{\wrtr_1}) = \overline{\mathfrak{M}}_{d}^{\wrtr} (\mathcal{P},\Sigma,P_1,\dots,P_n),
\end{align*}
implying that $Q_\wrtr = ev_o(\overline{\mathfrak{M}}_{d}^{\wrtr} (\mathcal{P},\Sigma,P_1,\dots,P_n)) $, and hence the dimension of $Q_\wrtr $ is exactly given by $\frac{MI(\wrtr)}{2}+1$ \footnote{We indeed have $(ev_{\wrtr_1,o})^{-1}(Q_{\wrtr_2}) \in  \mathfrak{M}_{d_1}^{\wrtr_1} (\mathcal{P},\Sigma,P_1,\dots,P_n)$ due to the generic assumption on $P_1,\dots,P_n$'s.}.
	
The second case is when both $\wrtr_1$ and $\wrtr_2$ have $k_1,k_2 \geq 1$. In this case we have $MI(\wrtr) = MI(\wrtr_1) + MI(\wrtr_2)$ and the two moduli spaces $\overline{\mathfrak{M}}_{d_i}^{\wrtr_i} (\mathcal{P},\Sigma,P_1,\dots,P_n)$ have dimensions $MI(\wrtr_i)/2 +1$ respectively if they are non empty. Using the decomposition in equation \eqref{eqn:moduli_space_splitting},
we notice that if $\overline{\mathfrak{M}}_{d_i}^{\wrtr_i} (\mathcal{P},\Sigma,P_1,\dots,P_n) = \emptyset$ for $i = 1$ or $2$, then $Q_{\wrtr}=ev_o(\overline{\mathfrak{M}}_{d}^{\wrtr} (\mathcal{P},\Sigma,P_1,\dots,P_n)) = \emptyset$. So $\overline{\mathfrak{M}}_{d}^{\wrtr} (\mathcal{P},\Sigma,P_1,\dots,P_n) = \emptyset $ if $MI(\wrtr)<0$. Therefore it remains to consider the cases when $MI(\wrtr_i) = 0,2$ and $MI(\wrtr) = 0,2$.
	
Assuming $\text{Mult}(\wrtr) = \text{Mult}(\wrtr_1) \text{Mult}(\wrtr_2) \text{Mult}_{v_r}(\wrtr) \neq 0$, from the induction hypothesis we have $Q_{\wrtr_i} = ev_o(\overline{\mathfrak{M}}_{d_i}^{\wrtr_i} (\mathcal{P},\Sigma,P_1,\dots,P_n))$ for $i = 1, 2$. Since we have $\text{Mult}_{v_r}(\wrtr) \neq 0$, $Q_{\wrtr_1}$ and $Q_{\wrtr_2}$ can only intersect transversally. Therefore, if $\text{Mult}(\wrtr) \neq 0$, then we have $Q_{\wrtr_1}$ and $Q_{\wrtr_2}$ intersecting transversally and $Q_\wrtr = Q_{\wrtr_1} \cap Q_{\wrtr_2} - \real_{\geq0} \bar{m}_{\wrtr} = (ev_o)_* (\overline{\mathfrak{M}}_{d}^{\wrtr} (\mathcal{P},\Sigma,P_1,\dots,P_n))$  from the decomposition \eqref{eqn:moduli_space_splitting} and the Definition \ref{def:Q_edge} for $Q_\wrtr$. Finally, by the generic assumption on $P_1,\dots,P_n$, $Q_\wrtr$ has dimension $\frac{MI(\wrtr)}{2}+1$ whenever it is nonempty.
\end{proof}

\begin{lemma}\label{lem:Q_wrtr_and_homotopy}
There exists a large enough $R>0$ such that the half space $U^{m}_{1,+} $ in Notations \ref{not:perpendicular_hyperplane} contains $\text{Sing}(\mathscr{D})$ and also the tropical polyhedral subset $Q_\wrtr$ for any $\wrtr$ with $m_\wrtr = m$, $MI(\wrtr) = 0$ or $2$, $\text{Mult}(\wrtr)\neq 0$, $u_\wrtr \neq 0$ and with at least one marked point.
\end{lemma}
\begin{proof}
The existence of a fixed $R$ depends on the finiteness of the total number of weighted ribbon trees $\wrtr$ (for arbitrary number of marked points and $k = |\wrtr^{[0]}_{in}|$) with $MI(\wrtr) = 0 ,2$, $\text{Mult}(\wrtr) \neq 0$ and  $u_\wrtr \neq 0$. We prove by induction on the number $N$ of vertices in $\wrtr^{[0]}$ the existence of $R_N>0$ satisfying the lemma for all $\wrtr$ with $|\wrtr^{[0]}|\leq N$.

The initial case concerns the tree $\wrtr$ with an unique internal vertex $v_r$, with two incoming edges $e_1$ and $e_2$, and one outgoing edge $e_o$ in clockwise orientation. Furthermore, we have $e_1 \in \wrtr^{[1]}_{in}$ and $e_2$ is an edge corresponding to a marked point with monomial weight $u_{e_2}$. In this case we have $MI(\wrtr) = 0$ and $Q_{\wrtr} = P_{e_2} - \real_{\geq 0} \cdot \bar{m}_{\wrtr}$ which is lying in $U^{m}_{1,+}$ as we required $\text{Sing}(\mathscr{D}) \subset U^{m}_{1,+}$ when we chose $U^{m}_{1,+}$ in Notations \ref{not:perpendicular_hyperplane}.
	
For the induction step, suppose we have a tree $\wrtr$ with $|\wrtr^{[0]}| = N+1$ with the unique root vertex $v_r \in \wrtr^{[0]}$ connecting to the outgoing edge $e_o$ with two incoming edges $e_1$ and $e_2$. We split $\wrtr$ at $v_r$ to obtain two trees $\wrtr_1, \wrtr_2$ with outgoing edges $e_1, e_2$ and $k_1, k_2$ incoming edges, $d_1$ and $d_2$ marked points respectively. There are two cases to consider (as in the proof of Lemma \ref{lem:asy_support_combinatorics}).
	
This first case is when one of the incoming edges, say $e_2$, is an edge corresponding to a marked point so that $k_2 = 0$ and $d_2 =1$. We let $Q_{\wrtr_2} = P_{e_2}$ to be the corresponding marked point. From the proof of Lemma \ref{lem:asy_support_combinatorics}, we know that we must have $MI(\wrtr_1) = 2$ and $MI(\wrtr) = 0$ for $Q_\wrtr \neq \emptyset$. In this case $Q_\wrtr = P_{e_2} - \real_{\geq 0} \cdot \bar{m}_\wrtr$ and we have $Q_\wrtr \subset  U^{m}_{1,+}$ by the same reason as in the initial step.
	
In the second case we have both $\wrtr_1$ and $\wrtr_2$ having $k_1,k_2 \geq 1$, and we have $MI(\wrtr) = MI(\wrtr_1) + MI(\wrtr_2)$. Assuming $Q_\wrtr \neq \emptyset$, then one of the $Q_{\wrtr_1}, Q_{\wrtr_2}$ is a ray or a line, and we assume that it is $Q_{\wrtr_1}$, with $MI(\wrtr_2) =0,2$. Therefore for any point $x \in Q_{\wrtr_1} \cap Q_{\wrtr_2}$ we have the relations $g_0(-\bar{m}_{\wrtr_1},x) \geq -R_N$ and $g_0(-\bar{m}_{\wrtr_2},x) \geq -R_N$, and hence $g_0(-\bar{m}_{\wrtr},x) \geq -2 R_N$ \footnote{Here $g_0$ is the linear metric introduced in Notation \ref{not:perpendicular_hyperplane}.}. Therefore by taking $R_{N+1} = 2R_N$, we have $g_0(-\bar{m}_{\wrtr},x) \geq -R_{N+1}$ and hence $Q_\wrtr = Q_{\wrtr_1} \cap Q_{\wrtr_2} - \real_{\geq 0} \cdot \bar{m}_{\wrtr} \subset U^{m}_{1,+}$ for $\bar{m} \neq 0$ as desired.
\end{proof}

We are now ready to prove the key lemma which relates our Maurer-Cartan solution with the locus $Q_\wrtr$ traced out by the stops of the tropical disks introduced in Definition \ref{def:Q_edge}.

\begin{notation}\label{not:alpha_operation}
Given a weighted ribbon $d$-pointed $k$-tree $\wrtr$, we define a differential form $\alpha_{\wrtr} \in \mathcal{W}^{0}_*$ as follows. First we align the marked points $p_1,\dots,p_d$ (recall that a marked point is itself an edge in $\partial^{-1}_{in}(\wrtr^{[0]}_{in})$) by $p_{i_1},\dots,p_{i_d}$ according to its cyclic ordering. Then $\alpha_{\wrtr}$ is the output of the following procedure:
\begin{enumerate}
\item
aligning $\delta_{P_{i_j}}$ as the input at the edge corresponding to the marked point $p_{i_j}$, if the monomial weight associated to $p_{i_j}$ is $u_{i_j}$,
\item
aligning the constant $1$ at each incoming edge in $ e \in \wrtr^{[1]}_{in} = \partial^{-1}_{in}(\wrtr^{[0]}_{in}) \setminus \{p_1,\dots,p_d\}$,
\item
applying the wedge product $\wedge$ at each vertex in $\wrtr^{[0]}$ according to the ordering of the ribbon structure,
\item
applying the homotopy operator $-H$ to each edge in $\wrtr^{[1]}$.
\end{enumerate}
\end{notation}

\begin{definition}\label{def:operator_partity}
Given a weighted ribbon $d$-pointed $k$-tree $\wrtr \in \wrtree{k}{d}$ with $\text{Mult}(\wrtr)\neq 0$, we set
$(-1)^{\chi(\wrtr)} := \prod_{v \in \wrtr^{[0]}} (-1)^{\chi(\wrtr,v)},$
where $(-1)^{\chi(\wrtr,v)}$ is defined by the rules (with the convention that $(-1)^{\chi(\wrtr)} = 1$ if $\wrtr^{[0]} = \emptyset$): if $v$ is connected to a marked point we set $\chi(\wrtr,v) = 0$, and $(-1)^{\chi(\wrtr,v)}$ is defined inductively along the tree $\wrtr$ for each trivalent vertex $v$ not connecting to any marked point $p_i$'s (attached to two incoming edges $e_1, e_2$ and one outgoing edge $e_3$ so that $e_1,e_2,e_3$ are arranged in the clockwise orientation) by comparing the orientation of the ordered basis $\{-\bar{m}_{e_1}, -\bar{m}_{e_2}\}$ with that of $B_0$.
\end{definition}

\begin{lemma}\label{lem:tree_support_lemma}
Let $\wrtr \in \wrtree{k}{d}$ be a weighted ribbon $d$-pointed $k$-tree. Then we have
\begin{equation*}
\mathfrak{l}_{\wrtr}(\incoming,\dots,\incoming) = \left\{
\begin{array}{ll}
0  &  \text{if $MI(\wrtr)\neq 0,2$ or $Q_\wrtr = \emptyset$ or $\text{Mult}(\wrtr) =0$},\\
(-1)^{\chi(\wrtr)} \alpha_{\wrtr} \text{Mult}(\wrtr)z^{m_\wrtr} u_{\wrtr}  &  \text{if $MI(\wrtr)=2$ and $Q_\wrtr \neq \emptyset$ and $\text{Mult}(\wrtr) \neq 0$},\\
(-1)^{\chi(\wrtr)} \alpha_{\wrtr} k_{\wrtr} \text{Mult}(\wrtr)z^{m_\wrtr} \check{\partial}_{n_{Q_\wrtr}} u_{\wrtr} & \text{if $MI(\wrtr)=0$ and $Q_\wrtr \neq \emptyset$ and $\text{Mult}(\wrtr) \neq 0$}
\end{array}
\right.
\end{equation*}
in $\mathcal{H}^{*,*}_n$, where $\alpha_{\wrtr} \in \mathcal{W}^{s_\wrtr}_{Q_{\wrtr}}$ in which $s_\wrtr := 1 - \frac{MI(\wrtr)}{2}$, and $n_{Q_\wrtr}$ is the clockwise oriented normal to the ray or line $Q_\wrtr$ when $Q_\wrtr \neq \emptyset$ in the case $MI(\wrtr) = 0$.
\end{lemma}
\begin{proof}
First of all, from Notations \ref{not:alpha_operation} we can see that the degree of the form $\alpha_\wrtr$ is exactly given by $s_{\wrtr}$ which can only be $0$ or $1$ since the operator associated to the outgoing edge is a homotopy operator and it decreases the degree by $1$. Therefore we notice that $\alpha_{\wrtr} \neq 0$ except when $MI(\wrtr) = 0$ or $2$.
	
Once again, we prove by induction on the number of vertices in $\wrtr^{[0]}$. The initial case is when $\wrtr^{[0]} = \emptyset$ and the only possible trees the ones with a unique edge $e$. In this case, we have $MI(\wrtr)=2$, $Q_\wrtr = M_\real$ and $\mathfrak{l}_{\wrtr}(\incoming,\dots,\incoming) = z^{m_\wrtr}$, so the lemma holds.
	
For the induction step, suppose we have a tree $\wrtr$ with $\wrtr^{[0]} \neq \emptyset$ with the unique root vertex $v_r \in \wrtr^{[0]}$ connecting to the outgoing edge $e_o$ with two incoming edges $e_1$ and $e_2$. We split $\wrtr$ at $v_r$ to obtain two trees $\wrtr_1, \wrtr_2$ with outgoing edges $e_1, e_2$ and $k_1, k_2$ incoming edges, $d_1$ and $d_2$ marked points respectively. As before, there are two possible scenarios.
	
This first case is when one of the incoming edges, say $e_2$, is an edge corresponding to a marked point so that $k_2 = 0$ and $d_2 =1$. In this case we let $P_{e_2} = Q_{\wrtr_2}$ to be the marked point associated to $e_2$. The proof of Lemma \ref{lem:asy_support_combinatorics} shows that we must have $MI(\wrtr_1)=2$ and $MI(\wrtr)=0$ in order to have $Q_\wrtr \neq \emptyset$, and $Q_\wrtr \neq \emptyset$ if and only if $P_{e_2} \in Q_{\wrtr_1}$. By the induction hypothesis we have $\mathfrak{l}_{\wrtr_1}(\incoming,\dots,\incoming)  =(-1)^{\chi(\wrtr_1)} \alpha_{\wrtr_1} \text{Mult}(\wrtr_1) z^{m_{\wrtr_1}} u_{\wrtr_1}$ with $\alpha_1 \in \mathcal{W}^0_{Q_{\wrtr_1}}$. Therefore we have
$$\mathfrak{l}_{\wrtr}(\incoming,\dots,\incoming) = -(-1)^{\chi(\wrtr_1)} \text{Mult}(\wrtr_1) H(\alpha_{\wrtr_1} \wedge \delta_{P_{e_2}}) [z^{m_{\wrtr_1}}, \check{\partial}_1\wedge \check{\partial}_2] u_{\wrtr_1} u_{e_2},$$
where $u_{\wrtr_1} u_{e_2} = u_\wrtr$.

Now Lemma \ref{lem:support_product} implies that $\alpha_1 \wedge \delta_{P_{e_2}} \in \mathcal{W}^2_{P_{e_2}}$. By our choice we have $P_{e_2} \in U^{m_\wrtr}_{1,+}$ and hence applying Lemma \ref{lem:integral_lemma_modified}, we get $\alpha_{\wrtr}=-H(\alpha_1 \wedge \delta_{P_{e_2}}) \in \mathcal{W}^{1}_{Q_\wrtr}$, where $Q_\wrtr = P_{e_2} - \real_{\geq 0 } \cdot \bar{m}_\wrtr$ as in Definition \ref{def:Q_edge}. Furthermore, we have $[z^{m_{\wrtr_1}}, \check{\partial}_1 \wedge \check{\partial}_2] = (e^*_2, \bar{m}_{\wrtr_1}) \check{\partial}_2 - (e^*_1, \bar{m}_{\wrtr_1}) \check{\partial}_1$, where $e^*_1, e^*_2$ is the dual basis to $e_1, e_2$ introduced in Notations \ref{not:holomorphic_coordinates}. As in Notations \ref{not:integer_weight}, we can write $k_{\wrtr} \hat{m}_{\wrtr} = \bar{m}_{\wrtr}$ for some primitive $\hat{m}_{\wrtr} \in M$. Since we have $m_{\wrtr_1} = m_{\wrtr}$, we find that $(e^*_2, \bar{m}_{\wrtr_1}) \check{\partial}_2 - (e^*_1, \bar{m}_{\wrtr_1}) \check{\partial}_1 =k_{\wrtr} n_{Q_{\wrtr}}$. Together with the fact that $\chi(\wrtr)= \chi(\wrtr_1)$ and $\text{Mult}(\wrtr) = \text{Mult}(\wrtr_1)$, we obtain the desired identity in this case.
	
In the second case we have both $\wrtr_1$ and $\wrtr_2$ having $k_1,k_2 \geq 1$, and $MI(\wrtr) = MI(\wrtr_1) + MI(\wrtr_2)$. Assuming $Q_\wrtr \neq \emptyset$, then one of $Q_{\wrtr_1}, Q_{\wrtr_2}$, say $Q_{\wrtr_1}$, is a ray or a line, and they intersect transversally. There are two subcases depending on whether $MI(\wrtr_2) = 0$ or $2$.

We first assume that $MI(\wrtr_2) = 0$. Then we can write
$\mathfrak{l}_{\wrtr_i}(\incoming,\dots,\incoming) =  (-1)^{\chi(\wrtr_i)} \text{Mult}(\wrtr_i) \alpha_{\wrtr_i} z^{m_{\wrtr_i}} \check{\partial}_{n_i} u_{\wrtr_i},$
where we abbreviate $n_i = k_{\wrtr_i} n_{Q_{\wrtr_i}}$ and $n_{Q_{\wrtr_i}}$ is the primitive clockwise oriented normal to $Q_{\wrtr_i}$. Therefore we have
$$\mathfrak{l}_{\wrtr}(\incoming,\dots,\incoming) = -(-1)^{\chi(\wrtr_1)+\chi(\wrtr_2)} \text{Mult}(\wrtr_1) \text{Mult}(\wrtr_2) H(\alpha_{\wrtr_1} \wedge \alpha_{\wrtr_2}) [z^{m_{\wrtr_1}} \check{\partial}_{n_1}, z^{m_{\wrtr_2}} \check{\partial}_{n_2}] u_{\wrtr_1} u_{\wrtr_2}.$$
Using Lemma \ref{lem:support_product} we see that $\alpha_{\wrtr_1} \wedge \alpha_{\wrtr_2} \in \mathcal{W}^2_{Q_{\wrtr_1} \cap Q_{\wrtr_2}}$ in the case that $Q_{\wrtr_1}$ and $Q_{\wrtr_2}$ are intersecting transversally (otherwise the product is $0 \in \mathcal{H}^{0,*}_n$).  Applying Lemma \ref{lem:Q_wrtr_and_homotopy} together with Lemma \ref{lem:integral_lemma_modified}, we get $\alpha_{\wrtr} = -H(\alpha_{\wrtr_1} \wedge \alpha_{\wrtr_2}) \in \mathcal{W}^{1}_{Q_\wrtr}$. Furthermore, we have $[z^{m_{\wrtr_1}} \check{\partial}_{n_1}, z^{m_{\wrtr_2}} \check{\partial}_{n_2}]  = z^{m_{\wrtr}} \big( (\bar{m}_{\wrtr_2},n_{\wrtr_1}) \check{\partial}_{n_2} - (\bar{m}_{\wrtr_1},n_{\wrtr_2}) \check{\partial}_{n_1}  \big)$, and $(\bar{m}_{\wrtr_2},n_{\wrtr_1}) \check{\partial}_{n_2} - (\bar{m}_{\wrtr_1},n_{\wrtr_2}) \check{\partial}_{n_1}  = \det(\bar{m}_{e_1},\bar{m}_{e_2}) (n_{\wrtr_1} + n_{\wrtr_2}).
$
If $\{-\bar{m}_{e_1}, -\bar{m}_{e_2}\}$ is positively oriented, then $\det(\bar{m}_{e_1},\bar{m}_{e_2}) > 0$ and $n_{\wrtr_1} + n_{\wrtr_2} = k_{\wrtr} n_{Q_{\wrtr}}$, where $k_{\wrtr}$ is introduced in Notations \ref{not:integer_weight}, and $\det(\bar{m}_{e_1},\bar{m}_{e_2}) = \text{Mult}_{v_r}(\wrtr)$. Notice that switching to the assumption that $\{-\bar{m}_{e_1}, -\bar{m}_{e_2}\}$ is negatively oriented would result in a minus sign in $\det(\bar{m}_{e_1},\bar{m}_{e_2})$ and hence contribute an extra $(-1)^{\chi(\wrtr,v_r)}$ in the formula (i.e. in this case $\chi(\wrtr,v_r) = 1$). Combining with the fact that $\text{Mult}(\wrtr) = \text{Mult}(\wrtr_1) \text{Mult}_{v_r}(\wrtr)$, $(-1)^{\chi(\wrtr)} = (-1)^{\chi(\wrtr_1)} (-1)^{\chi(\wrtr,v_r)}$, we obtain the desired formula.
	
In the second subcase we assume that $MI(\wrtr_2) = 2$, so by the induction hypothesis we have $\mathfrak{l}_{\wrtr_2}(\incoming,\dots,\incoming) = (-1)^{\chi(\wrtr_2)} \alpha_{\wrtr_2} \text{Mult}(\wrtr_2)z^{m_{\wrtr_2}} u_{\wrtr_2}$. Therefore we have
$$\mathfrak{l}_{\wrtr}(\incoming,\dots,\incoming) = -(-1)^{\chi(\wrtr_1)+\chi(\wrtr_2)} \text{Mult}(\wrtr_1) \text{Mult}(\wrtr_2) H(\alpha_{\wrtr_1} \wedge \alpha_{\wrtr_2}) [z^{m_{\wrtr_1}} \check{\partial}_{n_1}, z^{m_{\wrtr_2}}] u_{\wrtr_1} u_{\wrtr_2},$$
where we absorb the $k_{\wrtr_1}$ into $n_1 = k_{\wrtr_1} n_{Q_{\wrtr_1}}$ again. Applying Lemma \ref{lem:Q_wrtr_and_homotopy}, \ref{lem:support_product} and \ref{lem:integral_lemma_modified} as in the previous subcase, we obtain that $\alpha_{\wrtr} = -H(\alpha_{\wrtr_1} \wedge \alpha_{\wrtr_2}) \in \mathcal{W}^{0}_{Q_\wrtr}$. Furthermore, we have $[z^{m_{\wrtr_1}} \check{\partial}_{n_1}, z^{m_{\wrtr_2}}] = det(\bar{m}_{\wrtr_1}, \bar{m}_{\wrtr_2}) z^{m_{\wrtr}} = (-1)^{\chi(\wrtr,v_r)} \text{Mult}_{v_r}(\wrtr)$ which gives us the desired identity.
\end{proof}

Next we would like to take a closer look at the differential form $\alpha_{\wrtr}$ defined in Notations \ref{not:alpha_operation}.
\begin{definition}[cf. Definition 5.29 in \cite{kwchan-leung-ma}]\label{def:trop_tree_orientation}
We attach a differential form $\nu_e$ on $\real_{\leq 0}^{|\wrtr^{[1]}|}$ to each $e \in \bar{\wrtr}^{[1]}$ recursively by the rules: $\nu_e := 1$ for each incoming edge $e \in \partial_{in}^{-1}(\wrtr^{[0]}_{in})$; $\nu_{e_3} =  (-1)^{|\nu_{e_1}||\nu_{e_2}|} \nu_{e_1} \wedge \nu_{e_2} \wedge ds_{e_{3}}$ (here $|\nu_{e_2}|$ is the cohomological degree of $\nu_{e_2}$) if $v$ is an internal vertex with incoming edges $e_1,e_2 \in \wrtr_0$ and outgoing edge $e_3$ such that $e_1, e_2, e_3$ is clockwise oriented.

We let $\nu_{\wrtr}$ be the differential form attached to the unique outgoing edge $e_o \in \wrtr^{[1]}$, which defines a volume form or orientation on $\real_{\leq 0}^{|\wrtr^{[1]}|}$.
\end{definition}

Given a weighted ribbon $d$-pointed $k$-tree $\wrtr$ with $MI(\wrtr)=0$ with $Q_{\wrtr} \neq \emptyset$, which is either a ray or a line, we let $\eta_\wrtr$ be the unique affine function on $M_\real$ such that $\eta_\wrtr = 0$ on $Q_\wrtr$ and $\eta_\wrtr$ takes positive values on the anti-clockwise oriented normal to $Q_\wrtr$.
\begin{lemma}\label{lem:disk_moduli_orientation_comparsion}
For a weighted ribbon $d$-pointed $k$-tree $\wrtr$ with $MI(\wrtr)=0,2$ and $Q_{\wrtr} \neq \emptyset$ and $\text{Mult}(\wrtr) \neq 0$, there exists some $c>0$ such that
\begin{equation*}
(ev_{\wrtr,i_1})^*(d\eta_1 d\eta_2) \cdots (ev_{\wrtr,i_d})^*(d\eta_1 d\eta_2) = \left\{
\begin{array}{ll}
(-1)^{\chi(\wrtr)} c\nu_\wrtr + \varepsilon  & \text{if $MI(\wrtr) = 2$},\\
(-1)^{\chi(\wrtr)} c\nu_\wrtr \wedge d\eta_{\wrtr} + \varepsilon & \text{if $MI(\wrtr) = 0$},
\end{array}
\right.
\end{equation*}
where $\varepsilon$ satisfies $\iota_{\nu_{\wrtr}^\vee}\varepsilon = 0$ (here $\nu_{\wrtr}^{\vee}$ is a top polyvector field dual to $\nu_{\wrtr}$ over the component $\real_{\leq 0}^{|\wrtr^{[1]}|}$) and $\eta_1, \eta_2$ are the affine coordinates on $M_\real$ with respect to the oriented basis $e_1,e_2$ introduced in Notations \ref{not:holomorphic_coordinates}.
\end{lemma}
\begin{proof}
First of all, notice that both $\overline{\mathfrak{M}}_d^\wrtr(\mathcal{P},\Sigma)$ and $M_\real^{d}$ are affine manifolds and $\vec{ev}$ is affine linear. So all the differential forms appearing in this lemmma are affine differential forms. Therefore it suffices to check the equality at a point in $\overline{\mathfrak{M}}_d^\wrtr(\mathcal{P},\Sigma)$.
Also since $\overline{\mathfrak{M}}_d^\wrtr(\mathcal{P},\Sigma) \cong \real_{\leq 0}^{|\wrtr^{[1]}|} \times M_\real$, we can always write
\begin{equation*}
(ev_{\wrtr,i_1})^*(d\eta_1 d\eta_2) \cdots (ev_{\wrtr,i_d})^*(d\eta_1 d\eta_2) = \left\{
\begin{array}{ll}
c'\nu_\wrtr + \varepsilon  & \text{if $MI(\wrtr) = 2$},\\
c'\nu_\wrtr \wedge \alpha+ \varepsilon & \text{if $MI(\wrtr) = 0$}
\end{array}
\right.
\end{equation*}
for some $c' \in \real$, and some $1$-form $\alpha \in \Omega^1(M_\real)$ with $\iota_{\nu_{\wrtr}^\vee}\varepsilon = 0$. We need to show that $\alpha$ is a constant multiple of $d\eta_{\wrtr}$ and the constant $c' = (-1)^{\chi(\wrtr)} c$ for some $c > 0$.

In the case $MI(\wrtr) = 0$ with $Q_\wrtr \neq \emptyset$, the moduli space $\overline{\mathfrak{M}}_d^\wrtr(\mathcal{P},\Sigma,P_1,\dots,P_n)$ is a 1-dimensional affine subspace of $\overline{\mathfrak{M}}_d^\wrtr(\mathcal{P},\Sigma)$. We take any path $\varsigma$ lying inside $\overline{\mathfrak{M}}_d^\wrtr(\mathcal{P},\Sigma,P_1,\dots,P_n) \subset \overline{\mathfrak{M}}_d^\wrtr(\mathcal{P},\Sigma)$. Since $ev_{\wrtr,i_j}\circ \varsigma$ is a constant map for any $j = 1, \dots, d$, we have
$\iota_{\varsigma'} ((ev_{\wrtr,i_1})^*(d\eta_1 d\eta_2) \cdots (ev_{\wrtr,i_d})^*(d\eta_1 d\eta_2)) = 0,$
where $\varsigma'$ is the affine vector field on $\overline{\mathfrak{M}}_d^\wrtr(\mathcal{P},\Sigma)$ induced by $\varsigma$.
On the other hand, $(ev_o)_*(\varsigma')$ is tangent to $Q_\wrtr = ev_o \left( \overline{\mathfrak{M}}_d^\wrtr(\mathcal{P},\Sigma,P_1,\dots,P_n) \right)$. So $\alpha$ must be a constant multiple of $d\eta_{\wrtr}$ and we can write
$(ev_{\wrtr,i_1})^*(d\eta_1 d\eta_2) \cdots (ev_{\wrtr,i_d})^*(d\eta_1 d\eta_2) =  c'\nu_{\wrtr} \wedge d\eta_\wrtr + \varepsilon,$
for some constant $c'$.
	
We now prove that $c'$ is of the form $(-1)^{\chi(\wrtr)} c$ for some $c > 0$ by induction on the number of vertices in $\wrtr^{[0]}$. The initial case is when $\wrtr^{[0]} = \emptyset$ and the only possible trees are those with a unique edge $e$. Since there are no evaluation maps, we adopt the convention that the left hand side of the equality in the lemma is equal to 1 to make the statement true in this case.
	
For the induction step, suppose we have a tree $\wrtr$ with $\wrtr^{[0]} \neq \emptyset$ with the unique root vertex $v_r \in \wrtr^{[0]}$ connecting to the outgoing edge $e_o$ with two incoming edges $e_1$ and $e_2$. We split $\wrtr$ at $v_r$ to obtain two trees $\wrtr_1, \wrtr_2$ with outgoing edges $e_1, e_2$ and $k_1, k_2$ incoming edges, $d_1$ and $d_2$ marked points respectively. There are two possible cases.

This first case is when one of the incoming edges, say $e_2$, is an edge corresponding to a marked point so that $k_2 = 0$ and $d_2 =1$. As in the proof of Lemma \ref{lem:asy_support_combinatorics}, we must have $MI(\wrtr_1)=2$ and $MI(\wrtr) = 0$. We use the identification
$\overline{\mathfrak{M}}_{d_1}^{\wrtr_1} (\mathcal{P},\Sigma) _{\ ev_{\wrtr_1,o}}\times_{\tau_{e_o}} (\real_{\leq 0}\times M_\real) = \overline{\mathfrak{M}}_d^{\wrtr} (\mathcal{P},\Sigma),$
under which the evaluation map $ev_{\wrtr,o} : \overline{\mathfrak{M}}_d^{\wrtr} (\mathcal{P},\Sigma) \rightarrow M_\real$ is identified as the projection to the last coordinate of the product on the left hand side, and the evaluation at the marked point $e_2$ is identified as the projection $\tau_{e_o}$ to the second factor of $\real_{\leq 0}\times M_\real$. We have
$$(ev_{\wrtr_1,i_1})^*(d\eta_1 d\eta_2) \cdots (ev_{\wrtr_1,i_{d_1}})^*(d\eta_1 d\eta_2) = (-1)^{\chi(\wrtr_1)}c \nu_{\wrtr_1} + \varepsilon_{\wrtr_1}$$
for some $c > 0$ by the induction hypothesis. Since $MI(\wrtr) = 0$, $Q_\wrtr$ is a ray or a line. We take an affine path $\varrho$ in $M_\real$ transversal to $Q_\wrtr$ parametrized by the affine coordinate $\eta_{\wrtr}$. Then restricting to $\real_{\leq 0} \times \varrho$, we have
$ev_{\wrtr,i_d}^*(d\eta_1d\eta_2) = \tau_{e_o}^*(d\eta_1 d\eta_2) = ds_{e_o} \wedge d\eta_{\wrtr},$
where $s_{e_o}$ is the coordinate on $\real_{\leq 0}$ associated to the outgoing edge $e_o$. Putting these together we have
$(-1)^{\chi(\wrtr_1)}c \nu_{\wrtr_1} \wedge ds_{e_o} \wedge d\eta_\wrtr = (-1)^{\chi(\wrtr)} c \nu_\wrtr \wedge d\eta_\wrtr.$
	
In the second case we have both $\wrtr_1$ and $\wrtr_2$ having $k_1,k_2 \geq 1$, and we have $MI(\wrtr) = MI(\wrtr_1) + MI(\wrtr_2)$. Assuming $Q_\wrtr \neq \emptyset$, then one of $Q_{\wrtr_1}, Q_{\wrtr_2}$, say $Q_{\wrtr_1}$, must be a ray or a line. There are two subcases depending on whether $MI(\wrtr_2) = 0$ or $MI(\wrtr_2) = 2$.

We first assume that $MI(\wrtr_2) = 0$. In this case we have $MI(\wrtr) = 0$, and both $|\wrtr_1^{[1]}|, |\wrtr_2^{[1]}|$ are odd, and hence so is $|\wrtr_1^{[1]}||\wrtr_2^{[1]}|$. Similar to the previous case, we use the identification
$\overline{\mathfrak{M}}_{d_1}^{\wrtr_1} (\mathcal{P},\Sigma) _{\ ev_{\wrtr_1,o}}\times_{ev_{\wrtr_2,o}} \overline{\mathfrak{M}}_{d_2}^{\wrtr_2} (\mathcal{P},\Sigma) _{\ ev_{\wrtr_2,o}}\times_{\tau_{e_o}} (\real_{\leq 0}\times M_\real) = \overline{\mathfrak{M}}_{d}^{\wrtr} (\mathcal{P},\Sigma).$
From the induction hypothesis we have the relation
$(ev_{\wrtr_a,i_1})^*(d\eta_1 d\eta_2) \cdots (ev_{\wrtr_a,i_{d_1}})^*(d\eta_1 d\eta_2) = (-1)^{\chi(\wrtr_a)}c_a \nu_{\wrtr_a} \tau_{e_o}^*(d\eta_{\wrtr_a}) + \varepsilon_{\wrtr_a}$
with $c_a > 0$ and $\iota_{\nu_{\wrtr_a}^\vee}\varepsilon_{\wrtr_a} = 0$ for $a = 1, 2$. Taking their product we get
$$(ev_{\wrtr,i_1})^*(d\eta_1 d\eta_2) \cdots (ev_{\wrtr,i_d})^*(d\eta_1 d\eta_2) = (-1)^{\chi(\wrtr_1)+\chi(\wrtr_2)+|\wrtr_1^{[1]}||\wrtr_2^{[1]}|} \nu_{\wrtr_1} \wedge \nu_{\wrtr_2} \tau_{e_o}^*(d\eta_{\wrtr_1}\wedge d\eta_{\wrtr_2}) + \varepsilon,$$
where $c := c_1c_2 > 0$ and $\iota_{\nu_{\wrtr}^{\vee}} \varepsilon = 0$. Furthermore, we have
$\tau_{e_o}^*(d\eta_{\wrtr_1}\wedge d\eta_{\wrtr_2}) = (-1)^{\chi(\wrtr,v)} ds_{e_o}\wedge d\eta_{\wrtr},$
where $s_{e_o}$ is the coordinate on $\real_{\leq 0}$ associated to the outgoing edge $e_o$. Putting these together we obtain the desired identity.
	
Now assuming $MI(\wrtr_2) = 2$, we have
$(ev_{\wrtr_2,i_1})^*(d\eta_1 d\eta_2) \cdots (ev_{\wrtr_2,i_{d_1}})^*(d\eta_1 d\eta_2) = (-1)^{\chi(\wrtr_2)}c_2 \nu_{\wrtr_2}  + \varepsilon_{\wrtr_2}$
instead. In this case, $|\wrtr_1^{[1]}||\wrtr_2^{[1]}|$ is even and $\nu_{\wrtr_2}$ is an even degree differential form. Therefore we obtain
\begin{align*}
&(ev_{\wrtr,i_1})^*(d\eta_1 d\eta_2) \cdots (ev_{\wrtr,i_d})^*(d\eta_1 d\eta_2)  = (-1)^{\chi(\wrtr_1)+\chi(\wrtr_2)} c \nu_{\wrtr_1} \wedge \nu_{\wrtr_2} \tau_{e_o}^*(d\eta_{\wrtr_1}) + \varepsilon\\
 =& (-1)^{\chi(\wrtr)} c \nu_{\wrtr_1} \wedge \nu_{\wrtr_2} \wedge ds_{e_o}+ \varepsilon'
 = (-1)^{\chi(\wrtr)} c \nu_{\wrtr} + \varepsilon'
\end{align*}
using the fact that $\tau_{e_o}^*(d\eta_{\wrtr_1}) = (-1)^{\chi(\wrtr,v)}ds_{e_o} + \beta$ for some $1$-form $\beta$ on $M_\real$. Notice that switching the roles of $\wrtr_1$ and $\wrtr_2$ would yield the same result. This completes the proof of the lemma.
\end{proof}

\begin{lemma}\label{lem:iterated_integral}
For $MI(\wrtr) = 0,2$ with $Q_{\wrtr} \neq \emptyset$ and $\text{Mult}(\wrtr) \neq 0$, we have the identity
$$\alpha_{\wrtr} = (-1)^{k+d-1} (ev_o)_* \left( (ev_{i_1})^*(\delta_{P_{k_1}}) \cdots (ev_{i_d})^*(\delta_{P_{k_d}}) \right),$$
where $ev_{*} : \overline{\mathfrak{M}}_d^{\wrtr} (\mathcal{P},\Sigma) \rightarrow M_\real$'s are the evaluation maps introduced in Definition \ref{def:evaluation_map}, and the orientation on fibers of $ev_o$ is defined similarly as in Definition \ref{def:trop_tree_orientation} (notice that $(-1)^{k+d-1} = (-1)^{MI(\wrtr)/2-1}$).
\end{lemma}
\begin{proof}
We prove by using induction on the number of vertices in $\wrtr^{[0]}$. The initial case is when $\wrtr^{[0]} = \emptyset$ and the only possible trees the ones with a unique edge $e$, for which the statement is trivial.

For the induction step, suppose we have a tree $\wrtr$ with $\wrtr^{[0]} \neq \emptyset$ with the unique root vertex $v_r \in \wrtr^{[0]}$ connecting to the outgoing edge $e_o$ with two incoming edges $e_1$ and $e_2$. We split $\wrtr$ at $v_r$ to obtain two trees $\wrtr_1, \wrtr_2$ with outgoing edges $e_1, e_2$ and $k_1, k_2$ incoming edges, $d_1$ and $d_2$ marked points respectively. There are two possible cases.

This first case is when one of the incoming edges, say $e_2$, is an edge corresponding to a marked point so that $k_2 = 0$ and $d_2 =1$. As in the proof of Lemma \ref{lem:asy_support_combinatorics}, we must have $MI(\wrtr_1)=2$ and $MI(\wrtr) = 0$. In this case we let $P_{e_2} = Q_{\wrtr_2}$ be the marked point associated to $e_2$. The induction hypothesis says that
$\alpha_{\wrtr_1} = (-1)^{k+d-2} (ev_{\wrtr_1,o})_* \left( (ev_{\wrtr_1,i_1})^*(\delta_{P_{k_1}}) \cdots (ev_{\wrtr_1,i_{d_1}})^*(\delta_{P_{k_{d_1}}}) \right),$
which is a function with asymptotic support on $Q_{\wrtr_1}$. Then we have
$\alpha_{\wrtr} = -H(\alpha_{\wrtr_1} \wedge \delta_{P_{e_2}}) = -\int_{-\infty}^0 \tau_{e_o}^*(\alpha_{\wrtr_1} \wedge \delta_{P_{e_2}})$
in $\mathcal{W}^0_{*}/\mathcal{W}^{-1}_*$, where $\tau_{e_o} : \real \times M_{\real} \rightarrow M_\real$ is the flow associated to $-\bar{m}_{\wrtr}$. This equality holds because we have $P_{e_2} \in U^{m_{\wrtr}}_{1,+}$ and hence any integral over a domain not intersecting $P_{e_2}$ gives $0$ in $\mathcal{W}^0_{*}/\mathcal{W}^{-1}_*$.
		
Writing
$\overline{\mathfrak{M}}_{d_1}^{\wrtr_1} (\mathcal{P},\Sigma) _{\ ev_{\wrtr_1,o}}\times_{\tau_{e_o}} (\real_{\leq 0}\times M_\real) = \overline{\mathfrak{M}}_d^{\wrtr} (\mathcal{P},\Sigma),$
where the evaluation map $ev_{\wrtr,o} : \overline{\mathfrak{M}}_d^{\wrtr} (\mathcal{P},\Sigma) \rightarrow M_\real$ is identified as the projection to the last factor in the product on the left hand side, and the evaluation at the marked point $e_2$ is identified as $\tau_{e_o}$ on $\real_{\leq 0}\times M_\real$. Then we have
\begin{align*}
  & -\int_{-\infty}^0 \tau_{e_o}^*(\alpha_{\wrtr_1} \wedge \delta_{P_{e_2}})
=  (-1)^{k+d-1} \int_{-\infty}^0 \tau_{e_o}^*\left( (ev_{\wrtr_1,o})_*( (ev_{\wrtr_1,i_1})^*(\delta_{P_{k_1}}) \cdots (ev_{\wrtr_1,i_{d_1}})^*(\delta_{P_{k_{d_1}}})) \wedge \delta_{P_{e_2}} \right) \\
= & (-1)^{k+d-1} \int_{-\infty}^0  \left(\int_{\real_{\leq 0}^{|\wrtr_1|}}   (ev_{\wrtr,i_1})^*(\delta_{P_{k_1}}) \cdots (ev_{\wrtr,i_{d-1}})^*(\delta_{P_{k_{d-1}}})\right) \wedge ev_{i_{d}}^*(\delta_{P_{e_2}}) \\
= & (-1)^{k+d-1} (ev_{\wrtr,o})_* \left( (ev_{\wrtr,i_1})^*(\delta_{P_{k_1}}) \cdots (ev_{\wrtr,i_d})^*(\delta_{P_{e_2}}) \right).
\end{align*}

In the second case we have both $\wrtr_1$ and $\wrtr_2$ having $k_1,k_2 \geq 1$ and $MI(\wrtr) = MI(\wrtr_1) + MI(\wrtr_2)$. Making use of Lemma \ref{lem:Q_wrtr_and_homotopy} again, we notice that by comparing the domain of integration intersecting $Q_{\wrtr_1} \cap Q_{\wrtr_2}$ we have
$\alpha_{\wrtr} = - H(\alpha_{\wrtr_1} \wedge \alpha_{\wrtr_2}) = -\int_{-\infty}^0 \tau_{e_o}^* (\alpha_{\wrtr_1} \wedge \alpha_{\wrtr_2}),$
where $\tau_{e_o}$ is the flow of $-\bar{m}_{\wrtr}$.
		
Notice that we have
$\overline{\mathfrak{M}}_{d_1}^{\wrtr_1} (\mathcal{P},\Sigma) _{\ ev_{\wrtr_1,o}}\times_{ev_{\wrtr_2,o}} \overline{\mathfrak{M}}_{d_2}^{\wrtr_2} (\mathcal{P},\Sigma) _{\ ev_{\wrtr_2,o}}\times_{\tau_{e_o}} (\real_{\leq 0}\times M_\real) = \overline{\mathfrak{M}}_d^{\wrtr} (\mathcal{P},\Sigma),$
and therefore we obtain
\begin{align*}
  & -\int_{-\infty}^0 \tau_{e_o}^* (\alpha_{\wrtr_1} \wedge \alpha_{\wrtr_2})
=  (-1)^{k+d-1}\int_{-\infty}^0 \tau_{e_o}^* \big( (ev_{\wrtr_1,o})_*( (ev_{\wrtr_1,i_1})^*(\delta_{P_{k_1}}) \cdots (ev_{\wrtr_1,i_{d_1}})^*(\delta_{P_{k_{d_1}}})) \\
  &\qquad \qquad \qquad \qquad \wedge (ev_{\wrtr_2,o})_*( (ev_{\wrtr_2,i_1})^*(\delta_{P_{k_1}}) \cdots (ev_{\wrtr_1,i_{d_2}})^*(\delta_{P_{k_{d_2}}}))\big) \\
= &  (-1)^{k+d-1} \int_{-\infty}^0 (-1)^{|\wrtr_1^{[1]}||\wrtr_2^{[1]}|} \int_{\real_{\leq 0}^{|\wrtr_1|+|\wrtr_2|}} \left( (ev_{\wrtr,i_1})^*(\delta_{P_{k_1}}) \cdots (ev_{\wrtr,i_d})^*(\delta_{P_{k_d}}) \right)  \\
= & (-1)^{k+d-1} (ev_{\wrtr,o})_* \left( (ev_{\wrtr,i_1})^*(\delta_{P_{k_1}}) \cdots (ev_{\wrtr,i_d})^*(\delta_{P_{k_d}}) \right).
\end{align*}
\end{proof}

Lemma \ref{lem:iterated_integral} allows us to compute the contribution of $\alpha_{\wrtr}$ explicitly as follows:
\begin{lemma}\label{lem:explicit_contribution}
For $MI(\wrtr) = 2$ with $Q_{\wrtr} \neq \emptyset$ and $\text{Mult}(\wrtr) \neq 0$, and for any point $x$ in the interior $\text{Int}(Q_{\wrtr})$, we have
$\lim_{\hp \rightarrow 0} \alpha_{\wrtr}|_{x} = (-1)^{\chi(\wrtr)}.$
For $MI(\wrtr) = 0$ with $Q_{\wrtr} \neq \emptyset$and $\text{Mult}(\wrtr) \neq 0$, and for an arbitrary embedded path $\varrho : (a,b) \rightarrow M_\real$ intersecting the relative interior $\text{Int}_{\text{re}}(Q_\wrtr)$ transversally and positively (here positive means the orientation of $\{-\bar{m}_\wrtr, \varrho' \}$ agrees with that of $M_\real$), we have
$\lim_{\hp \rightarrow 0} \int_{\varrho} \alpha_{\wrtr} = (-1)^{\chi(\wrtr)+1}.$
\end{lemma}
\begin{proof}
We begin with $MI(\wrtr) = 2$. In this case, $k+d-1$ is even so we have the identity $\alpha_{\wrtr} =  (ev_o)_* \left( (ev_{i_1})^*(\delta_{P_{k_1}}) \cdots (ev_{i_d})^*(\delta_{P_{k_d}}) \right)$. Fixing a point $x \in \text{Int}(Q_\wrtr)$, we consider the evaluation map $\hat{ev}_x : \overline{\mathfrak{M}}_d^{\wrtr} (\mathcal{P},\Sigma) \cap ev_o^{-1}(x) \rightarrow M_\real^d$ which pulls back the volume form $\prod^d d\eta_1 \wedge d\eta_2$ to $(-1)^{\chi(\wrtr)} c\nu_\wrtr$, and in particular $\hat{ev}_x$ is a diffeomorphism onto its image (notice that $\hat{ev}_x$ is affine linear). We let $C_x := \text{Im}(\hat{ev}_x) \subset M_\real^d$. Then we have
$$(ev_o)_* \big( (ev_{i_1})^*(\delta_{P_{k_1}}) \cdots (ev_{i_d})^*(\delta_{P_{k_d}}) \big)|_{x} = (-1)^{\chi(\wrtr)}\int_{C_x} \delta_{P_{k_1}} \wedge \cdots \wedge \delta_{P_{k_d}}.$$
Using the fact that $x \in \text{Int}(Q_\wrtr)$ and the assumption that $P_1, \dots, P_n$ are in generic position (Definition \ref{def:generic_position}), we see that $(P_{k_1},\dots,P_{k_d}) \in \text{Int}(C_x)$. Together with the explicit form of $\delta_{P_i}$'s in \eqref{eqn:delta_P_i}, we have $\lim_{\hp \rightarrow 0}\int_{C_x} \delta_{P_{k_1}} \wedge \cdots \wedge \delta_{P_{k_d}} = 1$.

For $MI(\wrtr) = 0$, $k+d-1$ is odd. We consider $\mathcal{I}_\varrho : = \bigcup_{t \in (a,b)} \mathcal{I}_{\varrho(t)}$, where
we write $\mathcal{I}_x := \real_{\leq 0}^{|\wrtr^{[1]}|} \times \{x\} \cong \real^{|\wrtr^{[1]}|}_{\leq 0}$ and treat $\nu_{\wrtr}$ as a volume element on each $\mathcal{I}_x$.
Similar to the previous case we consider $\hat{ev}_{\varrho} : \mathcal{I}_\varrho \rightarrow M_{\real}^{d}$ which gives $\hat{ev}_{\varrho}^*(\prod^d d\eta_1 d\eta_2) = (-1)^{\chi(\wrtr)}  c \nu_\wrtr \wedge d\eta_{\wrtr}$. Therefore we have
\begin{align*}
\int_{\varrho} \alpha_{\wrtr} = (-1)\int_{\mathcal{I}_\varrho} (ev_{i_1})^*(\delta_{P_{k_1}}) \cdots (ev_{i_d})^*(\delta_{P_{k_d}})
= (-1)^{\chi(\wrtr)+1} \int_{C_\varrho} \delta_{P_{k_1}} \wedge \cdots \wedge \delta_{P_{k_d}}.
\end{align*}
Again using the generic assumption on the points $P_1, \dots, P_n$, we get $(P_{k_1},\dots,P_{k_d}) \in \text{Int}(C_\varrho)$ and therefore $\lim_{\hp \rightarrow 0}\int_{C_\varrho} \delta_{P_{k_1}} \wedge \cdots \wedge \delta_{P_{k_d}} =1$.
\end{proof}

For a weighted $d$-pointed $k$-tree $\wptr$ with $MI(\wptr) = 0,2$ and $Q_{\wptr}\neq \emptyset$ (notice that the definition of the polyhedral subset $Q_\wptr$ does not depend on the ribbon structure), since the monomial weights $u_{k_j}$'s at the marked points $p_{i_j}$'s are all distinct, there are exactly $2^{d-1}$ ribbon structures (up to isomorphisms) on $\wptr$. Notice that $\mathfrak{l}_\wrtr(\incoming,\dots,\incoming)$ does not depend on the ribbon structure as well because $\incoming \in \mathcal{H}^{2,2}_n$ and $\incoming$ commute with even elements in $\mathcal{H}^{*,*}_n$ (one can also see from Lemmas \ref{lem:tree_support_lemma} and \ref{lem:explicit_contribution} that the terms $(-1)^{\chi(\wrtr)}$, which depend on the ribbon structure, indeed cancel with each other).

Therefore for each weighted $d$-pointed $k$-tree $\wptr$, we can fix an arbitrary ribbon tree $\wrtr$ whose underlying tree $\underline{\wrtr}$ is $\wptr$, and write
$\mathfrak{l}_{k,d}(\incoming,\dots,\incoming) := \sum_{\wptr \in \wptree{k}{d}} \mathfrak{l}_{\wrtr}(\incoming,\dots,\incoming).$
By setting $\alpha_\wptr := (-1)^{\chi(\wrtr)} \alpha_\wrtr $ and combining Lemmas \ref{lem:tree_support_lemma} and \ref{lem:explicit_contribution}, we obtain our main theorem:
\begin{theorem}\label{thm:main_theorem}
The Maurer-Cartan solution $\Phi \in \mathcal{H}^{*,*}_n$ constructed in \eqref{eqn:solve_sum_over_trees} is of the form
$$\Phi = \incoming + \varXi^{0,0} + \varXi^{1,1},$$
with $\varXi^{i,i} \in \mathcal{H}^{i,i}_n$ for $i = 0, 1, 2$, and both correction terms $\varXi^{0,0}$ and $\varXi^{1,1}$ can be expressed as a sum over tropical disks:
\begin{align*}
\varXi^{0,0} & = \sum_{k,d} \sum_{\substack{\wptr \in \wptree{k}{d},\ MI(\wptr)=2\\ \overline{\mathfrak{M}}_d^{\wptr} (\mathcal{P},\Sigma,P_1,\dots,P_n) \neq \emptyset}} \alpha_{\wptr} \text{Mult}(\wptr) z^{m_\wptr} u_{\wptr},\\
\varXi^{1,1} & = \sum_{k,d} \sum_{\substack{\wptr \in \wptree{k}{d},\ MI(\wptr)=0\\ \overline{\mathfrak{M}}_d^{\wptr} (\mathcal{P},\Sigma,P_1,\dots,P_n) \neq \emptyset}} \alpha_{\wptr} k_{\wptr} \text{Mult}(\wptr)z^{m_\wptr} \check{\partial}_{n_{Q_\wptr}} u_{\wptr},
\end{align*}
where $\wptree{k}{d}$ is the set of isomorphism classes of weighted $d$-pointed $k$-trees introduced in Definition \ref{def:weighted_tree}.
Furthermore, in the above expressions we have
$\alpha_\wptr \in \mathcal{W}^{s_\wptr}_{Q_\wptr},$
where $Q_\wptr = ev_o(\overline{\mathfrak{M}}_d^{\wptr} (\mathcal{P},\Sigma,P_1,\dots,P_n))$ is of codimension $s_\wptr := 1-\frac{MI(\wptr)}{2}$ in $M_\real$, and
\begin{equation*}
\begin{array}{ll}
\displaystyle\lim_{\hp \rightarrow 0} \alpha_{\wptr}|_{x} =1 & \text{for any $x \in \text{Int}(Q_\wptr)$ when $MI(\wptr) = 2$},\\
\displaystyle\lim_{\hp \rightarrow 0} \int_{\varrho} \alpha_{\wptr} = - 1 & \text{for any $\varrho \pitchfork \text{Int}_{\text{re}}(Q_\wptr)$ positively when $MI(\wptr) = 0$}.\\
\end{array}
\end{equation*}
\end{theorem}

\begin{example}
\begin{figure}[h]
\centering
\includegraphics[scale=0.55]{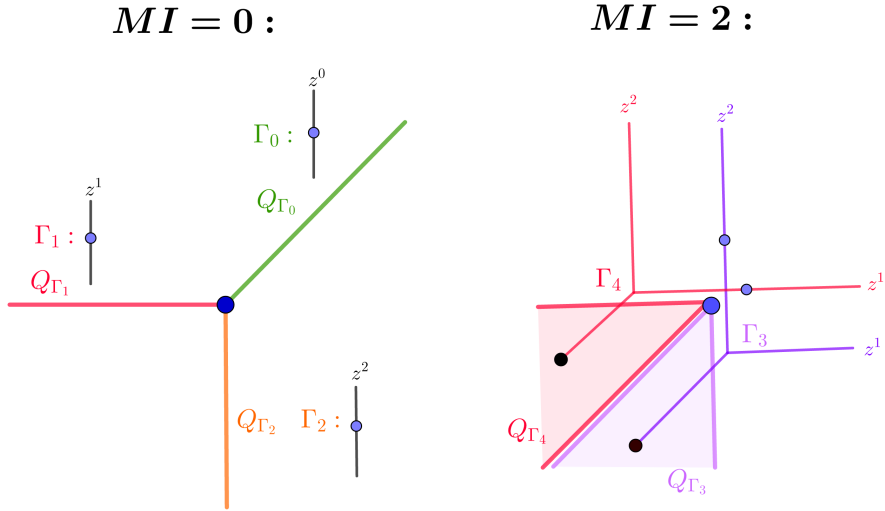}
\caption{Tropical disks and their moduli spaces for $n = 1$}
\label{fig:one_pointed_moduli}
\end{figure}	
We give an example of the locus $Q_{\wptr}$ traced out by weighted $1$-pointed $k$-trees $\wptr$ in the case $n=1$, i.e. when there is only $1$ marked point.  For a tree $\wptr$ with $MI(\wptr) = 0$, the only possibility is that $k=1$ and there are precisely 3 such trees $\wptr_0, \wptr_1,\wptr_2$ as shown in Figure \ref{fig:one_pointed_moduli} together with the corresponding $1$-dimensional loci $Q_{\wptr_0}, Q_{\wptr_1}, Q_{\wptr_2}$.
For the case $MI(\wptr) = 2$, we have $k=2$, and there are 6 such trees. Two of them, which we call $\wptr_3$ and $\wptr_4$, with the same attached monomial $\text{Mono}(\wptr) = z^1 z^2$, are shown in Figure \ref{fig:one_pointed_moduli}. Note that the boundary between $Q_{\wptr_3}$ and $Q_{\wptr_4}$ is not a wall in $\mathscr{D}$ although the moduli space jumps across it. This is because the attached monomial $\text{Mono}(\wptr)$ does not jump across the boundary, and this agrees with the fact that $\Phi$ is simply a holomorphic function outside the support $\text{Supp}(\mathscr{D})$.

\end{example}

\bibliographystyle{amsplain}
\bibliography{geometry}

\end{document}